\documentclass[1 [leqno,11pt]{amsart}
\usepackage{amssymb, amsmath,latexsym,amsfonts,amsbsy, amsthm,mathtools,graphicx,CJKutf8,CJKnumb,CJKulem,color}
\usepackage{float}
\usepackage{hyperref}

\numberwithin{equation}{section}

\allowdisplaybreaks

\let\al=\alpha

\let\d=\delta

\let\la=\lambda

\let\f=\frac

\let\om=\omega

\let\na=\nabla

\let\pa=\partial

\def\cN{{\mathcal N}}

\def\R{\mathbb R}

\def\T{\mathbb T}

\def\no{\noindent}

\def\bbD{\mathbb{D}}
\def\rmA{{\mathrm{A}}}
\def\rmB{\mathrm{B}}

\def\rmE{\mathrm{E}}
\def\Int{\mathrm{Int}}

\newcommand{\beq}{\begin{equation}}
\newcommand{\eeq}{\end{equation}}
\newcommand{\ben}{\begin{eqnarray}}
\newcommand{\een}{\end{eqnarray}}
\newcommand{\beno}{\begin{eqnarray*}}
\newcommand{\eeno}{\end{eqnarray*}}

\newtheorem{theorem}{Theorem}[section]

\newtheorem{lemma}[theorem]{Lemma}
\newtheorem{proposition}[theorem]{Proposition}

\begin{document}

\title[Pseudospectra and stability for 3D Kolmogorov flow]
{Pseudospectral bound and transition threshold for the  3D Kolmogorov flow}

\author{Te Li}
\address{School of Mathematical Science, Peking University, 100871, Beijing, P. R. China}
\email{little329@163.com}

\author{Dongyi Wei}
\address{School of Mathematical Science, Peking University, 100871, Beijing, P. R. China}
\email{jnwdyi@163.com}

\author{Zhifei Zhang}
\address{School of Mathematical Science, Peking University, 100871, Beijing, P. R. China}
\email{zfzhang@math.pku.edu.cn}

\date{\today}

\maketitle

\begin{abstract}
In this paper, we establish the pseudospectral bound for the linearized operator of the Navier-Stokes equations around the 3D Kolmogorov flow.
Using the pseudospectral bound and the wave operator method introduced in \cite{LWZ}, we prove the sharp enhanced dissipation rate for the linearized Navier-Stokes equations. As an application, we prove that if the initial velocity satisfies $\big\|U_0-\big(k_f^{-2}\sin(k_fy),0,0\big)\big\|_{H^2}\le c\nu^{\f 74}$($\nu$  the viscosity coefficient) and $k_f\in (0,1)$, then the solution does not transition away from the Kolmogorov flow.
\end{abstract}

\section{Introduction}

In this paper, we consider the 3-D incompressible Navier-Stokes equations on the torus $\Omega=\big\{(x,y,z): y\in\mathbb{T}_{2\pi/k_f},x,z\in\mathbb{T}_{2\pi}\big\}$:
\begin{align}
\left\{
\begin{aligned}
&\partial_t U-\nu\Delta U+U\cdot\nabla U+\na P=F,\\
&\nabla\cdot U=0,\\
&U(0,x,y,z)=U_0(x,y,z),
\end{aligned}
\right.\label{eq:NS}
\end{align}
where $\nu>0$ is the viscosity coefficient, $U(t,x,y,z)$ is the velocity, $P(t,x,y,z)$ is the pressure, and $F$ is the Kolmogorov force given by
\beno
F=\big(\gamma \sin(k_fy),0,0\big).
\eeno
Here $\gamma\in R$ is the amplitude, and $k_f$ is the wave number.
This force was introduced by Kolmogorov to study the transition problem at high Reynolds number($\nu\to 0$) in a setting
neglecting material boundaries. An exact solution of (\ref{eq:NS}) is given by
\begin{align}
U^{*}=\Big(\f{\gamma}{\nu k_f^2}\sin(k_fy),0,0\Big), \quad P^*=C,\label{LF}
\end{align}
which is called the Kolmogorov flow.

Beginning with Reynolds's famous paper \cite{Rey} in 1883, the stability and transition to turbulence of the laminar flows at high Reynolds number has been an active field in the fluid mechanics \cite{DR, Sch, Tre, Yag}. In this paper, we are concerned with the stability and transition of the Kolmogorov flow $U^*$.
To this end, we introduce the perturbation $V=(v_1,v_2,v_3)=U-U^{*}$, which satisfies
\begin{align}
\label{eq:NS-V}
\left\{
\begin{aligned}
&\Big(\partial_t+\f{\gamma}{\nu k_f^2}\sin(k_fy)\partial_x-\nu\Delta\Big)v_1+\f{\gamma\cos(k_fy)}{\nu k_f}v_2=-\partial_xP-V\cdot \nabla v_1,\\
&\Big(\partial_t+\f{\gamma}{\nu k_f^2}\sin(k_fy)\partial_x-\nu\Delta\Big)v_2=-\partial_yP-V\cdot \nabla v_2,\\
&\Big(\partial_t+\f{\gamma}{\nu k_f^2}\sin(k_fy)\partial_x-\nu\Delta\Big)v_3=-\partial_zP-V\cdot \nabla v_3,\\
&\nabla\cdot V=0,\\
&V(0)=V_0(x,y,z).
\end{aligned}
\right.
\end{align}

The formulation  in terms of the shear wise velocity $v_2$ and vorticity $\om_2=\pa_zv_1-\pa_xv_3$ plays an important role in this paper:
\begin{align}
\label{eq:NS-shear}\left\{
\begin{aligned}
&\Big(\partial_t+\f{\gamma}{\nu k_f^2}\sin(k_fy)\partial_x-\nu\Delta\Big){\Delta v_2}+\f{\gamma}{\nu}\sin (k_fy)\partial_xv_2=-\Delta(V\cdot \nabla v_2)-\partial_2(\Delta p),\\
&\Big(\partial_t+\f{\gamma}{\nu k_f^2}\sin(k_fy)\partial_x-\nu\Delta\Big)\omega_2+\f{\gamma}{\nu k_f}\cos(k_fy)\partial_zv_2=\nabla\cdot(-v \omega_2+\omega v_2).
\end{aligned}
\right.
\end{align}
Here $\omega=\nabla\times V=(\omega_1,\omega_2,\omega_3)$ and $p=-\Delta^{-1}\Big(\sum\limits_{i,j=1}^3\partial_iv_j\partial_jv_i\Big)$.
Recently, this formulation was used to study the regularity criterion in terms of one velocity component for the Navier-Stokes equations \cite{CZ, CZZ}.

The flow $U^*$ is linearly stable for any $\nu\ge 0$ if $k_f\le 1$  \cite{Chen, Vee}. However, it could be unstable and transition to turbulence to small perturbations at high Reynolds number \cite{Chen, She, Vee},
which is referred to as subcritical transition. Up to now, we are still lacking a good understanding of this transition. As suggested by Kelvin \cite{Kel}, the following stability threshold problem may be more accessible from a mathematical point of view: \smallskip

{\it
Given a norm $\|\cdot\|_X$, find a $\beta=\beta(X)$ so that
\beno
&&\|V_0\|_X\le \nu^\beta\Longrightarrow  {stability},\\
&&\|V_0\|_X\gg \nu^\beta\Longrightarrow  {instability}.
\eeno
}
Here the exponent $\beta$ is referred to as the transition threshold in the applied literatures. \smallskip

There are a lot of works \cite{Du, LK, LHR, RSB}  in applied mathematics and physics devoted to estimating $\beta$. Recently, Bedrossian, Germain, Masmoudi et al.  made an important progress on the stability threshold problem for the 3-D Couette flow $(y,0,0)$ in  $\mathbb{T}^2\times \R$ in a series of works \cite{BGM1, BGM2, BGM3, BMV, BWV}. Roughly speaking, their results could be summarized as follows:
\begin{itemize}

\item if the perturbation is in Gevrey class, then $\beta\le 1$ \cite{BGM1};

\item if the perturbation is in Sobolev space, then $\beta\le \f32$ \cite{BGM3}.

\end{itemize}
While in $\mathbb{T}\times \R$, the transition threshold is smaller:
\begin{itemize}

\item if the perturbation is in Gevrey class, then $\beta=0$ \cite{BMV};

\item if the perturbation is in Sobolev space, then $\beta\le \f12$\cite{BWV}.

\end{itemize}

The stability and transition for general shear flows are a challenging problem. One of main difficulties is that the linearized operator around general shear flow is non-selfadjoint and nonlocal so that the spectral analysis becomes very difficult and the spectral properties are very different from the selfadjoint operators.

 In the ideal case(i.e., $\nu=0$), the analysis for the 2-D linearized operator around shear flows can be reduced to solve the Rayleigh equation \cite{WZZ1, WZZ2, WZZ3}. In the viscous case, the problem is more difficult, since it involves solving a fourth order
ODE(Orr-Sommerfeld equation). Recently, there are some important progress on the enhanced dissipation for the 2-D linearized Navier-Stokes equations around the Lamb-Oseen vortex and Kolmogorov flow \cite{BW, DW, Ga, GGN, IMM,   LWZ, LX, WZZ3}.  To our knowledge, there are few results in 3-D except for the Couette flow. We refer to the survey paper \cite{BGM} for more recent results and open questions.\smallskip

The first goal of this paper is to study the pseudospectral bound for the linearized operator around the 3-D kolmogorov flow.  Pseudospectra has been an important concept in understanding the hydrodynamic stability due to the non-normality of the linearized operators \cite{Tre, Tre-SIAM}. We introduce the linearized operators
 \ben
&&\mathcal{L}v=\f{\gamma}{\nu k_f^2}\sin(k_fy)\partial_x v-\nu\Delta v+\f{\gamma}{\nu}\sin (k_fy)\partial_x\Delta^{-1}v,\\
&&\mathcal{H}v=\f{\gamma}{\nu k_f^2}\sin(k_fy)\partial_x v-\nu\Delta v.
\een
The linearized Navier-Stokes equations of \eqref{eq:NS-shear} take as follows
\begin{align}\label{eq:NS-L}
\left\{
\begin{array}{l}
(\partial_t+\mathcal{L})\Delta v_2=0,\\
(\partial_t+\mathcal{H})\omega_2=-\f{\gamma}{\nu k_f}\cos(k_fy)\partial_zv_2.
\end{array}\right.
\end{align}
Let $\Delta v_2=\varphi$ and $\hat{\phi}(k_1,y,k_3)$ denote the Fourier transform of $\phi(x,y,z)$ with respect to the directions $(x,z)$.
Using the pseudospectral bound and wave operator method introduced in \cite{LWZ}, we prove the enhanced decay estimates  of the solution for the linearized system \eqref{eq:NS-L}.

\begin{theorem} \label{thm:decay-L}
Given $k_f\in (0,1],$ there exist constants $c_2,\ c_3\in(0,1),$ such that if $0<\nu<c_2|\gamma|^{\frac{1}{2}},$ then for any $\ (k_1,k_3)\in\mathbb{Z}^2,\ k_1\neq 0$, the solution of \eqref{eq:NS-L} satisfies
\begin{itemize}
 \item[1.] if $k_1^2+k_3^2>k_f^2$, then for any $t>0$
\beno
&&\|\hat{\varphi}(t,k_1,\cdot,k_3)\|_{L^2}\leq Ce^{-at}\|\hat{\varphi}(0,k_1,\cdot,k_3)\|_{L^2},\\ &&\|\hat{\omega}_2(t,k_1,\cdot,k_3)\|_{L^2}\leq Ce^{-at}\|\hat{\omega}_2(0,k_1,\cdot,k_3)\|_{L^2}+C|k_1|^{-1}e^{-at}(1+at)\|\hat{\varphi}(0,k_1,\cdot,k_3)\|_{L^2};
\eeno
\item[2.] if $k_1^2+k_3^2=k_f^2=1$, then for any $t>0$
\beno
&&\|\mathbb{Q}_1\hat{\varphi}(t,k_1,\cdot,k_3)\|_{L^2}\leq Ce^{-c_3t|k_1\gamma|^{\frac{1}{2}}-\nu t}\|\mathbb{Q}_1\hat{\varphi}(0,k_1,\cdot,k_3)\|_{L^2},\\
&&\|\mathbb{P}_1\hat{\varphi}(t,k_1,\cdot,k_3)\|_{L^2}\leq e^{-\nu t}\|\hat{\varphi}(0,k_1,\cdot,k_3)\|_{L^2}+Ce^{-\nu t}|\gamma|^{\frac{1}{6}}\nu^{-\frac{1}{3}}\|\mathbb{Q}_1\hat{\varphi}(0,k_1,\cdot,k_3)\|_{L^2},\\ &&\|\hat{\omega}_2(t,k_1,\cdot,k_3)\|_{L^2}\leq Ce^{-c_3t|k_1\gamma|^{\frac{1}{2}}-\nu t}\|\hat{\omega}_2(0,k_1,\cdot,k_3)\|_{L^2}.
\eeno
\end{itemize}
Here $a=c_3|k_1\gamma|^{\frac{1}{2}}+\nu(k_1^2+k_3^2)$ is the decay rate, and $ \mathbb{Q}_1$ is the orthogonal projection from $L^2(\mathbb{T}_{2\pi})$ to $L_0^2(\mathbb{T}_{2\pi})=\big\{f\in L^2(\mathbb{T}_{2\pi});\int_0^{2\pi} f(y)dy=0\big\},$ $ \mathbb{P}_1=I-\mathbb{Q}_1.$
\end{theorem}

The second goal of this paper is to study the stability threshold of the  3-D Kolmogorov flow, whose proof relies on the enhanced decay estimates for the linearized system \eqref{eq:NS-L}.

\begin{theorem}\label{thm:NS}
Given $k_f\in (0,1),$ there exist constants $c_2,\ c',\ c_4\in(0,1),$ such that if $\gamma\in (-1,1),\ 0<\nu<c_2|\gamma|^{\frac{1}{2}},$ and  $\|V_0\|_{H^2}\leq c_4\nu^{\frac{5}{2}}|\gamma|^{-\frac{3}{4}}$, then the solution $V$ of \eqref{eq:NS-V} is global in time and satisfies
\begin{align*}
&\| v_2(t)\|_{H^2}+e^{c'\sqrt{\gamma}t}\|(\Delta v_2)_{\neq}(t)\|_{L^2}+e^{c'\sqrt{\gamma}t}\|\partial_x\omega_2(t)\|_{L^2}+\| P_0v_3(t)\|_{H^1}\leq C\|V_0\|_{H^2},\\
&\| V(t)\|_{H^2}\leq{C|\gamma|^{\frac{1}{4}}}/{\nu }\|V_0\|_{H^2}.
\end{align*}
Here $P_0f=\frac{1}{2\pi}\int f(x,y,z)dx,\ f_{\neq}=f-P_0f.$
\end{theorem}

\no{\bf Some remarks:}
\begin{itemize}
\item[1.]  If $\gamma=\nu$, $U^*$ is independent of $\nu$, thus the transition threshold $\beta\le \f 74$. In this case, the solution is included in a basin of attraction of the Kolmogorov flow with the size of $\nu$. The transition threshold should be not optimal. To improve it, more subtle analysis is needed. In 2-D, the last two authors and Zhao proved that the transition threshold $\beta$ is less than $\f 23+$ \cite{WZZ3}.

\item[2.] For simplifying the presentation, we consider the case of $k_f<1$.  Our result could be extended to the case of $k_f=1$. For $k_f>1$, the flow $U^*$ is linearly unstable.

\item[3.] The enhanced dissipation rate of the linearized operator around   the 2-D Kolmogorov flow $(\sin y,0)$ was independently proved by Wei et al.  \cite{WZZ3} and Ibrahim et al.\cite{IMM}. In 3-D, the problem is more difficult, since we need to handle a coupled linear system \eqref{eq:NS-L}.
Back to 2D, the bound of  $\|\mathbb{P}_1\hat{\varphi}(t,k_1,\cdot,k_3)\|_{L^2}$ is better than the corresponding one(see 3.203) in \cite{IMM}, where
there is loss of $\nu^{-\f13}$. Here the loss is $\nu^{-\f16}$ by taking $\gamma=\nu$.

\end{itemize}

\section{Sketch of the proof and ideas}

\subsection{Pseudospectra}
To study the linearized Navier-Stokes equations \eqref{eq:NS-L},
we first consider the resolvent estimates of the following linear operator in section 3:
\begin{align*}
L_\lambda w=i\f{\alpha}{\nu}\big[(\sin y-\lambda)w+\sin y\varphi\big]-\nu\partial_y^2w,
\end{align*}
where $(\partial_y^2-\beta^2)\varphi=w$, $\al, \la\in \R, \nu>0$.

In the case without nonlocal term $\sin y\varphi$, the problem is relatively
simple. For $\la>1$, we consider the estimate
\begin{align*}
\big|\big\langle i\f{\alpha}{\nu}(\sin y-\lambda)w-\nu\partial_y^2w,\chi_{[\hat{y}_{-},\hat{y}_{+}]}w\big\rangle\big|
\end{align*}
for some  $\hat{y}_{-}\in(\f{\pi}{2},\f{\pi}{2}+\delta)$ and $\hat{y}_{+}\in(\f{5\pi}{2}-\delta,\f{5\pi}{2})$ with $\delta=|\alpha|^{-\f14}\nu^{\f12}\ll1$.
For $\la\in [0,1]$, we take $0\leq y_1\leq\f{\pi}{2}\leq y_2\leq\pi$ such that $\lambda=\sin y_1=\sin y_2$, and choose some $\hat{y}_{-}\in(y_2,y_2+\delta)$ and $\hat{y}_{+}\in(y_1+2\pi-\delta,y_1+2\pi)$.
Then we consider the estimates
\begin{align*}
\big|\big\langle i\f{\alpha}{\nu}(\sin y-\lambda)w-\nu\partial_y^2w,\chi_{[y_1,y_2]}w\big\rangle\big|,\quad
\big|\big\langle i\f{\alpha}{\nu}(\sin y-\lambda)w-\nu\partial_y^2w,\chi_{[\hat{y}_{-},\hat{y}_{+}]}w\big\rangle\big|.
\end{align*}
The key point is that in the interval $(y_1+\delta, y_2-\delta)$ and $(y_2+\delta, y_1+2\pi-\delta)$, we have
\beno
\la-\sin y\gtrsim \delta^2.
\eeno
We control $L^2$ estimate of $w$ in the other intervals of $\delta$ size by using $L^\infty$ estimate of $w$, while
$\|w\|_{L^\infty}^2\leq\|w\|_{L^2}\|w'\|_{L^2}+\|w\|_{L^2}^2$.

For the full operator $L_\la$, the case of $\la>1$ is similar. For $\la\in [0,1]$, we can first  reduce it to consider the operator
$\tilde{L}_\la u=i\f{\alpha}{\nu}[(\sin y-\lambda)u+\lambda\varphi]-\nu\partial_y^2u$. Let $0\leq y_1\leq\f{\pi}{2}\leq y_2\leq\pi$ so that $\lambda=\sin y_1=\sin y_2$, and $(\varphi, u)$ satisfy
\begin{align*}
\left\{
\begin{aligned}
&(\partial_y^2-\tilde{\beta}^2)\varphi=u,\quad y_1\leq y\leq y_2,\\
&\varphi(y_1)=\varphi(y_2)=0.
\end{aligned}
\right.
\end{align*}
To handle the nonlocal term, one of the key points is the following coercive estimate:
 \begin{align*}
&\lambda\Big(\int_{y_1}^{y_2}(\sin y-\lambda)^2\Big|(\f{\varphi}{\sin y-\lambda})'\Big|^2dy+\int_{y_1}^{y_2}\tilde{\beta}^2|\varphi|^2dy\Big)\nonumber\\&\qquad\leq \int_{y_1}^{y_2}(\sin y-\lambda)|u|^2dy+\big\langle\lambda\varphi,u\chi_{(y_1,y_2)}\big\rangle.
\end{align*}
This is similar to the situation of the Lamb-Oseen operator when $|k|\ge 2$. Only when $|k|=1$, we need to use the wave operator.

After taking Fourier transformation in $x,z$,  the linearized operators $\mathcal{L}$ and $\mathcal{H}$ are reduced to
\beno
&&\mathcal{L}'_{k_1,k_3}=-\nu k_f^2\partial_y^2+\frac{ik_1}{k_f^2}\f{\gamma}{\nu }\sin y\big(1-(\beta^2-\partial_y^2)^{-1}\big),\\
&&\mathcal{H}'_{k_1,k_3}=-\nu k_f^2\partial_y^2+\frac{ik_1}{k_f^2}\f{\gamma}{\nu }\sin y,
\eeno
where $ \beta^2=\f{k_1^2+k_3^2}{k_f^2}$. These two operators are accretive. Define
\beno
\Psi(H)=\inf\big\{\|(H-i\la)f\|;f\in {D}(H),\ \la\in \mathbb{R},\  \|f\|=1\big\}.
\eeno
Using the resolvent estimates of $L_\la$, we can prove that
\ben\label{eq:spectra}
\Psi(\mathcal{H}'_{k_1,k_3})\geq c|k_1\gamma|^{\frac{1}{2}},\quad
\Psi(\mathcal{L}'_{k_1,k_3})\geq c|k_1\gamma|^{\frac{1}{2}}(1-\beta^{-2}).\een

\subsection{Enhanced dissipation}
For the accretive operator $H$, we can derive the following sharp semigroup bound from the pseudospectra $\Psi(H)$:
\beno
\|e^{-tH}\|\leq e^{-t\Psi(H)+{\pi}/{2}}.
\eeno
Thus, using the pseudospectra bounds \eqref{eq:spectra}, we can show that
\begin{align*}
&\|e^{-t\mathcal{H}}g_{\neq}\|_{L^2}\leq Ce^{-c|\gamma|^{\f12}t-\nu t}\|g_{\neq}\|_{L^2},\\
&\|e^{-t\mathcal{L}}g_{\neq}\|_{L^2}\leq Ce^{-c|\gamma|^{\f12}t-\nu t}\|g_{\neq}\|_{L^2}.
\end{align*}
However, these two estimates are not enough to handle the linearized
coupled system \eqref{eq:NS-L}.  Let $f(t,y)=\hat\varphi(t,k_1,y,k_3)$ and $g(t,y)=\hat{\omega}_2(t,k_1,y,k_3)$. Then \eqref{eq:NS-L} becomes
\begin{align*}
\left\{
\begin{array}{l}
\partial_tf+(\nu(k_1^2+k_3^2)+\mathcal{L}'_{k_1,k_3})f=0,\\
\partial_tg+(\nu(k_1^2+k_3^2)+\mathcal{H}'_{k_1,k_3})g=\frac{ik_3}{k_f^3}\f{\gamma}{\nu}\cos y (\alpha^2-\partial_y^2)^{-1}f,
\end{array}\right.
\end{align*}
where $ \alpha^2=\f{k_1^2+k_3^2}{k_f^2}$. The main trouble comes from
the term $\frac{ik_3}{k_f^3}\f{\gamma}{\nu}\cos y (\alpha^2-\partial_y^2)^{-1}f$, which can not be viewed as a source term of $g$.
Otherwise, this will lead to a bad decay estimate. Our key idea is to use the wave operator method introduced in \cite{LWZ}. Following the construction in \cite{WZZ3}, we can find an operator $\mathbb{D}_2$ so that
  \begin{align*}
&\mathbb{D}_2\big(\sin y(1+(\partial_y^2-\alpha^2)^{-1})\omega\big)=\sin y\mathbb{D}_2(\omega)-(\partial_y^2-\alpha^2)^{-1}\omega.
\end{align*}
Then we introduce a good unknown $g_1=g+\f{k_3}{k_1k_f}\cos y\mathbb{D}_2(f)$, which satisfies
\begin{align*}
\partial_tg_1+(\nu(k_1^2+k_3^2)+\mathcal{H}'_{k_1,k_3})g_1= \f{\nu k_fk_3}{k_1}[\cos y\mathbb{D}_2,\partial_y^2]f .
\end{align*}
Based on this formulation, we can show that for $\al>1$
\begin{align*}
\|g(t)\|_{L^2} \leq Ce^{-at}\|g_0\|_{L^2}+\f{Ce^{-at}(1+at)\|f_0\|_{L^2}}{|k_1|},
\end{align*}
where $a=\nu(k_1^2+k_3^2)+c|k_1\gamma|^{\frac{1}{2}}$. The price we pay is the extra growth factor $(1+at)$.

In the case of $\al=1$, we also need some new ideas to control the growth of $\|\mathbb{P}_1f(t)\|_{L^2}$. In this case, $f$ satisfies
\begin{align*}
\partial_t f&=-\nu f+\nu f''-i\f{\beta}{\nu}\sin y u=-\nu f+\mathcal{L}_1u,
\end{align*}
where $\mathcal{L}_1u=\nu u''-\nu u-i\f{\beta}{\nu}\sin y u$ with $f''=u''-u$ and $\beta=\gamma k_1$.  In particular, we need the following coercive estimate
\begin{align*}
 |\langle\mathcal{L}^{-1}_11,1\rangle|\geq c\nu/|\beta|,
\end{align*}
and the key observation:
\begin{align*}
\langle\mathcal{L}_1^{-1}f,1\rangle(t)=e^{-\nu t} \langle\mathcal{L}_1^{-1}f_0,1\rangle.
\end{align*}
Thus, we can control $|\mathbb{P}_{1}f(t)|$ in the following way
\begin{align*}
|\mathbb{P}_{1}f(t)|&\leq C(|\beta|/\nu) |\langle\mathcal{L}^{-1}_1\mathbb{P}_{1}f(t),1\rangle|\leq C(|\beta|/\nu) (|\langle\mathcal{L}^{-1}_1\mathbb{Q}_{1}f(t),1\rangle|+|\langle\mathcal{L}^{-1}_1f(t),1\rangle|)\\
&\leq C(|\beta|/\nu)\big(e^{-\nu t}|\langle\mathcal{L}^{-1}_1f_0,1\rangle|+\|\mathcal{L}^{-1}_1\mathbb{Q}_{1}f(t)\|_{L^1}\big).
\end{align*}

\subsection{Nonlinear stability threshold}
To study the nonlinear stability, the first step is to establish the enhanced decay estimates for the inhomogeneous system:
\begin{align*}
&(\partial_t+\mathcal{L})\Delta v_2=\text{div}f,\\
&(\partial_t+\mathcal{H})\omega_2=-\f{\gamma}{\nu k_f}\cos(k_fy)\partial_zv_2+\text{div}g.
\end{align*}
More precisely, we will establish the following decay estimates:
\begin{align*}
\|(\Delta v_2)_{\neq}\|_{X_{c'}}^2\lesssim&  \|(\Delta v_2)_{\neq}(0)\|_{L^2}^2+\nu^{-1}\|e^{c'\sqrt{|\gamma|}t}f_{\neq} \|_{L^2L^2}^2,\\
\|\partial_x\omega_2\|_{X_{c'}}^2\lesssim&\|(\Delta v_2)_{\neq}(0)\|_{L^2}^2+\|\partial_x\omega_2(0)\|_{L^2}^2\\&+ \nu^{-1}\Big(\|e^{c'\sqrt{|\gamma|}t}f_{\neq}\|_{L^{2}L^2}^2+\|e^{c'\sqrt{|\gamma|}t}\partial_xg\|_{L^{2}L^2}^2\Big),
\end{align*}
where
\begin{align*}
\|u\|_{X_{c'}}^2=\|e^{c'\sqrt{|\gamma|}t}u\|_{L^{\infty}L^2}^2+{|\gamma|}^{\frac{1}{2}}\|e^{c'\sqrt{|\gamma|}t}u\|_{L^{2}L^2}^2+\nu\|e^{c'\sqrt{|\gamma|}t}\nabla u\|_{L^{2}L^2}^2.
\end{align*}
The semigroup method does not work, since we only have the decay estimates of the semigroup in $L^2$. To overcome this difficulty, the key idea is to use the following energy estimate for the linearized equation:
\begin{align*}
\|(\Delta v_2)_{\neq}(t)\|_{L^2}^2+\nu\int_0^t\|(\nabla\Delta v_2)_{\neq}(s)\|_{L^2}^2ds\leq C\|(\Delta v_2)_{\neq}(0)\|_{L^2}^2+C \nu^{-1}\int_0^t\|f_{\neq}(s) \|_{L^2}^2ds.
\end{align*}
This property seems special for the Kolmogorov flow. Again, we need to use the wave operator for the decay estimates of $\om_2$.

Next we introduce the following two quantities:
\begin{align*}
&M_0(T)=\sup_{0\leq t\leq T}\big(\| v_2(t)\|_{H^2}+e^{c'\sqrt{\gamma}t}\|(\Delta v_2)_{\neq}(t)\|_{L^2}+e^{c'\sqrt{\gamma}t}\|\partial_x\omega_2(t)\|_{L^2}+\| P_0v_3(t)\|_{H^1}\big),\\
&M_1(T)=\sup_{0\leq t\leq T}\| V(t)\|_{H^2}.
\end{align*}
Our goal is to show the following uniform estimates:
\begin{align*}
M_0(T)\leq  C_1\|V_0\|_{H^2},\quad M_1(T)\leq  C_1(|\gamma|/\nu^2)\|V_0\|_{H^2}.
\end{align*}
The main difficulty is to control the part of the solution without the enhanced decay. For the nonlinear estimates, we have to study nonlinear interaction between two parts very carefully and use the good structures.  The most subtle estimate is $\|P_0v_1\|_{H^2}$, where $P_0v_1$
satisfies
\begin{align*}
&(\partial_t-\nu\Delta)P_0v_1+\f{\gamma\cos(k_fy)}{\nu k_f}P_0v_2\\
&=-P_0v_2\partial_yP_0v_1-P_0v_3\partial_zP_0v_1-P_0(V_{\neq}\cdot\nabla (v_1)_{\neq}).
\end{align*}
In this equation, the trouble terms are  $\f{\gamma\cos(k_fy)}{\nu k_f}P_0v_2$ and $P_0v_2\partial_yP_0v_1$, which give rise to the so called lift-up effect in 3-D. To overcome this difficulty, our idea is to introduce the steady solution $v_1^{(1)}$, i.e.,
\begin{align*}
&-\nu\Delta v_1^{(1)}+a_2\partial_yv_1^{(1)}+\f{\gamma\cos(k_fy)}{\nu k_f}a_2=0,
\end{align*}
where $a_2$ is the average of $v_2$ on the domain.  That is, $v_1^{(1)}=-\dfrac{\nu k_f \cos(k_fy)+a_2\sin(k_fy)}{(\nu k_f)^2+a_2^2}\dfrac{\gamma a_2}{\nu k_f^2}$, thus,
\begin{align*}
\| v_1^{(1)}\|_{H^2}\leq \dfrac{C|\gamma a_2|}{\nu^2}\leq \dfrac{C|\gamma|}{\nu^2}\| V_0\|_{H^2}.
\end{align*}
Let $v_1^{(2)}=P_0v_1-v_1^{(1)}$.  Then $v_1^{(2)}$ satisfies a heat equation with good source terms:
\begin{align*}
&(\partial_t-\nu\Delta)v_1^{(2)}+\f{\gamma\cos(k_fy)}{\nu k_f}(P_0v_2-a_2)\\
&=-P_0v_2\partial_yv_1^{(2)}-(P_0v_2-a_2)\partial_yv_1^{(1)}-P_0v_3\partial_zv_1^{(2)}-P_0(V_{\neq}\cdot\nabla (v_1)_{\neq}).\end{align*}

\section{Resolvent estimate of the linearized operator}

In this section, we study the resolvent estimate of the linearized operator
in $\T$:
\begin{align}\label{def:L}
L_\lambda w=i\f{\alpha}{\nu}\big[(\sin y-\lambda)w+\sin y\varphi\big]-\nu\partial_y^2w,
\end{align}
where $(\partial_y^2-\beta^2)\varphi=w$, $\al, \la\in \R, \nu>0$.  In \cite{IMM},  Ibrahim et al. studied the resolvent estimate of a similar operator by developing an abstract framework. We will follow the approach developed in \cite{LWZ}, which seems more elementary and simple. 

\subsection{Resolvent estimate without nonlocal term}
In this subsection, we consider the following linear operator $\cN_\la$ without nonlocal term
\ben\label{def:N}
\mathcal{N}_{\lambda}w= i\f{\alpha}{\nu}(\sin y-\lambda)w-\nu\partial_y^2w.
\een

\begin{proposition}\label{prop:Res-toy}
Given $0<\nu\leq1$, $|\alpha|\gg \nu^2$ and $\lambda\in\mathbb{R}$, there exists a constant $C>0$  independent of $\nu,\alpha,\lambda$, such that for any $w\in H^2(\T)$
\begin{align}
\|\cN_\lambda w\|_{L^2}\geq C|\alpha|^{\f12}\|w\|_{L^2}.
\end{align}
\end{proposition}

\begin{proof}
Let us first consider the case of  $|\lambda|>1$. Without loss of generality, we assume $\lambda>1$.
Let $\delta=|\alpha|^{-\f14}\nu^{\f12}\ll1$, then
$\f{\pi}{2}+\delta<\f{5\pi}{2}-\delta$.
We choose $\hat{y}_{-}\in(\f{\pi}{2},\f{\pi}{2}+\delta)$ and $\hat{y}_{+}\in(\f{5\pi}{2}-\delta,\f{5\pi}{2})$ so that
\begin{align}
|w'(\hat{y}_{-})|^2+|w'(\hat{y}_{+})|^2\leq\f{\|w'\|^2}{\delta}\label{I03}.
\end{align}

By integration by parts, we get
\begin{align*}
&\big|\big\langle i\f{\alpha}{\nu}(\sin y-\lambda)w-\nu\partial_y^2w,\chi_{[\hat{y}_{-},\hat{y}_{+}]}w\big\rangle\big|\\
&\geq\Big|i\f{\alpha}{\nu}\int_{\hat{y}_{-}}^{\hat{y}_{+}}(\sin y-\lambda)|w|^2dy+\nu\int_{\hat{y}_{-}}^{\hat{y}_{+}}|\partial_yw|^2dy\Big|-\nu(|w'\bar{w}(\hat{y}_{+})|+|w'\bar{w}(\hat{y}_{-})|)\\
&\geq \f{\sqrt{2}}{2}\Big(\f{|\alpha|}{\nu}\int_{\hat{y}_{-}}^{\hat{y}_{+}}(\lambda-\sin y)|w|^2dy+\nu\int_{\hat{y}_{-}}^{\hat{y}_{+}}|\partial_yw|^2dy\Big)-\nu\big(|w'\bar{w}(\hat{y}_{+})|+|w'\bar{w}(\hat{y}_{-})|\big),
\end{align*}
and by (\ref{I03}), we have
\begin{align*}
\nu\big(|w'\bar{w}(\hat{y}_{+})|+|w'\bar{w}(\hat{y}_{-})|\big)\leq \f{2\nu}{\sqrt{\delta}}\|w'\|_{L^2}\|w\|_{L^{\infty}},
\end{align*}
which yield that
\begin{align*}
&\| i\f{\alpha}{\nu}(\sin y-\lambda)w-\nu\partial_y^2w\| \|w\|+\f{2\nu}{\sqrt{\delta}}\|w'\|_{L^2}\|w\|_{L^{\infty}}\\
&\gtrsim \big|\big\langle i\f{\alpha}{\nu}(\sin y-\lambda)w-\nu\partial_y^2w,\chi_{[\hat{y}_{-},\hat{y}_{+}]}w\big\rangle \big|+\f{2\nu}{\sqrt{\delta}}\|w'\|_{L^2}\|w\|_{L^{\infty}}\\
&\gtrsim \f{|\alpha|}{\nu}\int_{\hat{y}_{-}}^{\hat{y}_{+}}(\lambda-\sin y)|w|^2dy\geq \f{|\alpha|}{\nu}\int_{\f{\pi}{2}+\delta}^{\f{5\pi}{2}-\delta}(1-\sin y)|w|^2dy.
\end{align*}
For any $y\in(\f{\pi}{2}+\delta,\f{5\pi}{2}-\delta)$, we have
\begin{align*}
1-\sin y&\geq \sin \f{\pi}{2}-\sin (\f{\pi}{2}+\delta)=\sin \f{\pi}{2}-\sin (\f{\pi}{2}-\delta)=2\sin \f{\delta}{2} \cos(\f{\pi}{2}-\f{\delta}{2})=2\sin^2\f{\delta}{2}\gtrsim \delta^2.
\end{align*}
Thus, we have
\begin{align*}
\|w\|_{L^2([\f{\pi}{2}+\delta,\f{5\pi}{2}-\delta])}^2\lesssim\f{\nu}{\delta^2|\alpha|}\big(\| i\f{\alpha}{\nu}(\sin y-\lambda)w-\nu\partial_y^2w\| \|w\|+\f{2\nu}{\sqrt{\delta}}\|w'\|_{L^2}\|w\|_{L^{\infty}}\big).
\end{align*}
On the other hand, we have
\begin{align}
&\|w\|_{L^\infty}^2\leq\|w\|_{L^2}\|w'\|_{L^2}+\|w\|_{L^2}^2,\label{eq:w-inf} \\
&\big|\big\langle i\f{\alpha}{\nu}(\sin y-\lambda)w-\nu\partial_y^2w,w\big\rangle \big|\geq\nu\|w'\|_{L^2}^2.\label{eq:Res-tri}
\end{align}
Thus, we deduce that
\begin{align*}
\|w\|_{L^2}^2&=\|w\|_{L^2([\f{\pi}{2}+\delta,\f{5\pi}{2}-\delta])}^2+\|w\|_{L^2([\f{\pi}{2},\f{\pi}{2}+\delta]\cup[\f{5\pi}{2}-\delta,\f{5\pi}{2}])}^2\\
&\lesssim \f{\nu}{\delta^2|\alpha|}\big(\|\mathcal{N}_{\lambda}w\|_{L^2} \|w\|_{L^2}+\f{2\nu}{\sqrt{\delta}}\|w'\|_{L^2}\|w\|_{L^{\infty}}\big)+\delta\|w\|_{L^\infty}^2\\
&\lesssim \f{\nu}{\delta^2|\alpha|}\big(\|\mathcal{N}_{\lambda}w\|_{L^2} \|w\|_{L^2}+\f{2\nu}{\sqrt{\delta}}\|w'\|_{L^2}^{\f32}\|w\|_{L^2}^{\f12}\big)+\delta \|w'\|_{L^2}\|w\|_{L^2}+\d\|w\|_{L^2}^2.
\end{align*}
As $\delta\ll1$, this shows that
\begin{align*}
\|w\|_{L^2}^2\lesssim \f{\nu}{\delta^2|\alpha|}(\|\mathcal{N}_{\lambda}w\|_{L^2} \|w\|_{L^2}+\f{2\nu^{\f14}}{\sqrt{\delta}}\|\mathcal{N}_{\lambda}w\|_{L^2}^{\f34}\|w\|_{L^2}^{\f54})+\f{\delta}{\sqrt{\nu}} \|\mathcal{N}_{\lambda}w\|_{L^2}^{\f12}\|w\|_{L^2}^{\f32},
\end{align*}
which implies that
\begin{align*}
\|\cN_\la w\|_{L^2}&\gtrsim\min\big\{\f{\delta^2|\alpha|}{\nu},\ \f{\delta^{\f{10}{3}}|\alpha|^{\f43}}{\nu^{\f53}},\ \f{\nu}{\delta^2}\big\}\|w\|_{L^2}\gtrsim |\alpha|^{\f12}\|w\|_{L^2}.\label{K4}
\end{align*}

Due to $\tau_{\pi}\mathcal{N}_{\lambda}=\mathcal{N}_{-\lambda}\tau_{\pi}$, where $\tau_{\pi}w(y)=w(y+\pi)$. Hence,  for $\lambda<-1$, we also have
\begin{align}
\|\mathcal{N}_{\lambda}w\|_{L^2}\gtrsim|\alpha|^{\f12}\|w\|_{L^2}.
\end{align}

The case of $|\la|\le 1$ follows from the following lemma.
\end{proof}

\begin{lemma}
Under the same assumptions as Proposition \ref{prop:Res-toy} and $\la \in [-1,1]$, it
holds that for any $w\in H^2(\T)$
\begin{align*}
\|\cN_\la w \|_{L^2}\geq C|\alpha|^{\f12}\|w\|_{L^2}.
\end{align*}
\end{lemma}

\begin{proof}Without loss of generality, let us first assume that $\lambda\in[0,1]$. Then there exist $0\leq y_1\leq\f{\pi}{2}\leq y_2\leq\pi$ such that $\lambda=\sin y_1=\sin y_2$. For $ 0<\delta\leq\f\pi2$,  we can choose $y_{-}\in(y_1,y_1+\delta)$ and $y_{+}\in(y_2-\delta,y_2)$ so that
\begin{align}
|w'(y_{-})|^2+|w'(y_{+})|^2\leq\f{\|w'\|^2}{\delta}\label{I0},
\end{align}
where $\delta$ will be determined later.

First of all, we consider the case of $0<\delta<\f{y_2-y_1}{2}$.
As above, by integration by parts, we have
\begin{align*}
&\big|\big\langle i\f{\alpha}{\nu}(\sin y-\lambda)w-\nu\partial_y^2w,\chi_{[y_{-},y_{+}]}w\big\rangle \big|\\
&\geq \f{\sqrt{2}}{2}\Big(\f{|\alpha|}{\nu}\int_{y_{-}}^{y_{+}}(\sin y-\lambda)|w|^2dy+\nu\int_{y_{-}}^{y_{+}}|\partial_yw|^2dy\Big)-\nu\big(|w'\bar{w}(y_{+})|+|w'\bar{w}(y_{-})|\big),
\end{align*}
and by (\ref{I0}),
\begin{align*}
\nu\big(|w'\bar{w}(y_{+})|+|w'\bar{w}(y_{-})|\big)\leq \f{2\nu}{\sqrt{\delta}}\|w'\|_{L^2}\|w\|_{L^{\infty}},
\end{align*}
which yield that
\begin{align*}
&\| i\f{\alpha}{\nu}(\sin y-\lambda)w-\nu\partial_y^2w\| \|w\|+\f{2\nu}{\sqrt{\delta}}\|w'\|_{L^2}\|w\|_{L^{\infty}}\\
&\gtrsim \f{|\alpha|}{\nu}\int_{y_{-}}^{y_{+}}(\sin y-\lambda)|w|^2dy\geq \f{|\alpha|}{\nu}\int_{y_1+\delta}^{y_2-\delta}(\sin y-\lambda)|w|^2dy.
\end{align*}
For $ y\in(y_1+\delta,y_2-\delta)$, we have
\begin{align*}
\sin y-\sin y_1&=\cos y_1(y-y_1)+\f{\sin(\theta y+(1-\theta)y_1)}{2}(y-y_1)^2\geq \cos y_1\delta+\sin y_1\delta^2\gtrsim \delta^2.
\end{align*}
Therefore,
\begin{align*}
\|w\|_{L^2([y_1+\delta,y_2-\delta])}^2\lesssim\f{\nu}{\delta^2|\alpha|}\big(\| i\f{\alpha}{\nu}(\sin y-\lambda)w-\nu\partial_y^2w\| \|w\|+\f{2\nu}{\sqrt{\delta}}\|w'\|_{L^2}\|w\|_{L^{\infty}}\big),
\end{align*}
which along with \eqref{eq:w-inf} and \eqref{eq:Res-tri} gives
\begin{align}
\nonumber\|w\|_{L^2([y_1,y_2])}^2&=\|w\|_{L^2([y_1+\delta,y_2-\delta])}^2+\|w\|_{L^2([y_1,y_1+\delta]\cup[y_2-\delta,y_2])}^2\\
\nonumber&\lesssim \f{\nu}{\delta^2|\alpha|}(\|\mathcal{N}_{\lambda}w\|_{L^2} \|w\|_{L^2}+\f{2\nu}{\sqrt{\delta}}\|w'\|_{L^2}\|w\|_{L^{\infty}})+\delta\|w\|_{L^\infty}^2\\
&\lesssim \f{\nu}{\delta^2|\alpha|}(\|\mathcal{N}_{\lambda}w\|_{L^2} \|w\|_{L^2}+\f{2\nu}{\sqrt{\delta}}\|w'\|_{L^2}^{\f32}\|w\|_{L^2}^{\f12})+\delta \|w'\|_{L^2}\|w\|_{L^2}+\delta\|w\|_{L^2}^2.\label{NI1}
\end{align}

Now we estimate $w$ in $[y_2,y_1+2\pi]$. Similarly, let us choose $\tilde{y}_ {-}\in(y_2,y_2+\delta)$ and $\tilde{y}_{+}\in(y_1+2\pi-\delta,y_1+2\pi)$ so that
\begin{align}
|w'(\tilde{y}_{-})|^2+|w'(\tilde{y}_{+})|^2\leq\f{\|w'\|^2}{\delta}\label{I01}.
\end{align}
In a similar way as above, we can deduce that
\begin{align*}
&\| i\f{\alpha}{\nu}(\sin y-\lambda)w-\nu\partial_y^2w\| \|w\|+\f{2\nu}{\sqrt{\delta}}\|w'\|_{L^2}\|w\|_{L^{\infty}}\\
&\gtrsim \f{|\alpha|}{\nu}\int_{\tilde{y}_{-}}^{\tilde{y}_{+}}(\lambda-\sin y)|w|^2dy\geq \f{|\alpha|}{\nu}\int_{y_2+\delta}^{y_1+2\pi-\delta}(\lambda-\sin y)|w|^2dy.
\end{align*}
For $ y\in(y_2+\delta,y_1+2\pi-\delta)$, we have
\begin{align*}
\sin y_2-\sin y\geq& \sin y_2-\sin (y_2+\delta)=\sin y_1-\sin (y_1-\delta)=2\sin \f{\delta}{2} \cos(y_1-\f{\delta}{2})\\
=&2\sin\f{\delta}{2}(\cos y_1\cos\f{\delta}{2}+\sin y_1\sin\f{\delta}{2})=\cos y_1\sin\delta+2\sin y_1(\sin\f{\delta}{2})^2\\
\gtrsim& \cos y_1\delta+\sin y_1\delta^2\gtrsim \delta^2.
\end{align*}
Therefore, we can conclude that
\begin{align}
\nonumber\|w\|_{L^2([y_2,y_1+2\pi])}^2&=\|w\|_{L^2([y_2+\delta,y_1+2\pi-\delta])}^2+\|w\|_{L^2([y_2,y_2+\delta]\cup[y_1+2\pi-\delta,y_1+2\pi])}^2\\
&\lesssim \f{\nu}{\delta^2|\alpha|}(\|\mathcal{N}_{\lambda}w\|_{L^2} \|w\|_{L^2}+\f{2\nu}{\sqrt{\delta}}\|w'\|_{L^2}^{\f32}\|w\|_{L^2}^{\f12})+\delta \|w'\|_{L^2}\|w\|_{L^2}+\delta\|w\|_{L^2}^2.\label{NI2}
\end{align}
Taking $\delta=|\alpha|^{-\f14}\nu^{\f12}$, it follows from (\ref{NI1}) and
\eqref{NI2} that for $\delta<\f{y_2-y_1}{2}$
\begin{align}
\|\cN_\la w\|_{L^2}\gtrsim |\alpha|^{\f12}\|w\|_{L^2}.
\end{align}

Next we consider the case of $\delta=|\alpha|^{-\f14}\nu^{\f12}\geq\f{y_2-y_1}{2}$. Let us choose $\hat{y}_{-}\in(y_2,y_2+\delta)$ and $\hat{y}_{+}\in(y_1+2\pi-\delta,y_1+2\pi)$ so that
\begin{align}
|w'(\hat{y}_{-})|^2+|w'(\hat{y}_{+})|^2\leq\f{\|w'\|^2}{\delta}\label{I02}.
\end{align}
Using integration by parts as above, we deduce that
\begin{align*}
&\| i\f{\alpha}{\nu}(\sin y-\lambda)w-\nu\partial_y^2w\| \|w\|+\f{2\nu}{\sqrt{\delta}}\|w'\|_{L^2}\|w\|_{L^{\infty}}\\
&\gtrsim \f{|\alpha|}{\nu}\int_{\hat{y}_{-}}^{\hat{y}_{+}}(\lambda-\sin y)|w|^2dy\geq \f{|\alpha|}{\nu}\int_{y_2+\delta}^{y_1+2\pi-\delta}(\lambda-\sin y)|w|^2dy.
\end{align*}
For $ y\in(y_2+\delta,y_1+2\pi-\delta)$, we have
\begin{align*}
&\sin y_2-\sin y\gtrsim \delta^2.
\end{align*}
Then we infer that
\begin{align*}
\|w\|_{L^2}^2&=\|w\|_{L^2([y_2,y_1+2\pi])}^2+\|w\|_{L^2([y_1,y_2])}^2\\
&=\|w\|_{L^2([y_2+\delta,y_1+2\pi-\delta])}^2+\|w\|_{L^2([y_2,y_2+\delta]\cup[y_1+2\pi-\delta,y_1+2\pi])}^2+\|w\|_{L^2([y_1,y_2])}^2\\
&\lesssim \f{\nu}{\delta^2|\alpha|}(\|\mathcal{N}_{\lambda}w\|_{L^2} \|w\|_{L^2}+\f{2\nu}{\sqrt{\delta}}\|w'\|_{L^2}^{\f32}\|w\|_{L^2}^{\f12})+\delta \|w'\|_{L^2}\|w\|_{L^2}+\delta\|w\|_{L^2(\pi, 2\pi)}^2,
\end{align*}
which shows that for $\delta\geq\f{y_2-y_1}{2}$,
\begin{align}
\|\cN_\la w\|_{L^2}\gtrsim \alpha|^{\f12}\|w\|_{L^2}.\nonumber
\end{align}

The same estimate holds for $\la\in [-1,0]$ due to $ \tau_{\pi}\mathcal{N}_{\lambda}=\mathcal{N}_{-\lambda}\tau_{\pi}$.
\end{proof}

\subsection{Resolvent estimate of $L_\la$}

In this subsection, we consider the full linearized operator $L_\la$.

\begin{proposition}\label{Prop: nonlocal1}
Given $0<\nu\leq1$, ${|\alpha|}\gg \nu^2$, $\lambda\in\mathbb{R}$ and $|\beta|>1(\beta\in\mathbb{R})$, there exists a constant $C>0$  independent of $\nu,\alpha,\lambda,\beta$,
such that for any $w\in H^2(\T)$
\begin{align*}
\|L_\la w \|_{L^2}\geq C|\alpha|^{\f12}(1-\beta^{-2})\|w\|_{L^2},
\end{align*}
where $(\partial_y^2-\beta^2)\varphi=w$.
\end{proposition}
\begin{proof}
Let $u=w+\varphi$. As $(\partial_y^2-\beta^2)\varphi=w$, $(\partial_y^2-\tilde{\beta}^2)\varphi=u$ for $\tilde{\beta}^2=\beta^2-1$. Thus,
\begin{align*}
\|\varphi''\|_{L^2}\leq\|u\|_{L^2},\quad \|\varphi\|_{L^2}\le \beta^{-2}\|w\|_{L^2}.
\end{align*}
Using Proposition \ref{Prop: nonlocal2}, we obtain
\begin{align*}
\big\|i\f{\alpha}{\nu}[(\sin y-\lambda)w+\sin y\varphi]-\nu\partial_y^2w\big\|_{L^2}&\geq \big\|i\f{\alpha}{\nu}[(\sin y-\lambda)u+\lambda\varphi]-\nu\partial_y^2u \|_{L^2}-\nu\|\varphi''\big\|_{L^2}\\
&\geq(C|\alpha|^{\f12}-\nu)\|u\|_{L^2}\gtrsim|\alpha|^{\f12}(1-\beta^{-2})\|w\|_{L^2}.
\end{align*}
\end{proof}

\begin{proposition}\label{Prop: nonlocal2}
Under the same assumptions as Proposition \ref{Prop: nonlocal1}, it holds that
\begin{align*}
\|i\f{\alpha}{\nu}[(\sin y-\lambda)u+\lambda\varphi]-\nu\partial_y^2u \|_{L^2}\geq C|\alpha|^{\f12}\|u\|_{L^2},
\end{align*}
where $(\partial_y^2-\tilde{\beta}^2)\varphi=u$ with $\beta^2-1=\tilde{\beta}^2$.
\end{proposition}

We need the following lemmas.

\begin{lemma}\label{lem:ell-1}
For $\lambda\in[0,1]$, let $0\leq y_1\leq\f{\pi}{2}\leq y_2\leq\pi$ so that $\lambda=\sin y_1=\sin y_2$. Let $\varphi$, $u$ satisfy the following Dirichlet problem
\begin{align}
\left\{
\begin{aligned}
&(\partial_y^2-\tilde{\beta}^2)\varphi=u,\quad y_1\leq y\leq y_2,\\
&\varphi(y_1)=\varphi(y_2)=0.
\end{aligned}
\right.\label{D1}
\end{align}
Then it holds that
\begin{align}
&-\big\langle\varphi,u\chi_{(y_1,y_2)}\big\rangle=\tilde{\beta}^2\int_{y_1}^{y_2}|\varphi(y)|^2dy+\int_{y_1}^{y_2}|\varphi'(y)|^2dy,\label{B0}\\
& -\big\langle\lambda\varphi,u\chi_{(y_1,y_2)}\big\rangle\leq \f{(\f{\pi}{2}-y_1)\sin y_1}{\cos y_1}\int_{y_1}^{y_2}(\sin y-\lambda)|u(y)|^2dy, \label{B1}
\end{align}
and
\begin{align}
&\lambda\Big(\int_{y_1}^{y_2}(\sin y-\lambda)^2\Big|(\f{\varphi}{\sin y-\lambda})'\Big|^2dy+\int_{y_1}^{y_2}\tilde{\beta}^2|\varphi|^2dy\Big)\nonumber\\&\qquad\leq \int_{y_1}^{y_2}(\sin y-\lambda)|u|^2dy+\big\langle\lambda\varphi,u\chi_{(y_1,y_2)}\big\rangle.\label{B2}
\end{align}
If $\varphi$, $u$ satisfy
\begin{align}
\left\{
\begin{aligned}
&(\partial_y^2-\tilde{\beta}^2)\varphi=u,\quad y_2\leq y\leq y_1+2\pi,\\
&\varphi(y_2)=\varphi(y_1+2\pi)=0.
\end{aligned}
\right.\label{D1}
\end{align}
Then we have
\begin{align}
-\big\langle\varphi,u\chi_{(y_2,y_1+2\pi)}\big\rangle=\tilde{\beta}^2\int_{y_2}^{y_1+2\pi}|\varphi(y)|^2dy+\int_{y_2}^{y_1+2\pi}|\varphi'(y)|^2dy.\label{B02}
\end{align}

 \end{lemma}

\begin{proof}
The equalities (\ref{B0}) and \eqref{B02} are obvious. Let\begin{align*}
0\geq K_1(y,z)=
\left\{
\begin{aligned}
&\f{\sinh\tilde{\beta}(y-y_2)\sinh\tilde{\beta}(z-y_1)}{\tilde{\beta}\sinh\tilde{\beta}(y_2-y_1)},\quad y\geq z,\\
&\f{\sinh\tilde{\beta}(y-y_1)\sinh\tilde{\beta}(z-y_2)}{\tilde{\beta}\sinh\tilde{\beta}(y_2-y_1)},\quad y\leq z.
\end{aligned}
\right.\label{G1}
\end{align*}
Then we have $\varphi(y)=\int_{y_1}^{y_2}K_1(y,z)u(z)dz$. Let\begin{align}
I(f)(y)=-\int_{y_1}^{y_2}K_1(y,z)f(z)dz=(y-y_1)(\pi-y-y_1)
\end{align}
for $f(y)=2+\tilde{\beta}^2(y-y_1)(\pi-y_1-y).$  Thus, for $y_1\leq y\leq y_2$
\begin{align*}
\f{\lambda I(f)}{f}(y)\leq\lambda\f{(y-y_1)(\pi-y-y_1)}{2(\sin y-\lambda)}(\sin y-\lambda),
\end{align*}
from which and (\ref{A.5.1}), we infer that for any $\la\in [0,1]$
\begin{align*}
\f{\lambda I(f)}{f}\leq \f{(\f{\pi}{2}-y_1)\sin y_1}{\cos y_1}(\sin y-\lambda).\end{align*}


Since $K_1(y,z)\leq0$ and $K_1(y,z)$ is symmetric in $y,z$, we obtain
\begin{align*}
&-\big\langle\lambda\varphi,u\chi_{(y_1,y_2)}\big\rangle\\
&=-\int_{y_1}^{y_2}\int_{y_1}^{y_2}K_1(y,z)u(z)\bar{u}(y)dzdy\\
&\leq \f{1}{2}\int_{y_1}^{y_2}\int_{y_1}^{y_2}-K_1(y,z)\f{f(z)}{f(y)}|u(y)|^2dzdy+\f{1}{2}\int_{y_1}^{y_2}\int_{y_1}^{y_2}-K_1(y,z)\f{f(y)}{f(z)}|u(z)|^2dydz\\
&=\int_{y_1}^{y_2}\f{I(f)(y)}{f(y)}|u(y)|^2dy\leq \f{(\f{\pi}{2}-y_1)\sin y_1}{\cos y_1}\int_{y_1}^{y_2}(\sin y-\lambda)|u(y)|^2dy,
\end{align*}
which shows (\ref{B1}).

By integration by parts, we have
\begin{align*}
\int_{y_1}^{y_2}|\varphi'|^2-\f{\sin y}{\sin y-\lambda}|\varphi|^2dy=\int_{y_1}^{y_2}(\sin y-\lambda)^2\Big|(\f{\varphi}{\sin y-\lambda})'\Big|^2dy.
\end{align*}
On the other hand, we have
\beno
(\sin y-\la)|u|^2+2\la \varphi u+\la^2\f {|\varphi|^2} {\sin y-\la}\ge 0.
\eeno
Therefore, we obtain
\begin{align*}
\int_{y_1}^{y_2}(\sin y-\lambda)|u|^2dy+\big\langle\lambda\varphi,u\chi_{(y_1,y_2)}\big\rangle\geq& -\big\langle\lambda\varphi,u\chi_{(y_1,y_2)}\big\rangle-\lambda^2\int_{y_1}^{y_2}\f{|\varphi|^2}{\sin y-\lambda}dy\\
=&\lambda\Big(\int_{y_1}^{y_2}|\varphi'|^2+\tilde{\beta}^2|\varphi|^2dy\Big)-\lambda^2\int_{y_1}^{y_2}\f{|\varphi|^2}{\sin y-\lambda}dy\\
\geq& \lambda\Big(\int_{y_1}^{y_2}|\varphi'|^2-\f{\sin y}{\sin y-\lambda}|\varphi|^2dy\Big)+\lambda\int_{y_1}^{y_2}\tilde{\beta}^2|\varphi|^2dy\\
=&\lambda \int_{y_1}^{y_2}(\sin y-\lambda)^2\Big|(\f{\varphi}{\sin y-\lambda})'\Big|^2dy +\lambda\int_{y_1}^{y_2}\tilde{\beta}^2|\varphi|^2dy,
\end{align*}
which shows (\ref{B2}).
\end{proof}

\begin{lemma}\label{lem:ell-2}
For $\lambda\in[0,1]$, let $ 0\leq y_1\leq\f{\pi}{2}\leq y_2\leq\pi$ so that $\lambda=\sin y_1=\sin y_2$. If $(\varphi,u)$ satisfies \eqref{D1}, then we have
\begin{align}
\int_{y_1}^{y_2}\f{|\varphi(y)|^2}{(\sin y-\lambda)^2}dy\lesssim \f{1}{(y_2-y_1)^2}\int_{y_1}^{y_2}\Big|(\f{\varphi(y)}{\sin y-\lambda})'\Big|^2(\sin y-\lambda)^2dy+\f{|\varphi(\f{\pi}{2})|^2}{(y_2-y_1)^3},\label{K.4.14}
\end{align}
\end{lemma}

\begin{proof} Let us first assume that $\varphi(\f{\pi}{2})=0$. Then we have
\begin{align}
\nonumber&\int_{y_1}^{y_2}\f{\varphi(y)}{\sin y-\lambda}g(y)dy=\int_{y_1}^{y_2}\int_{\f{\pi}{2}}^y(\f{\varphi(z)}{\sin z-\lambda})'dzg(y)dy\\
\nonumber&=\int_{\f{\pi}{2}}^{y_2}(\f{\varphi(z)}{\sin z-\lambda})'dz\int_{z}^{y_2}g(y)dy-\int_{y_1}^{\f{\pi}{2}}(\f{\varphi(z)}{\sin z-\lambda})'dz\int_{y_1}^zg(y)dy\\
\nonumber&\leq\Big(\int_{\f{\pi}{2}}^{y_2}\Big|(\f{\varphi(z)}{\sin z-\lambda})'(z-y_2)\Big|^2dz\Big)^{\f12}\Big(\int_{\f{\pi}{2}}^{y_2}\Big|\f{\int_{z}^{y_2}g(y)dy}{y_2-z}\Big|^2dz\Big)^{\f12}\\
\nonumber&\quad+\Big(\int_{y_1}^{\f{\pi}{2}}\Big|(\f{\varphi(z)}{\sin z-\lambda})'(z-y_1)\Big|^2dz\Big)^{\f12}\Big(\int_{y_1}^{\f{\pi}{2}}\Big|\f{\int_{y_1}^zg(y)dy}{y_1-z}\Big|^2dz\Big)^{\f12}\\
&\leq C \Big(\int_{y_1}^{y_2}\Big|(\f{\varphi(z)}{\sin z-\lambda})'\Big|^2\min((z-y_1)^2,(z-y_2)^2)dz\Big)^{\f12}\Big(\int_{y_1}^{y_2}|g(y)|^2dy\Big)^{\f12}.\label{Hardy1}
\end{align}
Thanks to (\ref{F4.3}), we know that
\begin{align}
\int_{y_1}^{y_2}\Big|(\f{\varphi(z)}{\sin z-\lambda})'\Big|^2\min((z-y_1)^2,(z-y_2)^2)dz\sim \int_{y_1}^{y_2}\Big|(\f{\varphi(z)}{\sin z-\lambda})'\Big|^2\f{(\sin z-\lambda)^2}{(y_2-y_1)^2}dz.\label{eq:varphi-int}
\end{align}
Taking $g(y)=\f{\overline{\varphi(y)}}{\sin y-\lambda}$ in \eqref{Hardy1}, we obtain
\begin{align*}
\int_{y_1}^{y_2}\f{|\varphi(y)|^2}{(\sin y-\lambda)^2}dy\lesssim \int_{y_1}^{y_2}\Big|(\f{\varphi(y)}{\sin y-\lambda})'\Big|^2\f{(\sin y-\lambda)^2}{(y_2-y_1)^2}dy.
\end{align*}

If $\varphi(\f{\pi}{2})\neq0$, then $\f{\varphi(y)}{\sin y-\lambda}=\int_{\f{\pi}{2}}^y(\f{\varphi(z)}{\sin z-\lambda})'dz+\f{\varphi(\f{\pi}{2})}{1-\lambda}$. Similarly, we have
\begin{align*}
&\int_{y_1}^{y_2}\Big(\int_{\f{\pi}{2}}^y(\f{\varphi(z)}{\sin z-\lambda})'dz+\f{\varphi(\f{\pi}{2})}{1-\lambda}\Big)g(y)dy\\
&\lesssim \Big(\f{1}{(y_2-y_1)^2}\int_{y_1}^{y_2}|(\f{\varphi(y)}{\sin y-\lambda})'|^2(\sin y-\lambda)^2dy\Big)^{\f12}\Big(\int_{y_1}^{y_2}|g(y)|^2dy\Big)^{\f12}\\
&\qquad+\f{|\varphi(\f{\pi}{2})|}{1-\lambda}(y_2-y_1)^{\f12}\Big(\int_{y_1}^{y_2}|g(y)|^2dy\Big)^{\f12},
\end{align*}
which implies (\ref{K.4.14}) by noting that $1-\lambda\sim(y_2-y_1)^2$ and taking $g(y)=\f{\overline{\varphi(y)}}{\sin y-\lambda}$.
\end{proof}

In the following lemmas, we assume that $(\varphi,u)$ satisfies $(\partial_y^2-\tilde{\beta}^2)\varphi=u$. We introduce the decomposition $\varphi=\varphi_1+\varphi_2$, where
\begin{align*}
\varphi_2(y)=
\left\{
\begin{aligned}
&\f{\sinh\tilde{\beta}(y-y_1)}{\sinh\tilde{\beta}(y_2-y_1)}\varphi(y_2)+\f{\sinh\tilde{\beta}(y_2-y)}{\sinh\tilde{\beta}(y_2-y_1)}\varphi(y_1),\quad y_1\leq y\leq y_2,\\
&\f{\sinh\tilde{\beta}(y-y_2)}{\sinh\tilde{\beta}(y_1+2\pi-y_2)}\varphi(y_1)+\f{\sinh\tilde{\beta}(y_1+2\pi-y)}{\sinh\tilde{\beta}(y_1+2\pi-y_2)}\varphi(y_2),\quad y_2\leq y\leq y_1+2\pi.
\end{aligned}
\right.
\end{align*}
Then we have
\begin{align*}
\left\{
\begin{aligned}
&(\partial_y^2-\tilde{\beta}^2)\varphi_1=u,\\
&\varphi_1(y_1)=\varphi_1(y_2)=\varphi_1(y_1+2\pi)=0,
\end{aligned}
\right.
\end{align*}
and
\begin{align*}
\left\{
\begin{aligned}
&(\partial_y^2-\tilde{\beta}^2)\varphi_2=0,\\
&\varphi_2(y_1)=\varphi_2(y_1+2\pi)=\varphi(y_1),\quad \varphi_2(y_2)=\varphi(y_2).
\end{aligned}
\right.
\end{align*}

\begin{lemma}
For $\lambda\in[0,1]$, let $ 0\leq y_1\leq\f{\pi}{2}\leq y_2\leq\pi$ so that $\lambda=\sin y_1=\sin y_2$. Then it holds that for any $\delta\in(0,1]$
\begin{align}
\| u\|_{L^2(y_2,y_1+2\pi)}^2\leq C\mathcal{E}_1(u),\quad\quad\forall u\in L^2,\label{Key1.1}
\end{align}
where
\begin{align*}
\mathcal{E}_1(u)=&\f{\nu}{|\alpha|(y_2-y_1+\delta)\delta}\|\mathcal{L}_{\lambda}u\|_{L^2}\|u\|_{L^{2}}+\f{\nu^2}{|\alpha|(y_2-y_1+\delta)\delta}\|u'\bar{u}\|_{L^{\infty}(B(y_1,\delta)\cup B(y_2,\delta))}\\
&+\f{\nu^2}{|\alpha|(y_2-y_1+\delta)\delta^{\f32}}\|u'\|_{L^2}\|u\|_{L^{\infty}}+\delta\|u\|_{L^{\infty}}^2+\f{\|\lambda\varphi_2\|_{L^{\infty}}^2}{(y_2-y_1+\delta)^2\delta}
\end{align*}
with $\mathcal{L}_{\lambda}u=i\f{\alpha}{\nu}[(\sin y-\lambda)u+\lambda\varphi]-\nu\partial_y^2u$.
\end{lemma}

\begin{proof}By integration by parts, we get
\begin{align*}
&\Big|\text{Im}\big\langle-\nu\partial_y^2u+i\f{\alpha}{\nu}[(\sin y-\lambda)u+\lambda\varphi],\chi_{(y_2,y_1+2\pi)}u\big\rangle\Big|\\
&=\Big|\f{\alpha}{\nu}\int_{y_2}^{y_1+2\pi}(\sin y-\lambda)|u(y)|^2dy+\nu\text{Im}(u'\bar{u}(y_2)-u'\bar{u}(y_{1}))+\text{Im}(i\f{\alpha}{\nu}\int_{y_2}^{y_1+2\pi}\lambda\varphi \bar{u}dy)\Big|\\
&\geq\f{|\alpha|}{\nu}\Big(\int_{y_2}^{y_1+2\pi}(\lambda-\sin y)|u(y)|^2dy-\int_{y_2}^{y_1+2\pi}\lambda\varphi_1\bar{u}dy\Big)\\
&\qquad-\nu\big(|u'\bar{u}(y_2)|+|u'\bar{u}(y_{1})|\big)-\f{|\alpha|}{\nu}\Big|\int_{y_2}^{y_1+2\pi}\lambda\varphi_2 \bar{u}dy\Big|,
\end{align*}
which implies that
\begin{align}
\nonumber&\int_{y_2}^{y_1+2\pi}(\lambda-\sin y)|u(y)|^2dy+\int_{y_2}^{y_1+2\pi}-\lambda \varphi_1\overline{u}dy\\
&\leq \f{\nu}{|\alpha|} \|\mathcal{L}_{\lambda}u\|_{L^2}\|u\|_{L^2}+\f{2\nu^2}{|\alpha|}\|u'\bar{u}\|_{L^{\infty}(B(y_1,\delta)\cup B(y_2,\delta))}+\Big|\int_{y_2}^{y_1+2\pi}\lambda\varphi_2\overline{u}dy\Big|\label{K.4.16.1}.
\end{align}
Similarly, we have
\begin{align*}
&\Big|\text{Im}\Big\langle-\nu\partial_y^2u+i\f{\alpha}{\nu}[(\sin y-\lambda)u+\lambda\varphi],\chi_{(y_2+\delta,y_1+2\pi-\delta)}\f{u}{\sin y-\lambda} \Big\rangle\Big|\\
&=\Big|\text{Im}\Big(-\nu\int_{y_2+\delta}^{y_1+2\pi-\delta}\f{u''\bar{u}}{\sin y-\lambda}dy+i\f{\alpha}{\nu}\int_{y_2+\delta}^{y_1+2\pi-\delta}|u|^2dy+i\f{\alpha}{\nu}\int_{y_2+\delta}^{y_1+2\pi-\delta}\f{\lambda\varphi \bar{u}}{\sin y-\lambda}dy\Big)\Big|\\
&=\Big|\text{Im}\Big(-\nu \f{u'\bar{u}}{\sin y-\lambda}\Big|_{y_2+\delta}^{y_1+2\pi-\delta}+\nu\int_{y_2+\delta}^{y_1+2\pi-\delta}\f{|u'|^2}{\sin y-\lambda}dy-\nu\int_{y_2+\delta}^{y_1+2\pi-\delta}u'\bar{u}\f{\cos y}{(\sin y-\lambda)^2}dy\Big)\\
&\quad+\text{Im}\Big(i\f{\alpha}{\nu}\int_{y_2+\delta}^{y_1+2\pi-\delta}|u|^2dy+i\f{\alpha}{\nu}\int_{y_2+\delta}^{y_1+2\pi-\delta}\f{\lambda\varphi\bar{u}}{\sin y-\lambda}dy\Big)\Big|\\
&\geq -\nu\|u'\bar{u}\|_{L^{\infty}(B(y_1,\delta)\cup B(y_2,\delta))}\Big(\f{1}{|\sin(y_2+\delta)-\sin y_2|}+\f{1}{|\sin (y_1-\delta)-\sin y_1|}\Big)\\
&\quad-\nu\|u\|_{L^{\infty}} \|u'\|_{L^2}\Big\|\f{\cos y}{(\sin y-\lambda)^2}\Big\|_{L^2(y_2+\delta,y_1+2\pi-\delta)}+\f{|\alpha|}{2\nu}\int_{y_2+\delta}^{y_1+2\pi-\delta}|u|^2dy\\
&\quad-\f{|\alpha|}{2\nu}\int_{y_2+\delta}^{y_1+2\pi-\delta}\f{\lambda^2|\varphi|^2}{(\sin y-\lambda)^2}dy,
\end{align*}
which implies that
\begin{align*}
&\f{|\alpha|}{2\nu}\int_{y_2+\delta}^{y_1+2\pi-\delta}|u|^2dy\leq \|\mathcal{L}_{\lambda}u\|_{L^2}\Big\|\f{u}{\sin y-\lambda}\Big\|_{L^2(y_2+\delta,y_1+2\pi-\delta)}+\f{2\nu\|u'\bar{u}\|_{L^{\infty}(B(y_1,\delta)\cup B(y_2,\delta))}}{|\sin (y_1-\delta)-\sin y_1|}\\
&\qquad+\nu\|u\|_{L^{\infty}} \|u'\|_{L^2}\Big\|\f{\cos y}{(\sin y-\lambda)^2}\Big\|_{L^2(y_2+\delta,y_1+2\pi-\delta)}+\f{|\alpha|}{2\nu}\int_{y_2+\delta}^{y_1+2\pi-\delta}\f{\lambda^2|\varphi|^2}{(\sin y-\lambda)^2}dy,
\end{align*}
which along with (\ref{A4.2.1}) and (\ref{A4.2.3}) gives
\begin{align}
\nonumber\|u\|_{L^2(y_2,y_1+2\pi)}^2&\leq \|u\|_{L^2(y_2+\delta,y_1+2\pi-\delta)}^2+2\delta\|u\|_{L^{\infty}}^2\\
&\lesssim\f{\nu}{|\alpha|(y_2-y_1+\delta){\delta}}\|\mathcal{L}_{\lambda}u\|_{L^2}\|u\|_{L^2}+\f{\nu^2}{|\alpha|\delta(y_2-y_1+\delta)}\|u'\bar{u}\|_{L^{\infty}(B(y_1,\delta)\cup B(y_2,\delta))}\label{O1}\\
\nonumber&+\f{\nu^2}{|\alpha|\delta^{\f32}(y_2-y_1+\delta)}\|u'\|_{L^2}\|u\|_{L^{\infty}}+\delta\|u\|_{L^{\infty}}^2+\int_{y_2+\delta}^{y_1+2\pi-\delta}\f{\lambda^2|\varphi|^2}{(\sin y-\lambda)^2}dy.
\end{align}

It remains to estimate $\int_{y_2+\delta}^{y_1+2\pi-\delta}\f{\lambda^2|\varphi|^2}{(\sin y-\lambda)^2}dy$. We have
\begin{align*}
\int_{y_2+\delta}^{y_1+2\pi-\delta}\f{\lambda^2|\varphi|^2}{(\sin y-\lambda)^2}dy\lesssim&
\int_{y_2+\delta}^{y_1+2\pi-\delta}\f{\lambda^2|\varphi_1|^2}{(\sin y-\lambda)^2}dy+\int_{y_2+\delta}^{y_1+2\pi-\delta}\f{\lambda^2|\varphi_2|^2}{(\sin y-\lambda)^2}dy.
\end{align*}
For $y\in[y_2+\delta,\f{3\pi}{2}]$, we have
\begin{align*}
|\sin y-\sin y_2|&=2|\cos \f{y+y_2}{2}\sin\f{y-y_2}{2}|=2|\sin\f{y_1-y}{2}\sin\f{y-y_2}{2}|\sim(y-y_1)(y-y_2),
\end{align*}
and for $ y\in[\f{3\pi}{2},y_1+2\pi-\delta]$, we have
\begin{align*}
\nonumber|\sin y-\sin y_2|&=2|\cos \f{y+y_2}{2}\sin\f{y-y_2}{2}|=2|\sin\f{y_1-y}{2}\sin\f{y-y_2}{2}|\\
&=2|\sin\f{2\pi+y_1-y}{2}\sin\f{2\pi+y_2-y}{2}|\sim(2\pi+y_1-y)(2\pi+y_2-y),
\end{align*}
from which, we infer that
\begin{align*}
\int_{y_2+\delta}^{y_1+2\pi-\delta}\f{\lambda^2|\varphi_1|^2}{(\sin y-\lambda)^2}dy&\lesssim \int_{y_2+\delta}^{\f{3\pi}{2}}\f{\lambda^2|\varphi_1|^2}{(\sin y-\lambda)^2}dy+\int_{\f{3\pi}{2}}^{y_1+2\pi-\delta}\f{\lambda^2|\varphi_1|^2}{(\sin y-\lambda)^2}dy\\
&\lesssim\f{1}{(y_2-y_1+\delta)^2} \Big(\int_{y_2+\delta}^{\f{3\pi}{2}}\f{\lambda^2|\int_{y_2}^y \varphi'_1dz|^2}{( y-y_2)^2}dy+\int_{\f{3\pi}{2}}^{y_1+2\pi-\delta}\f{\lambda^2|\int_{y}^{y_1+2\pi}\varphi'_1dz|^2}{(y_1+2\pi-y)^2}dy\Big).
\end{align*}
Then by Hardy's inequality, we get
\begin{align}
\int_{y_2+\delta}^{y_1+2\pi-\delta}\f{\lambda^2|\varphi_1|^2}{(\sin y-\lambda)^2}dy\lesssim \f{1}{(y_2-y_1+\delta)^2}\int_{y_2+\delta}^{y_1+2\pi-\delta}\lambda^2|\varphi'_1|^2dy\label{4.21.1}.
\end{align}
By (\ref{B02}) and (\ref{K.4.16.1}), we have
\begin{align*}
&\f{1}{(y_2-y_1+\delta)^2}\int_{y_2}^{y_1+2\pi}\lambda^2|\varphi'_1|^2dy\leq\f{\lambda^2}{(y_2-y_1+\delta)^2}\langle-\varphi_1,u\chi_{(y_2,y_1+2\pi)}\rangle\\
&\lesssim\f{\lambda}{(y_2-y_1+\delta)^2}\Big(\f{\nu}{|\alpha|} \|\mathcal{L}_{\lambda}u\|_{L^2}\|u\|_{L^2}+\f{2\nu^2}{|\alpha|}\|u'\bar{u}\|_{L^{\infty}(B(y_1,\delta)\cup B(y_2,\delta))}+\big|\int_{y_2}^{y_1+2\pi}\lambda\varphi_2\overline{u}dy\big|\Big)\\
&\lesssim\f{\nu}{|\alpha|(y_2-y_1+\delta)^2} \|\mathcal{L}_{\lambda}u\|_{L^2}\|u\|_{L^2}+ \f{\nu^2}{|\alpha|(y_2-y_1+\delta)^2}\|u'\bar{u}\|_{L^{\infty}(B(y_1,\delta)\cup B(y_2,\delta))}\\
&\quad+\lambda\f{\|\lambda\varphi_2\|_{L^{\infty}}}{\delta^{\f12}(y_2-y_1+\delta)}\f{\delta^{\f32}\|u\|_{L^{\infty}}+\delta^{\f12}\|u\|_{L^{1}(y_2+\d,y_1+2\pi-\d)}}{y_2-y_1+\delta}.
\end{align*}
By (\ref{K.4.16.1}) and (\ref{A4.2.4}), we have
\begin{align*}
&\|u\|_{L^{1}(y_2+\d,y_1+2\pi-\d)}^2\leq\left(\int_{y_2}^{y_1+2\pi}(\lambda-\sin y)|u(y)|^2dy\right)\Big\|\f{1}{\sin y-\lambda}\Big\|_{L^1(y_2+\delta,y_1+2\pi-\delta)}\\
&\lesssim \left(\f{\nu}{|\alpha|} \|\mathcal{L}_{\lambda}u\|_{L^2}\|u\|_{L^2}+\f{2\nu^2}{|\alpha|}\|u'\bar{u}\|_{L^{\infty}(B(y_1,\delta)\cup B(y_2,\delta))}+\Big|\int_{y_2}^{y_1+2\pi}\lambda\varphi_2\overline{u}dy\Big|\right)\f{1+\ln(1+\frac{y_2-y_1}{\delta})}{(y_2-y_1+\delta)}\\
 &\leq \big(\mathcal{E}_1(u)(y_2-y_1+\delta)\delta+\|\lambda\varphi_2\|_{L^{\infty}}(2\delta\|u\|_{L^{\infty}}+\|u\|_{L^{1}(y_2+\d,y_1+2\pi-\d)})\big)\f{1+\ln(1+\frac{y_2-y_1}{\delta})}{(y_2-y_1+\delta)}\\
 &\lesssim \mathcal{E}_1(u)\delta\big(1+\ln(1+\frac{y_2-y_1}{\delta})\big)+\|\lambda\varphi_2\|_{L^{\infty}}\|u\|_{L^{1}(y_2+\d,y_1+2\pi-\d)}\f{1+\ln(1+\frac{y_2-y_1}{\delta})}{(y_2-y_1+\delta)},
\end{align*}
which gives
\begin{align}
\nonumber\|u\|_{L^{1}(y_2+\d,y_1+2\pi-\d)}^2&\lesssim \mathcal{E}_1(u)\delta\big(1+\ln(1+\frac{y_2-y_1}{\delta})\big)+\|\lambda\varphi_2\|_{L^{\infty}}^2\frac{(1+\ln(1+\frac{y_2-y_1}{\delta}))^2}{(y_2-y_1+\delta)^2}\\&\lesssim \mathcal{E}_1(u)\delta\big(1+\ln(1+\frac{y_2-y_1}{\delta})\big)^2\nonumber\\
&\lesssim \mathcal{E}_1(u)\delta\big(1+\frac{y_2-y_1}{\delta}\big)=\mathcal{E}_1(u)(y_2-y_1+\delta).\label{eq:u-L1}
\end{align}
Then by (\ref{4.21.1}), we obtain
\begin{align}
\int_{y_2+\delta}^{y_1+2\pi-\delta}\f{\lambda^2|\varphi_1|^2}{(\sin y-\lambda)^2}dy\lesssim\mathcal{E}_1(u)\label{O2}.
\end{align}

Thanks to (\ref{A4.2.2}),  we have
\begin{align}
\int_{y_2+\delta}^{y_1+2\pi-\delta}\f{\lambda^2|\varphi_2|^2}{(\sin y-\lambda)^2}dy&\leq\|\lambda\varphi_2\|_{L^{\infty}}^2 \int_{y_2+\delta}^{y_1+2\pi-\delta}\f{1}{(\sin y-\lambda)^2}dy\nonumber\\
&\lesssim\f{\|\lambda\varphi_2\|_{L^{\infty}}^2}{\delta(y_2-y_1+\delta)^2}\leq\mathcal{E}_1(u)\label{O3}.
\end{align}

Finally, (\ref{Key1.1}) follows from (\ref{O1}), (\ref{O2}), and (\ref{O3}).\end{proof}

\begin{lemma}
For $\lambda\in[0,1]$, let $ 0\leq y_1\leq\f{\pi}{2}\leq y_2\leq\pi$ so that $\lambda=\sin y_1=\sin y_2$.  Then it holds that for any $\delta\in(0,\f{y_2-y_1}{4}]$
\begin{align}
\| u\|_{L^2}^2\leq C\mathcal{E}_2(u),\label{Key1.2}
\end{align}
where
\begin{align*}
\mathcal{E}_2(u)=&\f{\nu}{|\alpha|\delta(y_2-y_1)}\|\mathcal{L}_{\lambda}u\|_{L^2}\|u\|_{L^2}+\f{\nu^2}{|\alpha|\delta(y_2-y_1)}\|u'\bar{u}\|_{L^{\infty}(B(y_1,\delta)\cup B(y_2,\delta))}+\delta\|u\|_{L^{\infty}}^2\\
&+\f{\nu^2}{|\alpha|\delta^{\f32}(y_2-y_1)}\|u'\|_{L^2}\|u\|_{L^{\infty}}+\f{\|\lambda \varphi_2\|_{L^{\infty}}^2}{(y_2-y_1)^2\delta}+\delta^3\|u'\|_{L^{\infty}(B(y_1,\delta)\cup B(y_2,\delta))}^2\\
&+\f{\nu^2}{|\alpha|^2}\f{1}{(y_2-y_1)^3\delta}\|\mathcal{L}_{\lambda}u\|_{L^2}^2+\f{\nu^4}{|\alpha|^2}\Big(\f{\|u'\|_{L^{\infty}(B(y_1,\delta)\cup B(y_2,\delta))}^2}{(y_2-y_1)^3\delta^2}+\f{\|u'\|_{L^2}^2}{(y_2-y_1)^3\delta^3}\Big).
\end{align*}
\end{lemma}

\begin{proof}We get by integration by parts that
\begin{align*}
&\Big|\text{Im}\langle-\nu\partial_y^2u+i\f{\alpha}{\nu}[(\sin y-\lambda)u+\lambda\varphi],\chi_{(y_1,y_2)}u \rangle\Big|\\
&=\Big|\f{\alpha}{\nu}\int_{y_1}^{y_2}(\sin y-\lambda)|u(y)|^2dy-\nu\text{Im}(u'\bar{u}(y_2)-u'\bar{u}(y_{1}))+\text{Im}\Big(i\f{\alpha}{\nu}\int_{y_1}^{y_2}\lambda\varphi \bar{u}dy\Big)\Big|\\
&\geq\f{|\alpha|}{\nu}\Big(\int_{y_1}^{y_2}(\sin y-\lambda)|u(y)|^2dy+\int_{y_1}^{y_2}\lambda\varphi_1\bar{u}dy\Big)-\nu\big(|u'\bar{u}(y_2)|+|u'\bar{u}(y_{1})|\big)-\f{|\alpha|}{\nu}\Big|\int_{y_1}^{y_2}\lambda\varphi_2 \bar{u}dy\Big|,
\end{align*}
which implies that
\begin{align}
\nonumber&\f{1}{(y_2-y_1)^2}\Big(\int_{y_1}^{y_2}(\sin y-\lambda)|u(y)|^2dy+\int_{y_1}^{y_2}\lambda \varphi_1\overline{u}dy\Big)\\
\nonumber&\leq \f{1}{(y_2-y_1)^2}\Big(\f{\nu}{|\alpha|} \|\mathcal{L}_{\lambda}u\|_{L^2}\|u\|_{L^2}+\f{2\nu^2}{|\alpha|}\|u'\bar{u}\|_{L^{\infty}(B(y_1,\delta)\cup B(y_2,\delta))}+\Big|\int_{y_1}^{y_2}\lambda\varphi_2\overline{u}dy\Big|\Big)\\
&\lesssim \mathcal{E}_2(u)+\f{\|\lambda\varphi_2\|_{L^{\infty}}}{(y_2-y_1)\delta^{\f12}}\delta^{\f12}\|u\|_{L^{\infty}}\lesssim\mathcal{E}_2(u)\label{K.4.16}.
\end{align}

Similarly, we have
\begin{align*}
&\Big|\text{Im}\Big\langle-\nu\partial_y^2u+i\f{\alpha}{\nu}[(\sin y-\lambda)u+\lambda\varphi],\chi_{(y_1+\delta,y_2-\delta)}\f{u}{\sin y-\lambda}\Big \rangle\Big|\\
&=\Big|\text{Im}\Big(-\nu\int_{y_1+\delta}^{y_2-\delta}\f{u''\bar{u}}{\sin y-\lambda}dy+i\f{\alpha}{\nu}\int_{y_1+\delta}^{y_2-\delta}|u|^2dy+i\f{\alpha}{\nu}\int_{y_1+\delta}^{y_2-\delta}\f{\lambda\varphi \bar{u}}{\sin y-\lambda}dy\Big)\Big|\\
&=\Big|\text{Im}\Big(-\nu \f{u'\bar{u}}{\sin y-\lambda}\Big|_{y_1+\delta}^{y_2-\delta}+\nu\int_{y_1+\delta}^{y_2-\delta}\f{|u'|^2}{\sin y-\lambda}dy-\nu\int_{y_1+\delta}^{y_2-\delta}u'\bar{u}\f{\cos y}{(\sin y-\lambda)^2}dy\Big)\\
&\quad+\text{Im}\Big(i\f{\alpha}{\nu}\int_{y_1+\delta}^{y_2-\delta}|u|^2dy+i\f{\alpha}{\nu}\int_{y_1+\delta}^{y_2-\delta}\f{\lambda\varphi\bar{u}}{\sin y-\lambda}dy\Big)\Big|\\
&\geq -\nu\|u'\bar{u}\|_{L^{\infty}(B(y_1,\delta)\cup B(y_2,\delta))}\Big(\f{1}{|\sin(y_2-\delta)-\sin y_2|}+\f{1}{|\sin (y_1+\delta)-\sin y_1|}\Big)\\
&\quad-\nu\|u\|_{L^{\infty}} \|u'\|_{L^2}\Big\|\f{\cos y}{(\sin y-\lambda)^2}\Big\|_{L^2([y_1+\delta,y_2-\delta])}+\f{|\alpha|}{2\nu}\int_{y_1+\delta}^{y_2-\delta}|u|^2dy-\f{|\alpha|}{2\nu}\int_{y_1+\delta}^{y_2-\delta}\f{\lambda^2|\varphi|^2}{(\sin y-\lambda)^2}dy,
\end{align*}
which implies that
\begin{align*}
\f{|\alpha|}{2\nu}\int_{y_1+\delta}^{y_2-\delta}|u|^2dy\leq &\|\mathcal{L}_{\lambda}u\|_{L^2}\Big\|\f{u}{\sin y-\lambda}\Big\|_{L^2([y_1+\delta,y_2-\delta])}+\f{2\nu\|u'\bar{u}\|_{L^{\infty}(B(y_1,\delta)\cup B(y_2,\delta))}}{|\sin (y_1+\delta)-\sin y_1|}\\
&+\nu\|u\|_{L^{\infty}} \|u'\|_{L^2}\Big\|\f{\cos y}{(\sin y-\lambda)^2}\Big\|_{L^2([y_1+\delta,y_2-\delta])}+\f{|\alpha|}{2\nu}\int_{y_1+\delta}^{y_2-\delta}\f{\lambda^2|\varphi|^2}{(\sin y-\lambda)^2}dy,
\end{align*}
which along with (\ref{A4.3.1}), (\ref{A4.3.3}), and $y_2-y_1\geq 4\delta$
gives
\begin{align}
\|u\|_{L^2(y_1,y_2)}^2\lesssim& \f{\nu}{|\alpha|\delta(y_2-y_1)}\|\mathcal{L}_{\lambda}u\|_{L^2}\|u\|_{L^2}+\f{\nu^2}{|\alpha|\delta(y_2-y_1)}\|u'\bar{u}\|_{L^{\infty}(B(y_1,\delta)\cup B(y_2,\delta))}\label{N1}\\
\nonumber&+\f{\nu^2}{|\alpha|\delta^{\f32}(y_2-y_1)}\|u'\|_{L^2}\|u\|_{L^{\infty}}+\delta\|u\|_{L^{\infty}}^2+\int_{y_1+\delta}^{y_2-\delta}\f{\lambda^2|\varphi|^2}{(\sin y-\lambda)^2}dy.
\end{align}

We bound the last term as
\begin{align*}
\int_{y_1+\delta}^{y_2-\delta}\f{\lambda^2|\varphi|^2}{(\sin y-\lambda)^2}dy\lesssim \int_{y_1+\delta}^{y_2-\delta}\f{\lambda^2|\varphi_1|^2}{(\sin y-\lambda)^2}dy+\int_{y_1+\delta}^{y_2-\delta}\f{\lambda^2|\varphi_2|^2}{(\sin y-\lambda)^2}dy.
\end{align*}
First of all, we get by (\ref{A4.3.2}) that
\begin{align}
\int_{y_1+\delta}^{y_2-\delta}\f{\lambda^2|\varphi_2|^2}{(\sin y-\lambda)^2}dy\leq\|\lambda\varphi_2\|_{L^{\infty}}^2\int_{y_1+\delta}^{y_2-\delta}\f{1}{(\sin y-\lambda)^2}dy\lesssim\f{\|\lambda\varphi_2\|_{L^{\infty}}^2}{(y_2-y_1)^2\delta}\leq\mathcal{E}_2(u)\label{N2}.
\end{align}

Thanks to \eqref{N1}, (\ref{Key1.1}), and $\mathcal{E}_1(u)\leq\mathcal{E}_2(u)$, it remains to prove that
\beno
\int_{y_1+\delta}^{y_2-\delta}\f{\lambda^2|\varphi_1|^2}{(\sin y-\lambda)^2}dy\lesssim \mathcal{E}_2(u).
\eeno
The proof is split into two cases.\smallskip

\no$\mathbf{Case\ 1.\ \lambda\in[0,\f{\sqrt{3}}2]}$.\smallskip

In this case, we have $0\leq y_1\leq \frac{\pi}{3},\ y_2=\pi-y_1$. By Hardy's inequality and (\ref{F4.3}), we get
\begin{align*}
\int_{y_1+\delta}^{y_2-\delta}\f{\lambda^2|\varphi_1|^2}{(\sin y-\lambda)^2}dy=& \int_{y_1+\delta}^{\f{\pi}{2}}\f{\lambda^2|\varphi_1|^2}{(\sin y-\lambda)^2}dy+\int_{\f{\pi}{2}}^{y_2-\delta}\f{\lambda^2|\varphi_1|^2}{(\sin y-\lambda)^2}dy\\
\lesssim&\f{\lambda^2}{(y_2-y_1)^2}\Big(\int_{y_1+\delta}^{\f{\pi}{2}}\f{|\int_{y_1}^y\varphi'_1dz|^2}{(y-y_1)^2}dy+\int_{\f{\pi}{2}}^{y_2-\delta}\f{|\int_{y_2}^y\varphi'_1dz|^2}{(y-y_2)^2}dy\Big)\\
\lesssim&\f{\lambda^2}{(y_2-y_1)^2}\int_{y_1}^{y_2}|\varphi'_1|^2dy.
\end{align*}
Thanks to (\ref{B1}) and (\ref{B0}), we have
\begin{align*}
&\int_{y_1}^{y_2}(\sin y-\lambda)|u|^2dy+\langle\lambda\varphi_1,u\chi_{(y_1,y_2)}\rangle\geq \Big(1-\f{(\f{\pi}{2}-y_1)\sin y_1}{\cos y_1}\Big)\int_{y_1}^{y_2}(\sin y-\lambda)|u|^2dy\\
&\quad\gtrsim \int_{y_1}^{y_2}(\sin y-\lambda)|u|^2dy
\gtrsim\langle-{\lambda}\varphi_1,u\chi_{(y_1,y_2)}\rangle\geq\lambda\int_{y_1}^{y_2}|\varphi'_1|^2dy.
\end{align*}
Then by (\ref{K.4.16}), we obtain
\begin{align}
\int_{y_1+\delta}^{y_2-\delta}\f{\lambda^2|\varphi_1|^2}{(\sin y-\lambda)^2}dy\lesssim\f{\int_{y_1}^{y_2}(\sin y-\lambda)|u|^2dy+\langle\lambda\varphi_1,u\chi_{(y_1,y_2)}\rangle}{(y_2-y_1)^2}\lesssim\mathcal{E}_2(u).\label{N3.0}
\end{align}

\no $\mathbf{Case\ 2. \ \lambda\in[\f{\sqrt{3}}2,1]}$. \smallskip

Thanks to (\ref{K.4.14}), (\ref{B2}), (\ref{K.4.16}), we find
\begin{align}
\nonumber\int_{y_1+\delta}^{y_2-\delta}\f{\lambda^2|\varphi_1|^2}{(\sin y-\lambda)^2}dy\lesssim& \f{\lambda^2}{(y_2-y_1)^2}\int_{y_1}^{y_2}|(\f{\varphi_1(y)}{\sin y-\lambda})'|^2(\sin y-\lambda)^2dy+\f{|\lambda\varphi_1(\f{\pi}{2})|^2}{(y_2-y_1)^3}\\
\nonumber\lesssim& \f{\lambda}{(y_2-y_1)^2}\Big(\int_{y_1}^{y_2}(\sin y-\lambda)|u|^2dy+\langle\lambda\varphi_1,u\chi_{(y_1,y_2)}\rangle\Big)+\f{|\lambda\varphi_1(\f{\pi}{2})|^2}{(y_2-y_1)^3}\\
\lesssim&\mathcal{E}_2(u)+\f{|\lambda\varphi_1(\f{\pi}{2})|^2}{(y_2-y_1)^3}.\label{Key2}
\end{align}
Thus, it remains to control $\f{|\lambda\varphi_1(\f{\pi}{2})|^2}{(y_2-y_1)^3}$.  For any $\delta\leq\theta\leq\f{y_2-y_1}{4}$, we have
\begin{align*}
&\int_{\f{\pi}{2}}^{y_2-\theta}\Big(\f{\varphi_1(y)}{\sin y-\lambda}\Big)'(y_2-\theta-y)dy-\int_{y_1+\theta}^{\f{\pi}{2}}\Big(\f{\varphi_1(y)}{\sin y-\lambda}\Big)'(y-y_1-\theta)dy\\
&=\f{\varphi_1(y)}{\sin y-\lambda}(y_2-\theta-y)\Big|_{\f{\pi}{2}}^{y_2-\theta}+\int_{\f{\pi}{2}}^{y_2-\theta}\f{\varphi_1(y)}{\sin y-\lambda}dy\\
&\quad-\f{\varphi_1(y)}{\sin y-\lambda}(y-y_1-\theta)\Big|_{y_1+\theta}^{\f{\pi}{2}}+\int_{y_1+\theta}^{\f{\pi}{2}}\f{\varphi_1(y)}{\sin y-\lambda}dy\\
&=\int_{y_1+\theta}^{y_2-\theta}\f{\varphi_1(y)}{\sin y-\lambda}dy-\f{\varphi_1(\f{\pi}{2})}{1-\lambda}(y_2-y_1-2\theta),
\end{align*}
which gives
\begin{align*}
(y_2-y_1-2\theta)\f{|\varphi_1(\f{\pi}{2})|}{1-\lambda}&\leq\Big| \int_{y_1+\theta}^{y_2-\theta}\f{\varphi_1(y)}{\sin y-\lambda}dy\Big|+\int_{y_1+\theta}^{y_2-\theta}\Big|(\f{\varphi_1(y)}{\sin y-\lambda})'\Big|\min(y_2-y,y-y_1)dy\\
&\lesssim \Big| \int_{y_1+\theta}^{y_2-\theta}\f{\varphi_1(y)}{\sin y-\lambda}dy\Big|+(y_2-y_1)^{\f12}\Big(\int_{y_1+\theta}^{y_2-\theta}\Big|(\f{\varphi_1(y)}{\sin y-\lambda})'\Big|^2\min(y_2-y,y-y_1)^2dy\Big)^{\f12}\\
&\lesssim \Big| \int_{y_1+\theta}^{y_2-\theta}\f{\varphi_1(y)}{\sin y-\lambda}dy\Big|+(y_2-y_1)^{\f12}\Big(\int_{y_1+\theta}^{y_2-\theta}\Big|(\f{\varphi_1(y)}{\sin y-\lambda})'\Big|^2\f{(\sin y-\lambda)^2}{(y_2-y_1)^2}dy\Big)^{\f12}.
\end{align*}
As $1-\lambda\sim(y_2-y_1)^2$,  this gives
\begin{align}
\f{\lambda^2|\varphi_1(\f{\pi}{2})|^2}{(y_2-y_1)^3}\lesssim\f{\lambda^2 \big| \int_{y_1+\theta}^{y_2-\theta}\f{\varphi_1(y)}{\sin y-\lambda}dy\big|^2}{y_2-y_1}+\lambda^2\int_{y_1+\theta}^{y_2-\theta}\Big|(\f{\varphi_1(y)}{\sin y-\lambda})'\Big|^2\f{(\sin y-\lambda)^2}{(y_2-y_1)^2}dy\label{KK1}.
\end{align}
By (\ref{B2}) and (\ref{K.4.16}),  we get
\begin{align*}
\lambda^2\int_{y_1}^{y_2}|(\f{\varphi_1(y)}{\sin y-\lambda})'|^2\f{(\sin y-\lambda)^2}{(y_2-y_1)^2}\lesssim \mathcal{E}_2(u).
\end{align*}

To estimate $\f{\lambda^2}{y_2-y_1} \big| \int_{y_1+\theta}^{y_2-\theta}\f{\varphi_1(y)}{\sin y-\lambda}dy\big|^2$, we consider two subcases:\smallskip

\no$\mathbf{Case\ 2.1.\ \tilde\beta^2(y_2-y_1)^2\geq1}$. \smallskip

Taking $\theta=\f{y_2-y_1}{4}$, we get by (\ref{A4.3.2}) that
\begin{align*}
\f{\lambda^2}{y_2-y_1} \Big| \int_{y_1+\theta}^{y_2-\theta}\f{\varphi_1(y)}{\sin y-\lambda}dy\Big|^2\leq& \f{\lambda^2}{y_2-y_1}\int_{y_1+\theta}^{y_2-\theta}|\varphi_1|^2dy\int_{y_1+\theta}^{y_2-\theta}\f{1}{(\sin y-\lambda)^2}dy\\
\lesssim& \f{\lambda^2}{(y_2-y_1)^3\theta} \int_{y_1+\theta}^{y_2-\theta}|\varphi_1|^2dy\lesssim\f{\lambda^2\tilde{\beta}^2}{(y_2-y_1)^2} \int_{y_1+\theta}^{y_2-\theta}|\varphi_1|^2dy,
\end{align*}
which along with (\ref{B2}) and (\ref{K.4.16}) gives
\begin{align}
\f{\lambda^2}{y_2-y_1} \big| \int_{y_1+\theta}^{y_2-\theta}\f{\varphi_1(y)}{\sin y-\lambda}dy\big|^2\lesssim\mathcal{E}_2(u).
\end{align}

\no$\mathbf{Case\ 2.2.\  \tilde{\beta}^2(y_2-y_1)^2\leq1}$. \smallskip

Take $\theta=\delta$ and
\begin{align*}
\chi(y)=\eta\Big(\f{y-\f{\pi}{2}}{y_2-y_1}\Big)\quad\text{with}\quad\eta(z)=
\left\{
\begin{aligned}
&1,\quad |z|\leq1,\\
&0,\quad |z|\geq2.
\end{aligned}
\right.
\end{align*}
By integration by parts, we get
\begin{align*}
&\Big|\int_{-\frac{\pi}{2}}^{\frac{3\pi}{2}}\varphi'\chi'+\tilde{\beta}^2\varphi\chi dy \Big|=\Big|\int_{-\frac{\pi}{2}}^{\frac{3\pi}{2}}-u(y)\chi(y)dy \Big|.
\end{align*}
Due to $\chi(y)=1$ for $y\in [y_1+\delta, y_2-\delta]$, this implies that
\begin{align}
\Big|\int_{y_1+\delta}^{y_2-\delta}-u(y)dy\Big|\leq\Big|\int_{-\f \pi 2}^{\f {3\pi} 2}\varphi'\chi'+\tilde{\beta}^2\varphi\chi dy \Big|+4\delta\|u\|_{L^{\infty}}+\int_{y_2+\delta}^{2\pi+(y_1-\delta)}|u(y)|dy\label{C1}.
\end{align}
Recall that $\mathcal{L}_{\lambda}u=i\f{\alpha}{\nu}[(\sin y-\lambda)u+\lambda\varphi]-\nu\partial_y^2u$. Then $u=\f{-i\f{\nu}{\alpha}(\mathcal{L}_{\lambda}u+\nu\partial^2_yu)-\lambda\varphi}{\sin y-\lambda}$, hence,\begin{align*}
&\Big|\int_{y_1+\delta}^{y_2-\delta}\f{\lambda\varphi}{\sin y-\lambda}dy\Big|-\Big|\int_{y_1+\delta}^{y_2-\delta}\f{i\f{\nu}{\alpha}(\mathcal{L}_{\lambda}u+\nu\partial^2_yu)}{\sin y-\lambda}dy\Big|\\
&\leq\Big|\int_{y_1+\delta}^{y_2-\delta}\f{i\f{\nu}{\alpha}(\mathcal{L}_{\lambda}u+\nu\partial^2_yu)+\lambda\varphi}{\sin y-\lambda}dy\Big|\\
&=\Big|\int_{y_1+\delta}^{y_2-\delta}-u(y)dy\Big|\leq \Big|\int_{-\f \pi 2}^{\f {3\pi} 2}\varphi'\chi'+\tilde{\beta}^2\varphi\chi dy \Big|+4\delta\|u\|_{L^{\infty}}+\int_{y_2+\delta}^{2\pi+(y_1-\delta)}|u(y)|dy,
\end{align*}
which implies that
\begin{align*}
&\f{1}{y_2-y_1}\Big|\int_{y_1+\delta}^{y_2-\delta}\f{\lambda\varphi_1}{\sin y-\lambda}dy\Big|^2\\
&\lesssim \f{\delta^2}{y_2-y_1}\|u\|_{L^{\infty}}^2+\f{\nu^2}{|\alpha|^2(y_2-y_1)}\Big|\int_{y_1+\delta}^{y_2-\delta}\f{\mathcal{L}_{\lambda}u}{\sin y-\lambda}dy\Big|^2+\f{\nu^4}{|\alpha|^2(y_2-y_1)}\Big|\int_{y_1+\delta}^{y_2-\delta}\f{u''}{\sin y-\lambda}dy\Big|^2\\
&\quad+\f{1}{y_2-y_1}\Big|\int_{y_1+\delta}^{y_2-\delta}\f{\lambda\varphi_2}{\sin y-\lambda}dy\Big|^2+\f{1}{y_2-y_1}\Big|\int_{-\f \pi 2}^{\f {3\pi} 2}\varphi'\chi'+\tilde{\beta}^2\varphi\chi dy \Big|^2+\f{1}{y_2-y_1}\Big(\int_{y_2+\delta}^{2\pi+(y_1-\delta)}|u(y)|dy\Big)^2\\
&:=\text{I}_1+\cdots+\text{I}_6.
\end{align*}

Next we estimate $\text{I}_i(i=1,\cdots 6)$ one by one. Due to $y_2-y_1\ge 4\delta$, we have
\begin{align*}
\text{I}_1\leq\delta\|u\|_{L^{\infty}}^2\lesssim\mathcal{E}_2(u).
\end{align*}
Thanks to (\ref{A4.3.2}), we get
\begin{align*}
\text{I}_2\lesssim\f{\nu^2\|\mathcal{L}_{\lambda}u\|_{L^2}^2 }{|\alpha|^2(y_2-y_1)}\int_{y_1+\delta}^{y_2-\delta}\f{1}{(\sin y-\lambda)^2}dy\lesssim\f{\nu^2}{|\alpha|^2}\f{1}{\delta(y_2-y_1)^3}\|\mathcal{L}_{\lambda}u\|_{L^2}^2\leq\mathcal{E}_2(u).
\end{align*}
By (\ref{A4.3.1}) and (\ref{A4.3.3}), we have
\begin{align}
\nonumber\text{I}_3&=\f{\nu^4}{|\alpha|^2(y_2-y_1)}\Big|\f{u'}{\sin y-\lambda}\Big|_{y_1+\delta}^{y_2-\delta}+\int_{y_1+\delta}^{y_2-\delta}\f{u'\cos y}{(\sin y-\lambda)^2}dy  \Big|^2\\
&\lesssim\f{\nu^4}{|\alpha|^2}\Big(\f{\|u'\|_{L^{\infty}(B(y_1,\delta)\cup B(y_2,\delta))}^2}{\delta^2(y_2-y_1)^3}+\f{\|u'\|_{L^2}^2}{(y_2-y_1)^3\delta^3}\Big)\leq\mathcal{E}_2(u)\nonumber.
\end{align}

By (\ref{A4.3.2}), we get
\begin{align}
\text{I}_{4}&=\f{\big|\int_{y_1+\delta}^{y_2-\delta}\f{\lambda\varphi_2}{\sin y-\lambda}dy\big|^2}{y_2-y_1}\lesssim \|\lambda\varphi_2\|_{L^{\infty}}^2\int_{y_1+\delta}^{y_2-\delta}\f{dy}{(\sin y-\lambda)^2}\lesssim\f{\|\lambda\varphi_2\|_{L^{\infty}}^2}{(y_2-y_1)^2\delta}\leq\mathcal{E}_2(u)\nonumber
\end{align}
By \eqref{eq:u-L1}, we have
\begin{align*}
\text{I}_6=\f{1}{y_2-y_1}\|u\|_{L^{1}(y_2+\d,y_1+2\pi-\d)}^2\leq \f{C}{y_2-y_1}\mathcal{E}_1(u)(y_2-y_1+\d)\leq C\mathcal{E}_1(u)\leq C\mathcal{E}_2(u).
\end{align*}
For $\text{I}_5$, we have
\begin{align*}
\text{I}_{5}=&\f{1}{y_2-y_1}\Big|\int_{-\f \pi 2}^{\f {3\pi} 2}\varphi'\chi'+\tilde{\beta}^2\varphi\chi dy \Big|^2\\
\lesssim&\f{1}{y_2-y_1}\Big(\|\varphi'_1\|_{L^2([-\frac{\pi}{2},\frac{3\pi}{2}]\backslash(y_1-\delta,y_2+\delta))}^2\|\chi'\|_{L^2}^2
+\|\varphi'_2\|_{L^1}^2\|\chi'\|_{L^{\infty}}^2+\tilde{\beta}^4\|\varphi\|_{L^1(B(\f{\pi}{2},2(y_2-y_1)))}^2\Big)\\
\lesssim&\f{1}{y_2-y_1}\Big[\frac{\|\varphi'_1\|_{L^2([-\frac{\pi}{2},\frac{3\pi}{2}]\backslash(y_1-\delta,y_2+\delta))}^2}{y_2-y_1}
+\frac{\|\varphi'_2\|_{L^1}^2}{(y_2-y_1)^2}+\tilde{\beta}^4\Big(\int_{-2(y_2-y_1)+\f{\pi}{2}}^{2(y_2-y_1)+\f{\pi}{2}}|\varphi|dy\Big)^2\Big]\\
=&\text{I}_{5}^1+\text{I}_{5}^2+\text{I}_{5}^3.
\end{align*}
By \eqref{B02}, (\ref{K.4.16.1}), \eqref{eq:u-L1} and $\lambda\in[\f{\sqrt{3}}2,1]$, we have
\begin{align*}
\text{I}_{5}^1=&\f{1}{(y_2-y_1)^2}\|\varphi'_1\|_{L^2([-\frac{\pi}{2},\frac{3\pi}{2}]\backslash(y_1-\delta,y_2+\delta))}^2\lesssim \f{1}{(y_2-y_1)^2}\|\varphi'_1\|_{L^2{(y_2,y_1+2\pi)}}^2\\
\lesssim&\f{1}{(y_2-y_1)^2\lambda}\big\langle-\lambda\varphi_1,u\chi_{(y_2,y_1+2\pi)}\big\rangle\lesssim \f{1}{(y_2-y_1)^2}\big\langle-\lambda\varphi_1,u\chi_{(y_2,y_1+2\pi)}\big\rangle\\
\lesssim&\f{1}{(y_2-y_1)^2}\Big(\f{\nu}{|\alpha|} \|\mathcal{L}_{\lambda}u\|_{L^2}\|u\|_{L^2}+\f{2\nu^2}{|\alpha|}\|u'\bar{u}\|_{L^{\infty}(B(y_1,\delta)\cup B(y_2,\delta))}+\Big|\int_{y_2}^{y_1+2\pi}\lambda\varphi_2\overline{u}dy\Big|\Big)\\
\lesssim&\mathcal{E}_2(u)+\f{\|\lambda\varphi_2\|_{L^{\infty}}}{(y_2-y_1)^2\delta^{\f12}}
\big(\delta^{\f32}\|u\|_{L^{\infty}}+\delta^{\f12}\|u\|_{L^{1}(y_2+\d,y_1+2\pi-\d)}\big)\lesssim\mathcal{E}_2(u).
\end{align*}
By the definition of $\varphi_2$ and the monotonicity of $\sinh$, we get
\begin{align*}
\text{I}_{5}^2=\f{1}{(y_2-y_1)^3}\|\varphi'_2\|_{L^1}^2\lesssim\f{\|\varphi_2\|_{L^{\infty}}^2}{(y_2-y_1)^3} \lesssim\f{\|\lambda\varphi_2\|_{L^{\infty}}^2}{(y_2-y_1)^2\delta}\leq\mathcal{E}_2(u).
\end{align*}
Due to $\tilde{\beta}^2(y_2-y_1)^2\leq1$, we have
\begin{align*}
\text{I}_{5}^3&=\f{\tilde{\beta}^4\Big(\int_{-2(y_2-y_1)+\f{\pi}{2}}^{2(y_2-y_1)+\f{\pi}{2}}|\varphi|dy\Big)^2}{y_2-y_1}\lesssim \f{\tilde{\beta}^4\Big(\int_{-2(y_2-y_1)+\f{\pi}{2}}^{2(y_2-y_1)+\f{\pi}{2}}|\varphi_1|dy\Big)^2+\tilde{\beta}^4\Big(\int_{-2(y_2-y_1)+\f{\pi}{2}}^{2(y_2-y_1)+\f{\pi}{2}}|\varphi_2|dy\Big)^2}{y_2-y_1}\\
&\lesssim\tilde{\beta}^4 \int_{-2(y_2-y_1)+\f{\pi}{2}}^{2(y_2-y_1)+\f{\pi}{2}}|\varphi_1|^2dy+\tilde{\beta}^4(y_2-y_1)\|\varphi_2\|_{L^{\infty}}^2\\
&\lesssim\f{1}{\lambda(y_2-y_1)^2}\lambda\tilde{\beta}^2\int_{-2(y_2-y_1)+\f{\pi}{2}}^{2(y_2-y_1)+\f{\pi}{2}}|\varphi_1|^2dy+\f{\|\lambda\varphi_2\|_{L^{\infty}}^2}{(y_2-y_1)^2\delta},
\end{align*}
which along with (\ref{B2}), (\ref{K.4.16}), (\ref{B02}) and (\ref{K.4.16.1})
gives
\begin{align}
\text{I}_{5}^3\lesssim\mathcal{E}_2(u),\quad\text{thus,}\quad  \text{I}_{5}\lesssim\mathcal{E}_2(u).
\end{align}

Summing up the estimates of $\text{I}_i(i=1,\cdots 6)$, we conclude
\beno
\f{1}{y_2-y_1}\Big|\int_{y_1+\delta}^{y_2-\delta}\f{\lambda\varphi_1}{\sin y-\lambda}dy\Big|^2
\lesssim \mathcal{E}_2(u).
\eeno
This completes the proof of  the lemma.
\end{proof}

\begin{lemma}
For any $\nu\in(0,1]$, $\lambda\in[0,1]$ and $\delta\in(0,1]$, it holds that
\begin{align}
\f{\|\lambda \varphi_2\|_{L^{\infty}}^2}{(y_2-y_1+\delta)^2\delta}\lesssim\mathcal{F}_1(u)\label{K1.3},
\end{align}
where
\begin{align*}
 \mathcal{F}_1(u)=\f{\nu^2\|\mathcal{L}_{\lambda}u\|_{L^2}^2}{|\alpha|^2\delta^2(y_2-y_1+\delta)^2}+
 \f{\nu^3\|\mathcal{L}_{\lambda}u\|_{L^2}\|u\|_{L^2}}{|\alpha|^2\delta^4(y_2-y_1+\delta)^2}+\delta\|u\|_{L^{\infty}}^2.
 \end{align*}
 \end{lemma}

\begin{proof} For any $a>0, b>0$, we have
\begin{align*}
\int_{a-b}^{a+b}\varphi(y)dy=2b\varphi(a)+\int_a^{a+b}\varphi''(z)\f{(a-z)^2}{2}dz+\int_{a-b}^a\varphi''(z)\f{(z-b)^2}{2}dz,
\end{align*}
which implies that
\begin{align}
|\varphi(a)|\leq\f{1}{2b}|\int_{a-b}^{a+b}\varphi(y)dy|+b^2\|\varphi''\|_{L^{\infty}}.\label{eq:cal}
\end{align}
For any $ a\in B(y_1,\delta)$, there exits $b\in[\f{\delta}{2},\delta]$ so that
\begin{align*}
|u'(a+b)|^2+|u'(a-b)|^2\leq\f{2}{\delta}\|u'\|_{L^2}^2.
\end{align*}
Recall that $\mathcal{L}_{\lambda}u=i\f{\alpha}{\nu}[(\sin y-\lambda)u+\lambda\varphi]-\nu\partial_y^2u$. Then we have
\begin{align*}
|\int_{a-b}^{a+b}\lambda\varphi(y)dy|\leq&\f{\nu}{|\alpha|}\int_{a-b}^{a+b}|\mathcal{L}_{\lambda}u|dy+\f{\nu^2}{|\alpha|}(|u'(a+b)|+|u'(a-b)|)+\int_{a-b}^{a+b}|\sin y-\lambda|dy \|u\|_{L^{\infty}}\\
\lesssim&\f{\nu\delta^{\f12}}{|\alpha|}\|\mathcal{L}_{\lambda}u\|_{L^2}+\f{\nu^2}{|\alpha|\delta^{\f12}}\|u'\|_{L^2}+(y_2-y_1)\delta^2\|u\|_{L^{\infty}}.
\end{align*}
Thanks to $\|\varphi''\|_{L^{\infty}}\lesssim\|u\|_{L^{\infty}}$, we infer that for any $a\in B(y_1,\delta)$
\begin{align*}
\f{|\lambda\varphi(a)|}{(y_2-y_1+\delta)\delta^{\f12}}\lesssim\f{\nu}{|\alpha|(y_2-y_1+\delta)\delta} \|\mathcal{L}_{\lambda}u\|_{L^2}+\f{\nu^2}{|\alpha|\delta^2(y_2-y_1+\delta)}\|u'\|_{L^2}+\delta^{\f12}\|u\|_{L^{\infty}},
\end{align*}
which gives
\begin{align}
\f{\|\lambda \varphi\|_{L^{\infty}(B(y_1,\delta))}^2}{(y_2-y_1+\delta)^2\delta}\lesssim\f{\nu^2\|\mathcal{L}_{\lambda}u\|_{L^2}^2}{|\alpha|^2\delta^2(y_2-y_1+\delta)^2} + \f{\nu^4\|u'\|_{L^2}^2}{|\alpha|^2\delta^4(y_2-y_1+\delta)^2}+\delta\|u\|_{L^{\infty}}^2.\label{V1}
\end{align}
Similarly, we have
\begin{align}
\f{\|\lambda \varphi\|_{L^{\infty}(B(y_2,\delta))}^2}{(y_2-y_1+\delta)^2\delta}\lesssim\f{\nu^2\|\mathcal{L}_{\lambda}u\|_{L^2}^2}{|\alpha|^2\delta^2(y_2-y_1+\delta)^2} + \f{\nu^4\|u'\|_{L^2}^2}{|\alpha|^2\delta^4(y_2-y_1+\delta)^2}+\delta\|u\|_{L^{\infty}}^2.\nonumber
\end{align}
Thanks to $\|\lambda\varphi_2\|_{L^{\infty}}\leq\|\lambda\varphi\|_{L^{\infty}(B(y_1,\delta)\cup B(y_2,\delta))}$, we deduce that
\begin{align}
\f{\|\lambda \varphi_2\|_{L^{\infty}}^2}{(y_2-y_1+\delta)^2\delta}\lesssim\f{\nu^2\|\mathcal{L}_{\lambda}u\|_{L^2}^2}{|\alpha|^2\delta^2(y_2-y_1+\delta)^2} + \f{\nu^4\|u'\|_{L^2}^2}{|\alpha|^2\delta^4(y_2-y_1+\delta)^2}+\delta\|u\|_{L^{\infty}}^2\lesssim\mathcal{F}_1(u).\nonumber
\end{align}
this completes the proof.
\end{proof}

\begin{lemma}
For any $\nu\in(0,1]$, $ \lambda\in[0,1]$ and $\delta\in(0,1]$,
it holds that
\begin{align}
\delta^3\|u'\|_{L^{\infty}(B(y_1,\delta)\cup B(y_2,\delta))}^2\lesssim\mathcal{F}_2(u)\label{K2.3},
\end{align}
where
\begin{align*}
 \mathcal{F}_2(u)=\f{\delta^6(y_2-y_1+\delta)^2|\alpha|^2}{\nu^{4}}\mathcal{F}_1(u)+\f{\delta^4}{\nu^2}\|\mathcal{L}_{\lambda}u\|_{L^2}^2.
 \end{align*}
\end{lemma}
\begin{proof}
First of all, we have
\begin{align*}
\|u'\|_{L^{\infty}(B(y_1,\delta))}\lesssim\f{1}{\delta}\|u\|_{L^{\infty}}+\|u''\|_{L^1(B(y_1,\delta))}.
\end{align*}
Recalling $\mathcal{L}_{\lambda}u=i\f{\alpha}{\nu}[(\sin y-\lambda)u+\lambda\varphi]-\nu\partial_y^2u$, we find
\begin{align*}
&\delta^{\f32}\int_{B(y_1,\delta)}|u''|dy\leq\f{\delta^{\f32}|\alpha|}{\nu^2}\int_{y_1-\delta}^{y_1+\delta}|(\sin y-\lambda)u|dy+ \f{\delta^{\f32}|\alpha|}{\nu^2}\int_{y_1-\delta}^{y_1+\delta}|\lambda\varphi|dy+\f{\delta^{\f32}}{\nu} \int_{y_1-\delta}^{y_1+\delta}|\mathcal{L}_{\lambda}u|dy\\
&\quad\lesssim\f{\delta^3(y_2-y_1+\delta)|\alpha|}{\nu^2}\delta^{\f12}\|u\|_{L^{\infty}}+\f{\delta^2}{\nu}\|\mathcal{L}_{\lambda}u\|_{L^2}+\f{(y_2-y_1)\delta^3|\alpha|}{\nu^2}\f{\|\lambda\varphi\|_{L^{\infty}(B(y_1,\delta)\cup B(y_2,\delta))}}{(y_2-y_1+\delta)\delta^{\f12}},
\end{align*}
from which and (\ref{V1}), we infer that
\begin{align}
\nonumber\delta^3\|u'\|_{L^{\infty}(B(y_1,\delta))}^2\lesssim&\f{\delta^6(y_2-y_1)^2|\alpha|^2}{\nu^{4}}\delta\|u\|_{L^{\infty}}^2+\f{\delta^4}{\nu^2}\|\mathcal{L}_{\lambda}u\|_{L^2}^2
+\f{\delta^6(y_2-y_1+\delta)^2|\alpha|^2}{\nu^{4}}\mathcal{F}_1(u)\\
\lesssim&\f{\delta^6(y_2-y_1+\delta)^2|\alpha|^2}{\nu^{4}}\mathcal{F}_1(u)+\f{\delta^4}{\nu^2}\|\mathcal{L}_{\lambda}u\|_{L^2}^2.\nonumber
\end{align}
The estimate in $B(y_2,\delta)$ is similar.
\end{proof}

Now we are in a position to prove Proposition \ref{Prop: nonlocal2}.

\begin{proof}
First of all, we consider the case of $\lambda>1$. We get by integration by parts that
\begin{align*}
\big|\text{Im}\langle i\f{\alpha}{\nu}[(\sin y-\lambda)u+\lambda\varphi]-\nu\partial_y^2u, u\rangle\big|=\f{|\alpha|}{\nu}\Big(\int_0^{2\pi}(\lambda-\sin y)|u|^2dy+\lambda\|\varphi'\|_{L^2}^2+\lambda\tilde{\beta}^2\|\varphi\|_{L^2}^2\Big).
\end{align*}
which implies that
\begin{align}
\int_0^{2\pi}(\lambda-\sin y)|u|^2dy+\lambda\|\varphi'\|_{L^2}^2+\lambda\tilde{\beta}^2\|\varphi\|_{L^2}^2\leq\f{\nu}{|\alpha|}\|\mathcal{L}_{\lambda}u\|_{L^2}\|u\|_{L^2}.\label{L1.1}
\end{align}
Let $\delta\in(0,1]$. Then for any $y\in(\f{\pi}{2}+\delta,\f{5\pi}{2}-\delta)$, we have
\begin{align*}
1-\sin y&\geq \sin \f{\pi}{2}-\sin (\f{\pi}{2}+\delta)=\sin \f{\pi}{2}-\sin (\f{\pi}{2}-\delta)=2\sin \f{\delta}{2} \cos(\f{\pi}{2}-\f{\delta}{2})=2\sin^2\f{\delta}{2}\gtrsim \delta^2.
\end{align*}
Then it follows from (\ref{L1.1}) that
\begin{align}
\nonumber&\|u\|_{L^2}^2\leq \|u\|_{L^2(\f{\pi}{2}+\delta,\f{5\pi}{2}-\delta)}^2+2\delta\|u\|_{L^{\infty}}^2\lesssim\delta^{-2}\int_0^{2\pi}(\lambda-\sin y)|u|^2dy+\delta\|u\|_{L^{\infty}}^2\\
&\lesssim\f{\nu}{|\alpha|\delta^2}\|\mathcal{L}_{\lambda}u\|_{L^2}\|u\|_{L^2}+\nu^{-\f12}\delta{\|\mathcal{L}_{\lambda}u\|_{L^2}^{\f12}\|u\|_{L^2}^{\f32}}+\delta\|u\|_{L^2}^{2}.\label{LA1}
\end{align}
Here we used the following fact that
\begin{align}\label{eq:u-linfty}
\|u\|_{L^{\infty}}\leq \|u'\|_{L^2}^{\f12}\|u\|_{L^2}^{\f12}+\|u\|_{L^2}\leq \nu^{-\f14}\|\mathcal{L}_{\lambda}u\|_{L^2}^{\f14}\|u\|_{L^2}^{\f34}+\|u\|_{L^2},
\end{align}
due to $\nu\|u'\|_{L^2}^2=|\text{Re}\langle\mathcal{L}_{\lambda}u,u\rangle|$.
Taking $\delta=|\alpha|^{-\f14}\nu^{\f12}\ll1$, we infer that
\begin{align*}
&\|u\|_{L^2}^2\lesssim\frac{\|\mathcal{L}_{\lambda}u\|_{L^2}\|u\|_{L^2}}{|\alpha|^{\f12}}+\f{\|\mathcal{L}_{\lambda}u\|_{L^2}^{\f12}\|u\|_{L^2}^{\f32}}{|\alpha|^{\f14}},
\end{align*}
which implies that
\begin{align}
\|\mathcal{L}_{\lambda}u\|_{L^2}\gtrsim|\alpha|^{\f12}\|u\|_{L^2}.
\end{align}

Next we handle the case of  $\lambda\in[0,1]$. Then there exist $0\leq y_1\leq\f{\pi}{2}\leq y_2\leq\pi$ so that $\lambda=\sin y_1=\sin y_2$.\smallskip

\no$\mathbf{Case\ 1.}$ $|\alpha|^{-\f14}\nu^{\f12}\geq\f{y_2-y_1}{4}$.\smallskip

Let $1\gg\delta=|\alpha|^{-\f14}\nu^{\f12}\geq\f{y_2-y_1}{4}$. Using (\ref{Key1.1}) and $\| u\|_{L^2(y_1,y_2)}^2\lesssim\delta\|u\|_{L^{\infty}}^2\leq \mathcal{E}_1(u)$, we deduce that  $\| u\|_{L^2}^2\lesssim\mathcal{E}_1(u)$. Now we estimate each term in $\mathcal{E}_1(u)$.

Thanks to (\ref{K1.3}), we have
\begin{align*}
\f{\|\lambda \varphi_2\|_{L^{\infty}}^2}{(y_2-y_1+\delta)^2\delta}\lesssim& \f{\nu^2\|\mathcal{L}_{\lambda}u\|_{L^2}^2}{|\alpha|^2\delta^2(y_2-y_1+\delta)^2}+\f{\nu^3\|\mathcal{L}_{\lambda}u\|_{L^2}\|u\|_{L^2}}{|\alpha|^2\delta^4(y_2-y_1+\delta)^2}+\delta\|u\|_{L^{\infty}}^2\\
\lesssim& \f{\nu^2\|\mathcal{L}_{\lambda}u\|_{L^2}^2}{|\alpha|^2\delta^4}+\f{\nu^3\|\mathcal{L}_{\lambda}u\|_{L^2}\|u\|_{L^2}}{|\alpha|^2\delta^6}+\delta\|u\|_{L^{\infty}}^2,
\end{align*}
from which and \eqref{K2.3}, we infer that
\begin{align*}
\delta^3\|u'\|_{L^{\infty}(B(y_1,\delta)\cup B(y_2,\delta))}^2\lesssim&\f{\delta^8|\alpha|^2}{\nu^4}\Big(\f{\nu^2\|\mathcal{L}_{\lambda}u\|_{L^2}^2}{|\alpha|^2\delta^4}+\f{\nu^3\|\mathcal{L}_{\lambda}u\|_{L^2}\|u\|_{L^2}}{|\alpha|^2\delta^6}
+\delta\|u\|_{L^{\infty}}^2\Big)+\f{\delta^4}{\nu^2}\|\mathcal{L}_{\lambda}u\|_{L^2}^2\\
=&\f{2\|\mathcal{L}_{\lambda}u\|_{L^2}^2}{|\alpha|}+\f{\|\mathcal{L}_{\lambda}u\|_{L^2}\|u\|_{L^2}}{|\alpha|^{\f12}}
+\delta\|u\|_{L^{\infty}}^2,
\end{align*}
which in turn gives
\begin{align*}
\f{\nu^2}{|\alpha|(y_2-y_1+\delta)\delta}\|u'\bar{u}\|_{L^{\infty}(B(y_1,\delta)\cup B(y_2,\delta))}\lesssim&\f{\nu^2}{|\alpha|\delta^4}\delta^{\f32}\|u'\|_{L^{\infty}(B(y_1,\delta)\cup B(y_2,\delta))}\delta^{\f12}\|u\|_{L^{\infty}}\\\lesssim&\f{\nu^2}{|\alpha|\delta^4}\big(\delta^3\|u'\|_{L^{\infty}(B(y_1,\delta)\cup B(y_2,\delta))}^2+\delta\|u\|_{L^{\infty}}^2\big)\\ \nonumber
\lesssim&\f{\|\mathcal{L}_{\lambda}u\|_{L^2}^2}{|\alpha|}+\f{\|\mathcal{L}_{\lambda}u\|_{L^2}\|u\|_{L^2}}{|\alpha|^{\f12}}
+\delta\|u\|_{L^{\infty}}^2.
\end{align*}
For other terms, we have  by \eqref{eq:u-linfty} that
\begin{align*}
&\f{\nu^2}{|\alpha|\delta^{\f52}}\|u'\|_{L^2}\|u\|_{L^{\infty}}\leq\f{\nu^4}{|\alpha|^2\delta^{6}}\|u'\|_{L^2}^2+\delta\|u\|_{L^{\infty}}^2\leq
\f{\nu^3\|\mathcal{L}_{\lambda}u\|_{L^2}\|u\|_{L^2}}{|\alpha|^2\delta^{6}}+\delta\|u\|_{L^{\infty}}^2,\\
&\delta\|u\|_{L^{\infty}}^2\lesssim\f{\delta}{\nu^{\f12}}\|\mathcal{L}_{\lambda}u\|_{L^2}^{\f12}\|u\|_{L^2}^{\f32}+\delta\|u\|_{L^{2}}^2
\leq\f{\|\mathcal{L}_{\lambda}u\|_{L^2}^{\f12}\|u\|_{L^2}^{\f32}}{|\alpha|^{\f14}}+\delta\|u\|_{L^{2}}^2.
\end{align*}

Summing up, we conclude
\begin{align}
\|u\|_{L^2}^2\lesssim \f{\|\mathcal{L}_{\lambda}u\|_{L^2}^2}{|\alpha|}+\f{\|\mathcal{L}_{\lambda}u\|_{L^2}\|u\|_{L^2}}{|\alpha|^{\f12}}
+\f{\|\mathcal{L}_{\lambda}u\|_{L^2}^{\f12}\|u\|_{L^2}^{\f32}}{|\alpha|^{\f14}}+\delta\|u\|_{L^{2}}^2.
\end{align}
Due to $\delta\ll1 $, this implies that
\begin{align}
\|\mathcal{L}_{\lambda}u\|_{L^2}\gtrsim|\alpha|^{\f12}\|u\|_{L^2}.
\end{align}

\no $\mathbf{Case\ 2.}$ $|\alpha|^{-\f14}\nu^{\f12}\leq\f{y_2-y_1}{4}$.
\smallskip

Take $\delta^3(y_2-y_1)|\alpha|=\nu^{2}$, then $0<\delta\leq\f{y_2-y_1}{4}$. In this case, we have $\|u\|_{L^2}\le \mathcal{E}_2(u)$ by (\ref{Key1.2}). Let us estimate each term in $\mathcal{E}_2(u)$.

First of all, we have
\begin{align*}
\f{\nu}{|\alpha|\delta(y_2-y_1)}\|\mathcal{L}_{\lambda}u\|_{L^2}\|u\|_{L^2}\lesssim\f{\|\mathcal{L}_{\lambda}u\|_{L^2}\|u\|_{L^2}}{|\alpha|^{\f12}},
\end{align*}
and by \eqref{eq:u-linfty},
\begin{align*}
\mathcal{E}_{2}^1(u):=\delta\|u\|_{L^{\infty}}^2\leq\f{\delta}{\nu^{\f12}}\|\mathcal{L}_{\lambda}u\|_{L^2}^{\f12}\|u\|_{L^2}^{\f32}+\delta\|u\|_{L^2}^{2}\lesssim \f{\|\mathcal{L}_{\lambda}u\|_{L^2}^{\f12}\|u\|_{L^2}^{\f32}}{|\alpha|^{\f14}}+\delta\|u\|_{L^2}^{2},
\end{align*}
and
\begin{align*}
\f{\nu^2\|u'\|_{L^2}\|u\|_{L^{\infty}}}{|\alpha|\delta^{\f32}(y_2-y_1)}&=\delta\|u'\|_{L^2}\delta^{\f12}\|u\|_{L^{\infty}}\leq\delta^2 \|u'\|_{L^2}^2+\delta \|u\|_{L^{\infty}}^2\\
&\leq\f{\delta^2}{\nu}\|\mathcal{L}_{\lambda}u\|_{L^2}\|u\|_{L^2}+\delta \|u\|_{L^{\infty}}^2\\
&\lesssim \f{\|\mathcal{L}_{\lambda}u\|_{L^2}\|u\|_{L^2}}{|\alpha|^{\f12}}+\f{\|\mathcal{L}_{\lambda}u\|_{L^2}^{\f12}\|u\|_{L^2}^{\f32}}{|\alpha|^{\f14}}+\delta\|u\|_{L^2}^{2}.
\end{align*}
Thanks to (\ref{K1.3}), we have
\begin{align*}
\f{\|\lambda \varphi_2\|_{L^{\infty}}^2}{(y_2-y_1)^2\delta} \lesssim& \f{\nu^2\|\mathcal{L}_{\lambda}u\|_{L^2}^2}{|\alpha|^2\delta^2(y_2-y_1)^2}+\f{\nu^3\|\mathcal{L}_{\lambda}u\|_{L^2}\|u\|_{L^2}}{|\alpha|^2\delta^4(y_2-y_1)^2}+\delta\|u\|_{L^{\infty}}^2\\
\lesssim &\f{\|\mathcal{L}_{\lambda}u\|_{L^2}^2}{|\alpha|}+\f{\|\mathcal{L}_{\lambda}u\|_{L^2}\|u\|_{L^2}}{|\alpha|^{\f12}}+\f{\|\mathcal{L}_{\lambda}u\|_{L^2}^{\f12}\|u\|_{L^2}^{\f32}}{|\alpha|^{\f14}}+\delta\|u\|_{L^2}^{2}.
\end{align*}
Thanks to (\ref{K2.3}), we have
\begin{align*}
\nonumber\mathcal{E}_{2}^2(u):=&\delta^3\|u'\|_{L^{\infty}(B(y_1,\delta)\cup B(y_2,\delta))}^2\\
\nonumber\lesssim& \f{\delta^6(y_2-y_1)^2|\alpha|^2}{\nu^{4}}\mathcal{F}_1(u)+\f{\delta^4}{\nu^2}\|\mathcal{L}_{\lambda}u\|_{L^2}^2\\
\lesssim& \f{\|\mathcal{L}_{\lambda}u\|_{L^2}^2}{|\alpha|}+\f{\|\mathcal{L}_{\lambda}u\|_{L^2}\|u\|_{L^2}}{|\alpha|^{\f12}}+\f{\|\mathcal{L}_{\lambda}u\|_{L^2}^{\f12}\|u\|_{L^2}^{\f32}}{|\alpha|^{\f14}}+\delta\|u\|_{L^2}^{2}.
\end{align*}
Then we have
\begin{align*}
\f{\nu^2\|u'\bar{u}\|_{L^{\infty}(B(y_1,\delta)\cup B(y_2,\delta))}}{|\alpha|\delta(y_2-y_1)}&\leq\f{\nu^2\delta^{\f32}\|u'\|_{L^{\infty}(B(y_1,\delta)\cup B(y_2,\delta))}\delta^{\f12}\|u\|_{L^{\infty}}}{|\alpha|\delta^3(y_2-y_1)}\\
&\leq\delta^3\|u'\|_{L^{\infty}(B(y_1,\delta)\cup B(y_2,\delta))}^2+\delta\|u\|_{L^{\infty}}^2=\mathcal{E}_{2}^1(u)+\mathcal{E}_{2}^2(u).
\end{align*}
For the other terms, we have
\begin{align*}
&\f{\nu^2}{|\alpha|^2}\f{1}{(y_2-y_1)^3\delta}\|\mathcal{L}_{\lambda}u\|_{L^2}^2\lesssim\f{\|\mathcal{L}_{\lambda}u\|_{L^2}^2}{|\alpha|},\\
&\f{\nu^4}{|\alpha|^2}\f{\|u'\|_{L^{\infty}(B(y_1,\delta)\cup B(y_2,\delta))}^2}{(y_2-y_1)^3\delta^2}\leq \mathcal{E}_{2}^2(u),\\
&\f{\nu^4}{|\alpha|^2}\f{\|u'\|_{L^2}^2}{(y_2-y_1)^3\delta^3}\leq\f{\|\mathcal{L}_{\lambda}u\|_{L^2}\|u\|_{L^2}}{|\alpha|^{\f12}}.
\end{align*}

Summing up, we conclude
\begin{align*}
\|u\|_{L^2}^2\lesssim \f{\|\mathcal{L}_{\lambda}u\|_{L^2}^2}{|\alpha|}+\f{\|\mathcal{L}_{\lambda}u\|_{L^2}\|u\|_{L^2}}{|\alpha|^{\f12}}+\f{\|\mathcal{L}_{\lambda}u\|_{L^2}^{\f12}\|u\|_{L^2}^{\f32}}{|\alpha|^{\f14}}+\delta\|u\|_{L^2}^{2},
\end{align*}
which implies that
\begin{align}
\|\mathcal{L}_{\lambda}u\|_{L^2}\gtrsim|\alpha|^{\f12}\|u\|_{L^2}.\nonumber
\end{align}

As $\tau_{\pi}\mathcal{L}_{\lambda}=\mathcal{L}_{-\lambda}\tau_{\pi}$, the same result holds for $\la\le 0$. This completes the proof of the proposition.
\end{proof}

\section{Pseudospectral bound and semigroup bound}

\subsection{Pseudospectral bound}

Recall that an operator $H$ in a Hilbert space $X$ is accretive if ${\rm Re} \langle Hf,f\rangle\geq 0$ for all $f\in {D}(H),$ or equivalently $\|(\la+H)f\|\geq \la\|f\|$ for all $f\in {D}(H)$ and all $\la>0$ \cite{Pazy}. The operator $ H$ is called m-accretive if in addition any $\la<0$ belongs to the resolvent set of $ H$\cite{Kato}. We define
\beno
\Psi(H)=\inf\big\{\|(H-i\la)f\|;f\in {D}(H),\ \la\in \mathbb{R},\  \|f\|=1\big\}.
\eeno

Let $ \beta^2=\f{k_1^2+k_3^2}{k_f^2}(\beta>0)$. We introduce two operators
\beno
&&\mathcal{L}'_{k_1,k_3}=-\nu k_f^2\partial_y^2+\frac{ik_1}{k_f^2}\f{\gamma}{\nu }\sin y\big(1-(\beta^2-\partial_y^2)^{-1}\big),\\
&&\mathcal{H}'_{k_1,k_3}=-\nu k_f^2\partial_y^2+\frac{ik_1}{k_f^2}\f{\gamma}{\nu }\sin y.
\eeno
Since $\mathcal{L}'_{k_1,k_3} $ and $\mathcal{H}'_{k_1,k_3} $ are relatively compact perturbations of the operator $ -\nu k_f^2\partial_y^2,$ which itself has compact
resolvent, it is clear that the operators $\mathcal{L}'_{k_1,k_3} $ and $\mathcal{H}'_{k_1,k_3} $ have compact resolvent and only point spectrum. Thus, we only need to check $\mathcal{L}'_{k_1,k_3} $ and $\mathcal{H}'_{k_1,k_3} $ are accretive, which will imply that they are also m-accretive.

For the operator $\mathcal{H}'_{k_1,k_3}$, we take $X=L^2(\mathbb{T}_{2\pi})$ and $D(\mathcal{H}'_{k_1,k_3})=H^2(\mathbb{T}_{2\pi})$. Then ${\rm Re}\langle\mathcal{H}'_{k_1,k_3}f,f\rangle=\nu k_f^2\|\partial_y f\|_{L^2}^2\geq 0$, thus $\mathcal{H}'_{k_1,k_3}$ is m-accretive.
For the operator $\mathcal{L}'_{k_1,k_3}$, we take $X=L^2(\mathbb{T}_{2\pi})$ with the norm $\|f\|_*=\langle f,f-(\beta^2-\partial_y^2)^{-1}f\rangle^{1/2}$ and the inner product $\langle f,g\rangle_*=\langle f,g-(\beta^2-\partial_y^2)^{-1}g\rangle,$ and $D(\mathcal{L}'_{k_1,k_3})=H^2(\mathbb{T}_{2\pi})$.
If $k_1^2+k_3^2>k_f^2,$ $ \beta^2>1,$ then this norm is equivalent to the $ L^2$ norm: $(1-\beta^{-2})\|f\|_{L^2}^2\leq \|f\|_*^2\leq \|f\|_{L^2}^2,$ which can be easily proved by using Fourier transform. It is easy to see that
\beno
{\rm Re}\big\langle\mathcal{L}'_{k_1,k_3}f,f\big\rangle_*=-\nu k_f^2\langle \partial_y^2f,f\rangle_*
=\nu k_f^2\langle\partial_yf,\partial_yf-\partial_y(\beta^2-\partial_y^2)^{-1}f\rangle=\nu k_f^2\|\partial_y f\|_{*}^2\geq 0.
\eeno
Thus, $\mathcal{L}'_{k_1,k_3}$  is m-accretive.

\begin{lemma}\label{H1}
If $\gamma\gg \nu^2$, then
$\Psi(\mathcal{H}'_{k_1,k_3})\geq c|k_1\gamma|^{\frac{1}{2}}.$
\end{lemma}

\begin{proof}
Let $\al=k_1\gamma/k_f^4$. Then we have
\beno
(\mathcal{H}'_{k_1,k_3}-i\la)w=k_f^2\left(i\f{\alpha}{\nu}(\sin y-\frac{\nu\lambda}{\al})w-\nu\partial_y^2w\right).
\eeno
Then it follows from Proposition \ref{prop:Res-toy} that
\begin{align*}
\|(\mathcal{H}'_{k_1,k_3}-i\la)w\|_{L^2}&=k_f^2\big\|i\f{\alpha}{\nu}(\sin y-\frac{\nu\lambda}{k_f^2\al})w-\nu\partial_y^2w\big\|_{L^2}\\
&\geq ck_f^2|\al|^{\frac{1}{2}}\|w\|_{L^2}=c|k_1\gamma|^{\frac{1}{2}}\|w\|_{L^2}.
\end{align*}
\end{proof}

We define  $\Psi(\mathcal{L}'_{k_1,k_3})=\inf\big\{\|(\mathcal{L}'_{k_1,k_3}-i\la)f\|_*;f\in {D}(\mathcal{L}'_{k_1,k_3}),\ \la\in\mathbb{R},\  \|f\|_*=1\big\}.$

\begin{lemma}\label{L1}If $\gamma\gg \nu^2$ and $ \beta^2=(k_1^2+k_3^2)/k_f^2>1,$ then
$\Psi(\mathcal{L}'_{k_1,k_3})\geq c|k_1\gamma|^{\frac{1}{2}}(1-\beta^{-2}).$
\end{lemma}

\begin{proof}For $\la\in\R,$ let $\al=k_1\gamma/k_f^4,\ \la_1=\frac{\nu k_f^2\lambda}{k_1\gamma}$. Then we have
\beno
(\mathcal{L}'_{k_1,k_3}-i\la)w=k_f^2\left(i\f{\alpha}{\nu}((\sin y-\la_1)w+\sin y\varphi)-\nu\partial_y^2w\right),
\eeno
where $\varphi=(\beta^2-\partial_y^2)^{-1}w$. Let  $u=w+\varphi,$ then we have $(\partial_y^2-\tilde{\beta}^2)\varphi=u$ for $\tilde{\beta}^2=\beta^2-1$ and $\|\varphi''\|_{L^2}\leq\|u\|_{L^2}$.

Thanks to Proposition \ref{Prop: nonlocal2}, we deduce that
\begin{align*}
\big\|i\f{\alpha}{\nu}[(\sin y-\lambda_1)w+\sin y\varphi]-\nu\partial_y^2w\big \|_{L^2}&\geq \big\|i\f{\alpha}{\nu}[(\sin y-\lambda_1)u+\lambda_1\varphi]-\nu\partial_y^2u \big\|_{L^2}-\nu\|\varphi''\|_{L^2}\\
&\geq(c|\alpha|^{\f12}-\nu)\|u\|_{L^2}\geq c|\alpha|^{\f12}\|u\|_{L^2}.
\end{align*}
Using the fact that $(1-\beta^{-2})\|w\|_{L^2}^2\leq\|w\|_*^2=\langle w,u\rangle\leq \|w\|_{L^2}\|u\|_{L^2},$ we have
\beno
(1-\beta^{-2})\|w\|_{L^2}\leq \|u\|_{L^2},\quad (1-\beta^{-2})\|w\|_*^2\leq (1-\beta^{-2})\|w\|_{L^2}\|u\|_{L^2}\leq \|u\|_{L^2}^2.
\eeno
Then we conclude
\begin{align*}
\|(\mathcal{L}'_{k_1,k_3}-i\la)w\|_{*}\geq&(1-\beta^{-2})^{\frac{1}{2}}\|(\mathcal{L}'_{k_1,k_3}-i\la)w\|_{L^2}\\=&
(1-\beta^{-2})^{\frac{1}{2}}k_f^2\big\|i\f{\alpha}{\nu}[(\sin y-\lambda_1)w+\sin y\varphi]-\nu\partial_y^2w\big\|_{L^2}\\ \geq & c(1-\beta^{-2})^{\frac{1}{2}}k_f^2|\al|^{\frac{1}{2}}\|u\|_{L^2}\geq c(1-\beta^{-2})|k_1\gamma|^{\frac{1}{2}}\|w\|_{*}.
\end{align*}
\end{proof}

\subsection{Semigroup bounds}

To obtain the semigroup bound from the pseudospectral bound, we use the following Gearhart-Pr\"{u}ss type lemma with sharp bound from \cite{Wei}.

\begin{lemma}\label{Lem:GP}
Let $ H$ be a m-accretive operator in a Hilbert space $X$. Then
$\|e^{-tH}\|\leq e^{-t\Psi+{\pi}/{2}}$ for any $t\geq 0$.
\end{lemma}

\begin{proposition}\label{Prop: H1}
Given $k_f\in(0,1]$ and $\gamma>0$, if $|\gamma|\gg \nu^2$, then there exist constants $C,c>0$  independent of $\gamma, \nu$ such that
\begin{align}
\label{local shear part}\|e^{-t\mathcal{H}}g_{\neq}\|_{L^2}\leq Ce^{-c|\gamma|^{\f12}t-\nu t}\|g_{\neq}\|_{L^2}.
\end{align}
\end{proposition}

\begin{proof}We write\begin{align*}
&g(x,y,z)=\sum\limits_{(k_1,k_3)\in\mathbb{Z}^2}\hat{\varphi}(k_1,k_fy,k_3)e^{i(k_1,k_3)\cdot(x,z)},\ \ \|g_{\neq}\|_{L^2}^2=\frac{4\pi^2}{k_f}\sum\limits_{k_1\neq 0}\|\hat{\varphi}(k_1,\cdot,k_3)\|_{L^2}^2.
\end{align*}
Then we have
\begin{align*}
&e^{-t\mathcal{H}}g_{\neq}(x,y,z)=\sum\limits_{k_1\neq 0}e^{-\nu(k_1^2+k_3^2)t}e^{-t\mathcal{H}'_{k_1,k_3}}\hat{\varphi}(k_1,\cdot,k_3)(k_fy)e^{i(k_1,k_3)\cdot(x,z)}.
\end{align*}
By Lemma \ref{Lem:GP},  we have
\ben\label{eq:Hk-est}
\|e^{-t\mathcal{H}'_{k_1,k_3}}\|_{L^2\to L^2}\leq e^{-ct|k_1\gamma|^{\frac{1}{2}}+{\pi}/{2}},
\een
from which, it follows that
\begin{align*}
\|e^{-t\mathcal{H}}g_{\neq}\|_{L^2}^2=&\frac{4\pi^2}{k_f}\sum\limits_{k_1\neq 0}\|e^{-\nu(k_1^2+k_3^2)t}e^{-t\mathcal{H}'_{k_1,k_3}}\hat{\varphi}(k_1,\cdot,k_3)\|_{L^2}^2\\ \leq &\frac{4\pi^2}{k_f}\sum\limits_{k_1\neq 0}e^{-2\nu(k_1^2+k_3^2)t}e^{-2ct|k_1\gamma|^{\frac{1}{2}}+{\pi}}\|\hat{\varphi}(k_1,\cdot,k_3)\|_{L^2}^2\\ \leq &\frac{4\pi^2}{k_f}Ce^{-2\nu t}\sum\limits_{k_1\neq 0}e^{-2ct|\gamma|^{\frac{1}{2}}}\|\hat{\varphi}(k_1,\cdot,k_3)\|_{L^2}^2=Ce^{-2c|\gamma|^{\f12}t-2\nu t}\|g_{\neq}\|_{L^2}^2.
\end{align*}
This completes the proof.
\end{proof}

\begin{proposition}\label{Prop: L1}
Given $k_f\in(0,1)$ and $\gamma>0$, if $|\gamma|\gg \nu^2$,  there exists constants $C,c>0$  independent of $\gamma, \nu$  such that
\begin{align}
\label{shear part}\|e^{-t\mathcal{L}}g_{\neq}\|_{L^2}\leq Ce^{-c|\gamma|^{\f12}t-\nu t}\|g_{\neq}\|_{L^2}.
\end{align}
Moreover, for any $c'\in(0,c)$, there exists a constant $C(c,c')$ so that
\begin{align}
\label{shear part energy}\nu \int_0^{+\infty}e^{2c'|\gamma|^{\f12}s}\|(\nabla e^{-t\mathcal{L}}g)_{\neq}(s)\|_{L^2}^2ds\leq C(c,c') \|g_{\neq}\|_{L^2}^2.
\end{align}
\end{proposition}

The proposition follows from the following two lemmas.

\begin{lemma}\label{f3}
Given $k_f\in(0,1]$ and $\gamma>0$, let $f(t,y)$ solve
\begin{align}
\label{3.2}&\partial_tf+\big(\nu(k_1^2+k_3^2)+\mathcal{L}'_{k_1,k_3}\big)f=0,\quad f(0)=f_0.
\end{align}
If $|\gamma|\gg \nu^2$, then there exist constants $C,c>0$, such that if $k_1,k_3\in\mathbb{Z},\ k_1\neq 0,\ \beta^2={(k_1^2+k_3^2)}/k_f^2>1$, then
\begin{align}
\label{shear part1}\|f(t)\|_{L^2}\leq Ce^{-c|k_1\gamma|^{\f12}t-\nu(k_1^2+k_2^2) t}\|f_0\|_{L^2}.
\end{align}
\end{lemma}

\begin{proof}
As $f(t)=e^{-\nu(k_1^2+k_3^2)t}e^{-t\mathcal{L}'_{k_1,k_3}}f_0$,  we have
\ben
\|f(t)\|_{L^2}\le e^{-\nu(k_1^2+k_3^2)t}\|e^{-t\mathcal{L}'_{k_1,k_3}}f_0\|_{L^2}.\label{eq:f-L2}
\een
By Lemma \ref{L1} and Lemma \ref{Lem:GP},  we have
\beno
\|e^{-t\mathcal{L}'_{k_1,k_3}}f_0\|_{*}\leq e^{-ct(1-\beta^{-2})|k_1\gamma|^{\frac{1}{2}}+{\pi}/{2}}\|f_0\|_{*}.
\eeno
Notice that $\beta^2\geq c_0(k_f)>1$ with $c_0=2$ for $k_f=1$ and $c_0=1/k_f^2$ for $k_f<1.$ Using the fact that $(1-\beta^{-2})\|f\|_{L^2}^2\leq \|f\|_*^2\leq \|f\|_{L^2}^2, $ we have
\beno
\|e^{-t\mathcal{L}'_{k_1,k_3}}f_0\|_{L^2}\leq Ce^{-ct|k_1\gamma|^{\frac{1}{2}}}\|f_0\|_{L^2}
\eeno
with $C>0,c>0$ depending only on $k_f$, which along with \eqref{eq:f-L2}
gives \eqref{shear part1}.
\end{proof}

\begin{lemma}
Given $k_f\in(0,1]$ and $\gamma>0$, let $f(t,y)$ solve \eqref{3.2}.
If $|\gamma|\gg \nu^2$, then there exist constants $C,c>0$, such that if $k_1,k_3\in\mathbb{Z},\ k_1\neq 0,\ \beta^2={(k_1^2+k_3^2)}/k_f^2>1$,
then
\begin{align}
&\nu k_f^2\int_0^t\big(\beta^2\|f(s)\|_{L^2}^2+\|\pa_{y}f(s)\|_{L^2}^2\big)ds\leq C \|f_0\|_{L^2}^2,\label{est of L(k_1,k_3)2}\\
&\nu k_f^2\int_0^te^{2as}\big(\beta^2\|f(s)\|_{L^2}^2+\|\pa_{y}f(s)\|_{L^2}^2\big)ds\leq C(1+at)\|f_0\|_{L^2}^2,\label{est of L(k_1,k_3)3}
\end{align}
where $a=\nu(k_1^2+k_3^2)+c|k_1\gamma|^{\frac{1}{2}}$ with $c$ as in  (\ref{shear part1}). Moreover, for any $c'\in(0,c)$, there exists a constant $C(c,c')$, so that
\begin{align}
\nu k_f^2\int_0^{+\infty}e^{2c'|k_1\gamma|^{\f12}s}\big(\al^2\|f(s)\|_{L^2}^2+\|\pa_{y}f(s)\|_{L^2}^2)ds\leq C(c,c') \|f_0\|_{L^2}^2.\label{est of L(k_1,k_3)4}
\end{align}
\end{lemma}

\begin{proof}
We define $-(\partial_y^2-\alpha^2)\varphi=f$. Then we have
\begin{align*}
&\big\langle\partial_tf+(\nu(k_1^2+k_3^2)+\mathcal{L}'_{k_1,k_3})f,f\big\rangle\\
&=\f{1}{2}\f{d}{dt}\|f(t)\|_{L^2}^2+\nu k_f^2\beta^2\|f\|_{L^2}^2+\nu k_f^2\|\partial_yf\|_{L^2}^2+\text{Re}\f{ik_1}{k_f^2}\f{\gamma}{\nu}\langle\sin y(f-\varphi),f\rangle=0,
\end{align*}
and
\begin{align*}
&\big\langle\partial_tf+(\nu(k_1^2+k_3^2)+\mathcal{L}'_{k_1,k_3})f,\varphi\big\rangle\\
&=\f{1}{2}\f{d}{dt}\big(\beta^2\|\varphi(t)\|_{L^2}^2+\|\varphi'(t)\|_{L^2}^2\big)+\nu k_f^2\|f(t)\|_{L^2}^2+\text{Re}\f{ik_1}{k_f^2}\f{\gamma}{\nu}\langle\sin y(f-\varphi),\varphi\rangle=0,
\end{align*}
which imply that
\begin{align}
\label{dissp1}\f{1}{2}\f{d}{dt}\big(\|f(t)\|_{L^2}^2-\alpha^2\|\varphi(t)\|_{L^2}^2-\|\varphi'(t)\|_{L^2}^2\big)+\nu k_f^2\big((\alpha^2-1)\|f(t)\|_{L^2}^2+\|\partial_yf\|_{L^2}^2\big)=0.
\end{align}
Thanks to  $\|f\|_{L^2}^2=\beta^4\|\varphi\|_{L^2}^2+2\beta^2\|\varphi'\|_{L^2}^2+\|\varphi''\|_{L^2}^2$, we obtain
\begin{align*}
&\nu k_f^2\int_0^t\big((\beta^2-1)\|f(s)\|_{L^2}^2+\|\pa_{y}f(s)\|_{L^2}^2\big)ds\\
&=\f{1}{2}\Big[(\|f(0)\|_{L^2}^2-\beta^2\|\varphi(0)\|_{L^2}^2-\|\varphi'(0)\|_{L^2}^2)-(\|f(t)\|_{L^2}^2-\beta^2\|\varphi(t)\|_{L^2}^2-\|\varphi'(t)\|_{L^2}^2)\Big]\\
&\leq C \|f_0\|_{L^2}^2,
\end{align*}
which gives (\ref{est of L(k_1,k_3)2}) due to $\beta^2-1\sim\beta^2$.
Moreover, by \eqref{shear part1},
\begin{align*}
&\nu k_f^2\int_0^te^{2as}\big((\beta^2-1)\|f(s)\|_{L^2}^2+\|\pa_{y}f(s)\|_{L^2}^2\big)ds\\
&\leq C(\|f_0\|_{L^2}^2-e^{2at}\|f(t)\|_{L^2}^2)+\int_0^t2a\|f(s)\|_{L^2}^2e^{2as}ds\\
&\lesssim (1+at)\|f_0\|_{L^2}^2,
\end{align*}
which gives (\ref{est of L(k_1,k_3)3}). It follows from (\ref{dissp1}) that $\forall c'\in(0,c)$
\begin{align*}
&\nu k_f^2\int_0^te^{2c'\sqrt{|k_1\gamma|}s}((\al^2-1)\|f(s)\|_{L^2}^2+\|\pa_{y}f(s)\|_{L^2}^2)ds\\
&=\f{1}{2}\big[(\|f(0)\|_{L^2}^2-\alpha^2\|\varphi(0)\|_{L^2}^2-\|\varphi'(0)\|_{L^2}^2)-e^{2c'\sqrt{|k_1\gamma|}t}(\|f(t)\|_{L^2}^2-\alpha^2\|\varphi(t)\|_{L^2}^2-\|\varphi'(t)\|_{L^2}^2)\big]\\
&\quad+2c'\sqrt{|k_1\gamma|}\int_0^t\|f(s)\|_{L^2}^2e^{2c'\sqrt{|k_1\gamma|}s}ds\\
&\lesssim\|f_0\|_{L^2}^2+2c'\sqrt{|k_1\gamma|}\int_0^te^{2(c'-c)\sqrt{|k_1\gamma|}s}ds\|f(0)\|_{L^2}^2\lesssim\|f_0\|_{L^2}^2,
\end{align*}
which gives (\ref{est of L(k_1,k_3)4}).
\end{proof}

\section{Enhanced dissipation of the linearized Navier-Stokes equations}

In this section, we prove Theorem \ref{thm:decay-L}.
We will need to use the wave operator method introduced in \cite{LWZ}.

\subsection{Wave operator}

In the appendix, we will construct a wave operator $\mathbb{D}_1=\mathbb{D}_1^{(\al)}$ with the following properties.

\begin{proposition}\label{Prop: oper_D}
Let $\al>1$. It holds that
\begin{align}
&\mathbb{D}_1\big(\cos y(1+(\partial_y^2-\alpha^2)^{-1})\omega\big)=\cos y\mathbb{D}_1(\omega)+(\partial_y^2-\alpha^2)^{-1}\omega,\label{oper,cosyD1}\\
&\|\sin y\mathbb{D}_1(\partial_y^2\omega)-\partial_y^2(\sin y\mathbb{D}_1(\omega))\|_{L^2}\leq C\big(|\alpha|\|\omega\|_{L^2}+\|\partial_y\omega\|_{L^2}\big).\label{est of D1}
\end{align}
If $\om\in H^2( \mathbb{T})$, then $\bbD_1(\om)\in  H^2( \mathbb{T})$ and
\begin{align}
&\|\sin y\bbD_1(\om)\|_{L^2}\leq C\al^{-1}\|\om\|_{L^2},\label{D1.1}\\
&\|\pa_{y}(\sin y\bbD_1(\om))\|_{L^2}\leq C\big(\|\om\|_{L^2}+\al^{-1}\|\pa_{y}\om\|_{L^2}\big)\label{D1.2}.
\end{align}
\end{proposition}

Let us make a translation $ \mathbb{D}_2=\mathbb{D}_2^{(\al)}=\tau_{-\pi/2}\circ\mathbb{D}_1\circ\tau_{\pi/2},$ where $\tau_{a}f(y)=f(a+y)$. Then
we find that
\begin{align}
&\mathbb{D}_2\big(\sin y(1+(\partial_y^2-\alpha^2)^{-1})\omega\big)=\sin y\mathbb{D}_2(\omega)-(\partial_y^2-\alpha^2)^{-1}\omega,\label{oper,sinyD2}\\
&\|\cos y\mathbb{D}_2(\partial_y^2\omega)-\partial_y^2(\cos y\mathbb{D}_2(\omega))\|_{L^2}\leq C\big(|\alpha|\|\omega\|_{L^2}+\|\partial_y\omega\|_{L^2}\big),\label{est of D2}\\
&\|\cos y\mathbb{D}_2(\omega)\|_{L^2}\leq C\alpha^{-1}\|\omega\|_{L^2},\label{D2.1}\\
\label{D6}
&\|\partial_{y}(\cos y\mathbb{D}_2(\omega))\|_{L^2}\leq C\big(\|\omega\|_{L^2}+\alpha^{-1}\|\partial_{y}\omega\|_{L^2}\big).
\end{align}

\subsection{Enhanced dissipation}
Let $ \Delta v_2=\varphi$, and taking the Fourier transform in $x,z$ and changing $y$ to $ k_f y$, we have
\begin{align*}
&\varphi(t,x,y,z)=\sum\limits_{(k_1,k_3)\in\mathbb{Z}^2}\hat{\varphi}(t,k_1,k_fy,k_3)e^{i(k_1,k_3)\cdot(x,z)},\\
&\omega_2(t,x,y,z)=\sum\limits_{(k_1,k_3)\in\mathbb{Z}^2}\hat{\omega}_2(t,k_1,k_fy,k_3)e^{i(k_1,k_3)\cdot(x,z)}.
\end{align*}
Let $f(t,y)=\hat\varphi(t,k_1,y,k_3)$ and $g(t,y)=\hat{\omega}_2(t,k_1,y,k_3)$. Then the linearized Navier-Stokes equations \eqref{eq:NS-L} become
\begin{align}\label{eq:NS-LF}
\left\{
\begin{array}{l}
\partial_tf+(\nu(k_1^2+k_3^2)+\mathcal{L}'_{k_1,k_3})f=0,\\
\partial_tg+(\nu(k_1^2+k_3^2)+\mathcal{H}'_{k_1,k_3})g=\frac{ik_3}{k_f^3}\f{\gamma}{\nu}\cos y (\alpha^2-\partial_y^2)^{-1}f,\\
f(0)=f_0,\quad g(0)=g_0.
\end{array}\right.
\end{align}
where $ \alpha^2=\f{k_1^2+k_3^2}{k_f^2}$ and
\beno
&&\mathcal{L}'_{k_1,k_3}=-\nu k_f^2\partial_y^2+\frac{ik_1}{k_f^2}\f{\gamma}{\nu }\sin y(1-(\alpha^2-\partial_y^2)^{-1}),\\
&&\mathcal{H}'_{k_1,k_3}=-\nu k_f^2\partial_y^2+\frac{ik_1}{k_f^2}\f{\gamma}{\nu }\sin y.
\eeno

In what follows, we assume $\gamma\gg \nu^2$.

\subsubsection{$\mathbf{Case\ of\ \al>1}$}

\begin{proposition}\label{f4}
If $k_1\neq 0$ and $\al>1$, then there exist constants $C,c>0$ so that
\begin{align}
\|g(t)\|_{L^2} \leq Ce^{-at}\|g_0\|_{L^2}+\f{Ce^{-at}(1+at)\|f_0\|_{L^2}}{|k_1|},\label{eq:g-L2}
\end{align}
where $a=\nu(k_1^2+k_3^2)+c|k_1\gamma|^{\frac{1}{2}}$.\end{proposition}
\begin{proof}
By \eqref{oper,sinyD2}, we have
\begin{align*}
&\cos y\mathbb{D}_2\circ\mathcal{L}'_{k_1,k_3}-\mathcal{H}'_{k_1,k_3}\circ\cos y\mathbb{D}_2=-\nu k_f^2[\cos y\mathbb{D}_2,\partial_y^2]+\frac{ik_1}{k_f^3}\f{\gamma}{\nu }\cos y(\alpha^2-\partial_y^2)^{-1}.
\end{align*}
Let $g_1=g+\f{k_3}{k_1k_f}\cos y\mathbb{D}_2(f)$. Then we have
\begin{align}
\partial_tg_1+(\nu(k_1^2+k_3^2)+\mathcal{H}'_{k_1,k_3})g_1= \f{\nu k_fk_3}{k_1}[\cos y\mathbb{D}_2,\partial_y^2]f .
\end{align}

It follows from \eqref{eq:Hk-est} that
\begin{align*}
\|e^{-t(\nu(k_1^2+k_3^2)+\mathcal{H}_{k_1,k_3}^{'})}\|_{L^2\to L^2}\leq Ce^{-\nu(k_1^2+k_3^2)t-c|k_1\gamma|^{\frac{1}{2}}t},
\end{align*}
which along with \eqref{est of D2} and (\ref{est of L(k_1,k_3)3}) gives
\begin{align*}
\|g_1(t)\|_{L^2}\leq& Ce^{-\nu(k_1^2+k_3^2)t-c|k_1\gamma|^{\frac{1}{2}}t} \|g_1(0)\|_{L^2}\\
&+C\f{\nu |k_fk_3|}{|k_1|}\int_0^te^{-(\nu(k_1^2+k_3^2)+c|k_1\gamma|^{\frac{1}{2}})(t-s)}\big\|[\cos y\mathbb{D}_2,\partial_y^2]f(s)\big\|_{L^2}ds\\
\leq& Ce^{-at}\|g_1(0)\|_{L^2}+Ce^{-at}\f{\nu |k_f k_3|}{|k_1|}\int_0^te^{as}\big(\alpha\|f(s)\|_{L^2}+\|\partial_{y}f(s)\|_{L^2}\big)ds\\
\leq& Ce^{-at}\|g_1(0)\|_{L^2}+Ce^{-at}\f{\nu |k_f k_3|}{|k_1|}\sqrt{t}\Big(\int_0^te^{2as}\big(\alpha^2\|f(s)\|_{L^2}^2+\|\partial_{y}f(s)\|_{L^2}^2\big)ds\Big)^{\f12}\\
\leq& Ce^{-at}\|g_1(0)\|_{L^2}+Ce^{-at} \big|\f{k_3}{k_1}\big| (\nu t)^{\f12}(1+at)^{\frac{1}{2}}\|f_0\|_{L^2}\\
 \leq &Ce^{-at}\|g_1(0)\|_{L^2}+\f{Ce^{-at}(1+at)\|f_0\|_{L^2}}{|k_1|},
\end{align*}
where by \eqref{D2.1}, we have
\begin{align*}
\|g_1(0)\|_{L^2}\leq& \|g_0\|_{L^2}+\big|\f{k_3}{k_f k_1}\big|\|\cos y\mathbb{D}_2f(0)\big\|_{L^2}\\
\leq& \|g_0\|_{L^2}+C\big|\f{k_3}{k_f k_1}\alpha^{-1}\big|\|f_0\|_{L^2}\leq \|g_0\|_{L^2}+C\f{\|f_0\|_{L^2}}{|k_1|}.
\end{align*}
Then by \eqref{shear part1} and \eqref{D2.1}, we obtain
\begin{align*}
\|g(t)\|_{L^2}\leq& \|g_1(t)\|_{L^2}+\big|\f{k_3}{k_f k_1}\big|\|\cos y\mathbb{D}_2f(t)\|_{L^2}\\
\leq& Ce^{-at}\|g_1(0)\|_{L^2}+\f{Ce^{-at}(1+at)\|f_0\|_{L^2}}{|k_1|}+C\big|\f{k_3}{k_f k_1}\alpha^{-1}\big|\|f(t)\|_{L^2}\\
 \leq& Ce^{-at}\Big(\|g_0\|_{L^2}+\f{\|f_0\|_{L^2}}{|k_1|}\Big)+\f{Ce^{-at}(1+at)\|f_0\|_{L^2}}{|k_1|}+\f{Ce^{-at}\|f_0\|_{L^2}}{|k_1|}\\
 \leq &Ce^{-at}\|g_0\|_{L^2}+\f{Ce^{-at}(1+at)\|f_0\|_{L^2}}{|k_1|},
\end{align*}
which gives \eqref{eq:g-L2}.
\end{proof}

Now Theorem \ref{thm:decay-L} follows from \eqref{eq:f-L2} and \eqref{eq:g-L2}. Moreover, it follows from \eqref{eq:g-L2}  that if $0<k_f<1$, then we have
\begin{align}
&\|\partial_x(\omega_2)(t)\|_{L^2}\leq Ce^{-c|\gamma|^{\f12}t-\nu t}\big(\|\partial_x(\omega_2)(0)\|_{L^2}+\|(\Delta v_2)_{\neq}(0)\|_{L^2}\big).
\end{align}

\subsubsection{$\mathbf{Case\ of\ \al=1}$}
In this case, $k_1^2+k_3^2=k_f^2=1$. Since we assume $ k_1\neq 0$, we must have
$|k_1|=1,k_3=0.$  Then the first equation of (\ref{eq:NS-LF}) becomes
\begin{align*}
\partial_t f+\nu(1-\partial_y^2)f+i\f{\gamma k_1}{\nu}\sin y f+i\f{\gamma k_1}{\nu}\sin y\varphi=0,
\end{align*}
where $(\partial_y^2-1)\varphi=f$. Let $\gamma k_1=\beta $, $u=f+\varphi$, then $f''=u''-u$. Thus, we obtain
\begin{align}\label{eq:f-new}
\partial_t f&=-\nu f+\nu f''-i\f{\beta}{\nu}\sin y u=-\nu f+\mathcal{L}_1u,
\end{align}
where $\mathcal{L}_1u=\nu u''-\nu u-i\f{\beta}{\nu}\sin y u=\nu f''-i\f{\beta}{\nu}\sin y u$.

\begin{lemma}
It holds that
\begin{align}
&\|\mathcal{L}^{-1}_1u\|_{L^2}\leq C|\beta|^{-\f23}\nu^{\f13}\|u \|_{L^2}\label{UpB2},\\
&\|\mathcal{L}^{-1}_1u\|_{L^1}\leq C|\beta|^{-\f56}\nu^{\f23}\|u \|_{L^2}\label{UpB1}.
\end{align}
Here $C$ is a constant independent of $\beta,\nu$.
\end{lemma}
\begin{proof}
It is easy to see that
\begin{align*}
\|u \|_{L^1((\delta,\pi-\delta)\cup(\pi+\delta,2\pi-\delta))}^2=&\big\|u\sin y\f{1}{\sin y} \big\|_{L^1((\delta,\pi-\delta)\cup(\pi+\delta,2\pi-\delta))}^2\\
\leq&\| u\sin y\|_{L^2}^2\big\|\f{1}{\sin y}\big\|_{L^2((\delta,\pi-\delta)\cup(\pi+\delta,2\pi-\delta))}^2\lesssim\f{\| u\sin y\|_{L^2}^2}{\delta},
\end{align*}
which gives
\begin{align}
\|u \|_{L^1}^2\lesssim\delta^2\|u \|_{L^{\infty}}^2+\f{\| u\sin y\|_{L^2}^2}{\delta}.\label{FL1}
\end{align}
Similarly, we have
\begin{align*}
\|u \|_{L^2((\delta,\pi-\delta)\cup(\pi+\delta,2\pi-\delta))}^2=&\big\|u\sin y\f{1}{\sin y}\big\|_{L^2((\delta,\pi-\delta)\cup(\pi+\delta,2\pi-\delta))}^2\\
\leq&\|u\sin y\|_{L^2}^2\big\|\f{1}{\sin y}\big\|_{L^{\infty}((\delta,\pi-\delta)\cup(\pi+\delta,2\pi-\delta))}^2\lesssim\f{\|u\sin y\|_{L^2}^2}{\delta^2}.
\end{align*}
which along with (\ref{FL1}) gives
\begin{align}
\f{\|u \|_{L^1}^2}{\delta}+\|u \|_{L^2}^2\lesssim\delta\|u \|_{L^{\infty}}^2+\f{\| u\sin y\|_{L^2}^2}{\delta^2}.\label{FL2}
\end{align}
Obviously, we have
\begin{align*}
&|\text{Im}\langle\mathcal{L}_1u,u\sin y\rangle|=\big|\f{\beta}{\nu}\|u\sin y \|_{L^2}^2+\text{Im}\langle\nu u',u\cos y \rangle\big|,\end{align*}
which gives
\begin{align*}\f{|\beta|}{\nu}\|u\sin y \|_{L^2}^2\leq \|\mathcal{L}_1u \|_{L^2}\|u\sin y  \|_{L^2}+\nu\|u' \|_{L^2}\|u \|_{L^2},
\end{align*}
hence,
\begin{align}
\f{|\beta|}{\nu}\|u\sin y \|_{L^2}^2\leq\f{\nu}{|\beta|}\|\mathcal{L}_1u\|_{L^2}^2+2\nu\|u' \|_{L^2}\|u \|_{L^2}.\label{FL3}
\end{align}
It follows from (\ref{FL2}) and (\ref{FL3}) that
\begin{align*}
\f{\|u \|_{L^1}^2}{\delta}+\|u \|_{L^2}^2\lesssim\delta\|u \|_{L^{\infty}}^2+\f{\nu^2}{\delta^2|\beta|^2}\|\mathcal{L}_1u \|_{L^2}^2+\f{\nu^2}{\delta^2|\beta|}\|u'\|_{L^2}\|u\|_{L^2}.
\end{align*}
Using the facts that
\beno
\|u \|_{L^{\infty}}^2\leq(\|u' \|_{L^2}+\|u \|_{L^2})\|u \|_{L^2},\quad |\langle\mathcal{L}_1u,u\rangle|\geq\nu\|u' \|_{L^2}^2+\nu\|u\|_{L^2}^2,
\eeno
we deduce that
\begin{align*}
\f{\|u \|_{L^1}^2}{\delta}+\|u \|_{L^2}^2\lesssim \f{\delta}{\nu^{\f12}}\|\mathcal{L}_1u \|_{L^2}^{\f12}\|u \|_{L^2}^{\f32}+\f{\nu^2}{\delta^2|\beta|^2}\|\mathcal{L}_1u\|_{L^2}^2+\f{\nu^{\f32}}{\delta^2|\beta|}\|\mathcal{L}_1u \|_{L^2}^{\f12}\|u \|_{L^2}^{\f32}.
\end{align*}
which implies that by taking $\delta=\nu^{\f23}|\beta|^{-\f13}\ll1$
that
\begin{align*}
&\|\mathcal{L}_1u \|_{L^2}\gtrsim\min\big\{\f{\nu}{\delta^2},\f{\delta|\beta|}{\nu},\f{\delta^4|\beta|^2}{\nu^3} \big\}\|u \|_{L^2}=\f{|\beta|^{\f23}}{\nu^{\f13}}\|u \|_{L^2},\\
&\|\mathcal{L}_1u \|_{L^2}\gtrsim\min\big\{\f{\nu}{\delta^4},\f{\delta^{\f12}|\beta|}{\nu},\f{\delta^2|\beta|^2}{\nu^3}\big\}\|u \|_{L^1}=\f{|\beta|^{\f56}}{\nu^{\f23}}\|u \|_{L^1}.
\end{align*}
Thus, $\mathcal{L}_1$ is injective, thanks to Fredholm alternative theorem, $\mathcal{L}_1$ is invertible. This completes the proof.
\end{proof}

\begin{lemma}\label{L5}
There exists a constant $c$ independent of $\beta,\nu$ so that
\begin{align}
 |\langle\mathcal{L}^{-1}_11,1\rangle|\geq c\nu/|\beta|\label{LowerB}.
\end{align}
\end{lemma}

\begin{proof}
Let $u=\mathcal{L}^{-1}_11$. Then we have
\begin{align}
 -\text{Re}\langle\mathcal{L}^{-1}_11,1\rangle=-\text{Re}\langle u,1\rangle=-\text{Re}\langle u,\mathcal{L}_1u\rangle =\nu\big(\|u'\|_{L^2}^2+\|u\|_{L^2}^2\big)\label{Lo1}.
\end{align}
For every $ \d\in(0,1)$, we can choose $\d_1\in(\d/2,\d)$ such that $ |u'(-\d_1)|^2+|u'(\d_1)|^2\leq 2\|u'\|_{L^2}^2/\d$.
Then we have $|u(y)-u(0)|\leq |y|^{\f12}\|u'\|_{L^2}$. Thanks to $\mathcal{L}_1u=1$, we get
\begin{align*}
 \d\leq 2\d_1&=\int_{-\d_1}^{\d_1}\mathcal{L}_1udy=\nu\big(u'(\d_1)-u'(-\d_1)\big)-\int_{-\d_1}^{\d_1}\big(\nu u(y)+i\f{\beta}{\nu}\sin y u(y)\big)dy\\&=\nu\big(u'(\d_1)-u'(-\d_1)\big)-\int_{-\d_1}^{\d_1}\big(\nu u(y)+i\f{\beta}{\nu}\sin y (u(y)-u(0))\big)dy\\&\leq\nu\big(|u'(\d_1)|+|u'(-\d_1)|\big)+\nu\| u\|_{L^1(-\d,\d)}+\f{|\beta|}{\nu}\int_{-\d}^{\d}|\sin y| |u(y)-u(0)|dy\\&\leq2\nu \|u'\|_{L^2}/\d^{\f12}+\nu(2\d)^{\f12}\| u\|_{L^2}+\f{|\beta|}{\nu}\int_{-\d}^{\d}|y||y|^{\f12}\|u'\|_{L^2}dy\\&\leq2\nu \|u'\|_{L^2}/\d^{\f12}+2\nu\| u\|_{L^2}+\f{|\beta|}{\nu}\d^{\f52}\|u'\|_{L^2},
\end{align*}
which by taking $\delta=\nu^{\f23}|\beta|^{-\f13}\ll1$ gives
\begin{align*}
\d\leq 3\nu\|u'\|_{L^2}/\d^{\f12}+2\nu\| u\|_{L^2}\leq C(\nu /\d^{\f12})(\|u'\|_{L^2}^2+\|u\|_{L^2}^2)^{\f12}.
\end{align*}
Then we infer from \eqref{Lo1}  that
 \begin{align*}
 |\langle\mathcal{L}^{-1}_11,1\rangle|\geq \nu(\|u'\|_{L^2}^2+\|u\|_{L^2}^2)\ge C^{-1}\nu(\d^{\f32}/\nu)^2=C^{-1}\d^3/\nu=C^{-1}\nu/|\beta|.
\end{align*}
\end{proof}

Let $X=L_0^2(\mathbb{T}_{2\pi})=\big\{f\in L^2(\mathbb{T}_{2\pi});\int_0^{2\pi} f(y)dy=0\big\}$ with the norm $\|f\|_*=\big\langle f,f-(1-\partial_y^2)^{-1}f\big\rangle^{1/2}$ and the inner product $\langle f,g\rangle_*=\big\langle f,g-(1-\partial_y^2)^{-1}g\big\rangle.$
We denote by $ \mathbb{Q}_1$ the orthogonal projection from $L^2(\mathbb{T}_{2\pi})$ to $X$. Then $ \mathbb{Q}_1\mathcal{L}'_{k_1,k_3}$ is a closed operator in $X$ with $D(\mathbb{Q}_1\mathcal{L}'_{k_1,k_3})=H^2(\mathbb{T}_{2\pi})\cap X.$ We have the following norm equivalence.

\begin{lemma}\label{L3}If $ f\in X=L_0^2(\mathbb{T}_{2\pi}),$ then
$\frac{1}{2}\|f\|_{L^2}^2\leq \|f\|_*^2\leq \|f\|_{L^2}^2.$
\end{lemma}
This result can be easily proved by using Fourier transform. \smallskip

Since $\langle \mathbb{Q}_1\mathcal{L}'_{k_1,k_3}f,f\rangle_*=\langle \mathcal{L}'_{k_1,k_3}f,f\rangle_*=-\nu k_f^2\langle \partial_y^2f,f\rangle_*=\nu k_f^2\|\partial_y f\|_{*}^2\geq 0$ for $f\in X,$ $\mathbb{Q}_1\mathcal{L}'_{k_1,k_3}$ is m-accretive. We define
\beno
\Psi(\mathbb{Q}_1\mathcal{L}'_{k_1,k_3})=\inf\Big\{\|(\mathbb{Q}_1\mathcal{L}'_{k_1,k_3}-i\la)f\|_*;f\in {D}(\mathbb{Q}_1\mathcal{L}'_{k_1,k_3}),\ \la\in\R,\  \|f\|_*=1\Big\}.
\eeno

\begin{lemma}\label{L2}If $ k_1^2+k_3^2=k_f^2=1,$ then
$\Psi(\mathbb{Q}_1\mathcal{L}'_{k_1,k_3})\geq c|k_1\gamma|^{\frac{1}{2}}.$
\end{lemma}
\begin{proof}
For $w\in {D}\big(\mathbb{Q}_1\mathcal{L}'_{k_1,k_3}\big)=H^2(\mathbb{T}_{2\pi})\cap X,\ \la\in\R,$ let $\al=k_1\gamma/k_f^4,\ \la_1=\frac{\nu k_f^2\lambda}{k_1\gamma},\ \varphi_0=(\partial_y^2-1)^{-1}w,\ u=w+\varphi_0$. Then $\partial_y^2\varphi_0=u$ and
\begin{align*}
(\mathcal{L}'_{k_1,k_3}-i\la)w&=k_f^2\left(i\f{\alpha}{\nu}\big((\sin y-\la_1)w+\sin y\varphi_0\big)-\nu\partial_y^2w\right)\\
&=k_f^2\left(i\f{\alpha}{\nu}\big((\sin y-\la_1)u+\la_1\varphi_0\big)-\nu\partial_y^2w\right) . \end{align*}
Since $(\mathbb{Q}_1\mathcal{L}'_{k_1,k_3}-i\la)w=(\mathcal{L}'_{k_1,k_3}-i\la)w-a $ for some constant $a$, if $\la\neq0$, then we have
\begin{align*}
(\mathbb{Q}_1\mathcal{L}'_{k_1,k_3}-i\la)w&=k_f^2\left(i\f{\alpha}{\nu}\big((\sin y-\la_1)u+\la_1\varphi\big)-\nu\partial_y^2w\right)
\end{align*}
with $\varphi=\varphi_0-\frac{a\nu}{i\al k_f^2\la_1}=\varphi_0-\frac{a}{i\la}, $ and $\partial_y^2\varphi=u$.

Using the fact that $\frac{1}{2}\|w\|_{L^2}^2\leq\|w\|_*^2=\langle w,u\rangle\leq \|w\|_{L^2}\|u\|_{L^2},$ we have
\beno
\|w\|_{L^2}\leq 2\|u\|_{L^2},\quad \|w\|_*^2\leq \|w\|_{L^2}\|u\|_{L^2}\leq 2\|u\|_{L^2}^2.
\eeno
Thanks to Proposition \ref{Prop: nonlocal2}, we have (fix $\varphi$ and let $ \widetilde{\beta}\to 0$)
\begin{align*}
2^{\frac{1}{2}}\|(\mathbb{Q}_1\mathcal{L}'_{k_1,k_3}-i\la)w\|_{*}\geq& \|(\mathbb{Q}_1\mathcal{L}'_{k_1,k_3}-i\la)w\|_{L^2}\\
=&\|i\f{\alpha}{\nu}[(\sin y-\lambda_1)w+\sin y\varphi]-\nu\partial_y^2w \|_{L^2}\\
\geq& \|i\f{\alpha}{\nu}[(\sin y-\lambda_1)u+\lambda_1\varphi]-\nu\partial_y^2u \|_{L^2}-\nu\|\varphi''\|_{L^2}\\
\geq&(c|\alpha|^{\f12}-\nu)\|u\|_{L^2}\geq c|\alpha|^{\f12}\|w\|_{*}
\end{align*} for $\la\neq 0.$
This is also true for $\la=0$ by taking the limit.
\end{proof}

Now we are in a position to prove Theorem \ref{thm:decay-L}. \smallskip

Thanks to  $\mathcal{L}'_{k_1,k_3}1=0$, we have $ \mathbb{Q}_1\mathcal{L}'_{k_1,k_3}=\mathbb{Q}_1\mathcal{L}'_{k_1,k_3}\mathbb{Q}_1$. Notice that $ \partial_t f=-\nu f-\mathcal{L}'_{k_1,k_3}f$. Hence,
\beno
\partial_t \mathbb{Q}_1f=-\nu \mathbb{Q}_1f-\mathbb{Q}_1\mathcal{L}'_{k_1,k_3}\mathbb{Q}_1f,
\eeno
 which gives $\mathbb{Q}_1f(t)=e^{-\nu t} e^{-t\mathbb{Q}_1\mathcal{L}'_{k_1,k_3}}\mathbb{Q}_1f_0.$
Then  by Lemma \ref{Lem:GP}, Lemma \ref{L2} and Lemma \ref{L3}, we deduce that
\beno
\|\mathbb{Q}_1f(t)\|_{L^2}\le 2 \|e^{-t\mathbb{Q}_1\mathcal{L}'_{k_1,k_3}}\mathbb{Q}_1f_0\|_{*}\leq e^{-ct|k_1\gamma|^{\frac{1}{2}}+{\pi}/{2}}\|\mathbb{Q}_1f_0\|_{L^2}.
\eeno

Thanks to $\partial_t\mathcal{L}^{-1}_1f=-\nu\mathcal{L}^{-1}_1f+u$, we have
\beno
\partial_t\langle\mathcal{L}_1^{-1}f,1\rangle=-\nu\langle\mathcal{L}_1^{-1}f,1\rangle,
\eeno
hence,
\begin{align}\label{L4}
& \langle\mathcal{L}_1^{-1}f,1\rangle(t)=e^{-\nu t} \langle\mathcal{L}_1^{-1}f_0,1\rangle.
\end{align}
Let $P_1=I-Q_1$. By Lemma \ref{L5} and \eqref{L4}, we get
\begin{align*}
|\mathbb{P}_{1}f(t)|&\leq C(|\beta|/\nu) |\langle\mathcal{L}^{-1}_1\mathbb{P}_{1}f(t),1\rangle|\leq C(|\beta|/\nu) (|\langle\mathcal{L}^{-1}_1\mathbb{Q}_{1}f(t),1\rangle|+|\langle\mathcal{L}^{-1}_1f(t),1\rangle|)\\
&\leq C(|\beta|/\nu)\big(e^{-\nu t}|\langle\mathcal{L}^{-1}_1f_0,1\rangle|+\|\mathcal{L}^{-1}_1\mathbb{Q}_{1}f(t)\|_{L^1}\big)\\
&\leq C(|\beta|/\nu)\big(e^{-\nu t}|\beta|^{-\f56}\nu^{\f23}\|f_0\|_{L^2}+|\beta|^{-\f56}\nu^{\f23}\|\mathbb{Q}_{1}f(t)\|_{L^2}\big)\\
&\leq C|\beta|^{\f16}\nu^{-\f13}\big(e^{-\nu t}\|f_0\|_{L^2}+e^{-\nu t}\|\mathbb{Q}_{1}f_0\|_{L^2}\big)\leq C|\gamma|^{\f16}\nu^{-\f13}e^{-\nu t}\|f_0\|_{L^2}.
\end{align*}
Thus, if $\mathbb{P}_1f_0=0$, then
\beno
\|\mathbb{P}_{1}f(t)\|_{L^2}\leq C|\gamma|^{\f16}\nu^{-\f13}e^{-\nu t}\|\mathbb{Q}_{1}f_0\|_{L^2}.
\eeno
 Otherwise, $f(t)-e^{-\nu t}\mathbb{P}_1f_0$ is also a solution of \eqref{eq:f-new}, and then
 \beno
 |\mathbb{P}_{1}f(t)-e^{-\nu t}\mathbb{P}_1f_0|\leq C|\gamma|^{\f16}\nu^{-\f13}e^{-\nu t}\|f_0-\mathbb{P}_1f_0\|_{L^2}=C|\gamma|^{\f16}\nu^{-\f13}e^{-\nu t}\|\mathbb{Q}_1f_0\|_{L^2},
\eeno
which gives
\begin{align*}
\|\mathbb{P}_{1}f(t)\|_{L^2}\leq e^{-\nu t}\|\mathbb{P}_{1}f_0\|_{L^2}+C|\gamma|^{\f16}\nu^{-\f13}e^{-\nu t}\|\mathbb{Q}_{1}f_0\|_{L^2}.
\end{align*}

Finally, due to $k_3=0,\ k_1^2+k_3^2=1,$ we have $g(t)=e^{-\nu t} e^{-t\mathcal{H}'_{k_1,k_3}}g_0$, and then
\begin{align*}
\|g(t)\|_{L^2}\leq Ce^{-ct|k_1\gamma|^{\frac{1}{2}}-\nu t}\|g_0\|_{L^2}.
\end{align*}


\section{Nonlinear stability threshold}

We will work on the system in terms of the shear wise velocity $v_2$ and vorticity $\om_2$
\begin{align}
\label{eq:NS-shear-new}
\left\{
\begin{aligned}
&\Big(\partial_t+\f{\gamma}{\nu k_f^2}\sin(k_fy)\partial_x-\nu\Delta\Big)v_2+\f{\gamma}{\nu}\sin (k_fy)\partial_xv_2=-\Delta(V\cdot \nabla v_2)-\partial_2(\Delta p),\\
&\Big(\partial_t+\f{\gamma}{\nu k_f^2}\sin(k_fy)\partial_x-\nu\Delta\Big)\omega_2+\f{\gamma}{\nu k_f}\cos(k_fy)\partial_zv_2=\nabla\cdot(-v \omega_2+\omega v_2),
\end{aligned}
\right.
\end{align}
where $\omega=\nabla\times V=(\omega_1,\omega_2,\omega_3)$ and $p=-\Delta^{-1}\Big(\sum\limits_{i,j=1}^3\partial_iv_j\partial_jv_i\Big)$.\smallskip

We introduce the norm
\begin{align*}
\|\varphi\|_{L^p_TL^q}=\left\|\|\varphi\|_{L^q} \right\|_{L^p(0,T)},\quad \|\varphi\|_{L^p_TH^N}=\left\|\|\varphi\|_{H^N}\right\|_{L^p(0,T)}.
\end{align*}
For the simplicity, $\|\varphi\|_{L^pL^q}=\|\varphi\|_{L^p_TL^q}$
and $\|\varphi\|_{L^pH^N}=\|\varphi\|_{L^p_TH^N}$ for $T=+\infty$.

\subsection{Decay estimates of the shear wise velocity}

We decompose the shear wise velocity as $v_2=v^{L}_2+v_2^{NL}$ with
\begin{align}
\label{NLNSK1}
\left\{
\begin{aligned}
&(\partial_t+\mathcal{L})\Delta v_2^{NL}=\text{div}f,\\
&\Delta v_2^{NL}(0)=0,
\end{aligned}
\right.
\end{align}
and
\begin{align}
\label{LNSK1}
\left\{
\begin{aligned}
&(\partial_t+\mathcal{L})\Delta v_2^{L}=0,\\
&\Delta v_2^{L}(0)=\Delta v_2(0),
\end{aligned}
\right.
\end{align}
where $\mathcal{L}=\f{\gamma}{\nu k_f^2}\sin(k_fy)\partial_x-\nu\Delta+\f{\gamma}{\nu}\sin (k_fy)\partial_x\Delta^{-1}$ and $f=-\nabla(V\cdot \nabla v_2)-(0,\Delta p,0).$

 \begin{lemma}
 There exists a a constant $C>0$ independent of $t$ so that
\begin{align}
\label{NLNSK bound1}&\|(\Delta v_2^{NL})_{\neq}(t)\|_{L^2}^2+\nu\int_0^t\|(\nabla\Delta v_2^{NL})_{\neq}(s)\|_{L^2}^2ds\leq C \nu^{-1}\int_0^t\|f_{\neq}(s) \|_{L^2}^2ds,\\
\label{div bound}&\Big\|\int_{0}^{t}e^{-(t-s)\mathcal{L}}(\text{div}f_{\neq}(s))ds\Big\|_{L^2}^2\leq C\nu^{-1}\int_0^t \|f_{\neq}(s)\|_{L^2}^2ds.
\end{align}
\end{lemma}

\begin{proof} By integration by parts, we get
\begin{align*}
\text{Re}\big\langle(\partial_t+\mathcal{L})\Delta v^{NL}_2,\Delta v^{NL}_2\big\rangle=&\f{1}{2}\partial_t\|\Delta v^{NL}_2\|_{L^2}^2+\nu\|\nabla\Delta v^{NL}_2\|_{L^2}^2+\text{Re}\big\langle \f{\gamma}{\nu}\sin (k_fy)\partial_xv^{NL}_2,\Delta v^{NL}_2\big\rangle\\
=&\text{Re}\big\langle \text{div}f,\Delta v^{NL}_2\big\rangle,
\end{align*}
and
\begin{align*}
\text{Re}\big\langle(\partial_t+\mathcal{L})\Delta v^{NL}_2,- v^{NL}_2\big\rangle=&\f{1}{2}\partial_t\|\nabla  v^{NL}_2\|_{L^2}^2+\nu\|\Delta v^{NL}_2\|_{L^2}^2-\text{Re}\big\langle \f{\gamma}{\nu k_f^2}\sin (k_fy)\partial_x\Delta  v^{NL}_2,  v^{NL}_2\big\rangle\\
=&\text{Re}\langle \text{div}f,- v^{NL}_2\rangle.
\end{align*}
Then we have
\begin{align*}
&\f{1}{2}\partial_t\big(\|\Delta  v^{NL}_2\|_{L^2}^2-k_f^2\|\nabla  v^{NL}_2\|_{L^2}^2\big)+\nu\big(\|\nabla\Delta  v^{NL}_2\|_{L^2}^2-k_f^2\|\Delta  v^{NL}_2\|_{L^2}^2\big)\\
&=\text{Re}\langle \text{div}f,\Delta v^{NL}_2\rangle-k_f^2\text{Re}\langle \text{div}f,- v^{NL}_2\rangle=-\text{Re}\langle f,\nabla\Delta  v^{NL}_2\rangle-k_f^2\text{Re}\langle  f,\nabla v^{NL}_2\rangle.
\end{align*}
It is easy to see that
\begin{align*}
-\text{Re}\langle f,\nabla\Delta  v^{NL}_2\rangle-k_f^2\text{Re}\langle  f,\nabla  v^{NL}_2\rangle&\leq\|f\|_{L^2}\|\nabla\Delta  v^{NL}_2\|_{L^2}+\|f\|_{L^2}\|\nabla v^{NL}_2\|_{L^2}\\
&\leq C\nu^{-1}(1-k_f^2)^{-1}\|f\|_{L^2}^2+\f{\nu(1-k_f^2)}{2}\|\nabla\Delta  v^{NL}_2\|_{L^2}^2,
\end{align*}
and
\beno
\|(\Delta  v_2^{NL})_{\neq}\|_{L^2}^2\geq\|(\nabla v_2^{NL})_{\neq}\|_{L^2}^2,\quad \|(\nabla\Delta v_2^{NL})_{\neq}\|_{L^2}^2\geq\|(\Delta v_2^{NL})_{\neq}\|_{L^2}^2.
\eeno
Thus, we conclude
\begin{align*}
\f{1}{2}\partial_t\big(\|(\Delta  v_2^{NL})_{\neq}\|_{L^2}^2-k_f^2\|(\nabla v_2^{NL})_{\neq}\|_{L^2}^2\big)+\f{\nu}{2}(1-k_f^2)\|(\nabla\Delta v_2^{NL})_{\neq}\|_{L^2}^2\lesssim\nu^{-1} \|f_{\neq}\|_{L^2}^2,
\end{align*}
which gives
\begin{align}\label{03}
(1-k_f^2)\Big(\|(\Delta v_2^{NL})_{\neq}(t)\|_{L^2}^2+\nu\int_0^t\|(\nabla\Delta v_2^{NL})_{\neq}(s)\|_{L^{2}}^2ds\Big)\lesssim \nu^{-1}\int_0^t \|f_{\neq}(s)\|_{L^2}^2ds.
\end{align}
This proves (\ref{NLNSK bound1}). Using the fact that
\begin{align*}
&\Delta v_2^{NL}(t)=\int_{0}^{t}e^{-(t-s)\mathcal{L}}(\text{div}\ f(s))ds,
\end{align*}
and (\ref{NLNSK bound1}), we infer that
\begin{align*}
\Big\|\int_{0}^{t}e^{-(t-s)\mathcal{L}}(\text{div}\ f_{\neq}(s))ds\|_{L^2}^2= \| (\Delta v_2^{NL})_{\neq}(t)\Big\|_{L^2}^2\lesssim \nu^{-1}\int_0^t\|f_{\neq}(s)\|_{L^2}^2ds.
\end{align*}
This proves (\ref{div bound}).
\end{proof}

\begin{lemma}\label{5.2}
Let $c$ be in Proposition \ref{Prop: L1} and $c'\in (0,c)$. It holds that
\begin{align}
\label{NLNSK decay1}
&\big\|e^{c'\sqrt{|\gamma|}t}(\Delta v_2^{NL})_{\neq}\big\|_{L^{\infty}L^2}^2+|\gamma|^{\f12}\big\|e^{c'\sqrt{|\gamma|}t}(\Delta v_2^{NL})_{\neq}\big\|_{L^2L^2}^2\\
&\qquad+\nu\|e^{c'\sqrt{|\gamma|}t}(\nabla \Delta v_2^{NL})_{\neq}\big\|_{L^{2}L^2}^2\lesssim \nu^{-1}\|e^{c'\sqrt{|\gamma|}t}f_{\neq} \|_{L^2L^2}^2.\nonumber
\end{align}
\end{lemma}

\begin{proof}
We define $J_k(\tau)$ by
\begin{align}
J_k(\tau)=\left\{
\begin{aligned}
&\int_{(k-1)\tau_0}^{k\tau_0}e^{-(\tau-s)\mathcal{L}}\text{div}\ f_{\neq}(s)ds,\quad\quad \tau\geq k\tau_0,\\
&\int_{(k-1)\tau_0}^{\tau}e^{-(\tau-s)\mathcal{L}}\text{div}\ f_{\neq}(s)ds ,\quad\quad (k-1)\tau_0\leq \tau\leq k\tau_0,\\
&0, \quad\quad  \tau\leq (k-1)\tau_0.
\end{aligned}
\right.
\end{align}
Thanks to (\ref{div bound}), we deduce that for $\tau\in [(k-1)\tau_0, k\tau_0]$
\begin{align}
\label{cut1}\|J_k(\tau)\|_{L^2}^2\lesssim \nu^{-1}\int_{(k-1)\tau_0}^\tau\|f_{\neq}(s)\|_{L^2}^2ds\leq \nu^{-1}\int_{(k-1)\tau_0}^{k\tau_0}\|f_{\neq}(s)\|_{L^2}^2ds:=I_k^2.
\end{align}
As $J_k(\tau)=e^{-(\tau-k\tau_0)\mathcal{L}}J_k(k\tau_0)$ for $\tau\geq k\tau_0$, we deduce from (\ref{shear part})  that for $\tau\geq k\tau_0$\begin{align}
\label{cut2}\|J_k(\tau)\|_{L^2}^2\lesssim  e^{-2c\sqrt{|\gamma|}(\tau-k\tau_0)} \|J_k(k\tau_0)\|_{L^2}^2\lesssim e^{-2c\sqrt{|\gamma|}(\tau-k\tau_0)}I_k^2.
\end{align}
Taking $\tau_0=\f{1}{\sqrt{|\gamma|}}$. It follows from (\ref{cut1}) and (\ref{cut2}) that  for $t\in[(l-1)\tau_0,l\tau_0](l\in\mathbb{N^+})$\begin{align*}
\|(\Delta v_2^{NL})_{\neq}(t)\|_{L^2}\leq&\sum\limits_{k=1}^{+\infty} \|J_k(t)\|_{L^2}\leq \sum\limits_{k=1}^{l-1} \|J_k(t)\|_{L^2}+CI_{l}\leq C\sum\limits_{k=1}^{l} e^{-c(l-k)}I_k.
\end{align*}
Notice that
\begin{align*}
\nu^{-1}\int_{(k-1)\tau_0}^{k\tau_0}e^{2c'\sqrt{|\gamma|}s}\|f_{\neq}(s)\|_{L^2}^2ds\sim \nu^{-1}e^{2c'k}\int_{(k-1)\tau_0}^{k\tau_0}\|f_{\neq}(s)\|_{L^2}^2ds,
\end{align*}
which implies that
\beno
\sum\limits_{k=1}^{+\infty} e^{2c'k}I_k^2\leq C \nu^{-1}\|e^{c'\sqrt{|\gamma|}t}f_{\neq} \|_{L^2L^2}^2.
\eeno
 Then we deduce that for $c'\in(0,c)$
\begin{align*}
\|(\Delta v_2^{NL})_{\neq}(t)\|_{L^2}&\lesssim\sum\limits_{k=1}^{l} e^{-c(l-k)}I_k\lesssim \Big(\sum\limits_{k=1}^{l} e^{-2c'(l-k)}I_k^2\Big)^{\f12}\Big(\sum\limits_{k=1}^{l} e^{-2(c-c')(l-k)}\Big)^{\f12}\\
&\lesssim \nu^{-1}e^{-c'l} \|e^{c'\sqrt{|\gamma|}t}f_{\neq} \|_{L^2L^2}
\end{align*}
for any $t\in[(l-1)\tau_0,l\tau_0](l\in\mathbb{N^+})$, which implies that
\begin{align}
\label{NLNSK decay1.1}
\big\|e^{c'\sqrt{|\gamma|}t}(\Delta v_2^{NL})_{\neq}\big\|_{L^{\infty}L^2}^2\lesssim \nu^{-1}\big\|e^{c'\sqrt{|\gamma|}t}f _{\neq}\big\|_{L^2L^2}^2.
\end{align}

On the other hand, we also have
\begin{align*}
\|(\Delta v_2^{NL})_{\neq}(t)\|_{L^2}&\lesssim\sum\limits_{k=1}^{l} e^{-c(l-k)}I_k\lesssim \Big(\sum\limits_{k=1}^{l} e^{-(c+c')(l-k)}I_k^2\Big)^{\f12}\Big(\sum\limits_{k=1}^{l} e^{-(c-c')(l-k)}\Big)^{\f12}\\
&\lesssim \Big(\sum\limits_{k=1}^{l} e^{-(c+c')(l-k)}I_k^2\Big)^{\f12}\end{align*}
for any $t\in[(l-1)\tau_0,l\tau_0](l\in\mathbb{N^+})$, which implies that
\begin{align}
\nonumber{|\gamma|}^{\frac{1}{2}}\int_0^{+\infty}e^{2c'\sqrt{|\gamma|}t} \| (\Delta v_2^{NL})_{\neq}(t)\|_{L^2}^2dt\leq& C {|\gamma|}^{\frac{1}{2}}\sum\limits_{l=0}^{+\infty}e^{2c'l}\int_{l\tau_0}^{(l+1)\tau_0} \| (\Delta v_2^{NL})_{\neq}(t)\|_{L^2}^2dt\\
\nonumber\leq&C {|\gamma|}^{\frac{1}{2}}\sum\limits_{l=0}^{+\infty}e^{2c'l} \tau_0\sum\limits_{k=1}^{l+1} e^{-(c+c')(l+1-k)}I_k^2\\
=& C \sum\limits_{k=1}^{+\infty}e^{2c'k}I_k^2 \sum\limits_{l=k-1}^{+\infty} e^{-2(c-c')(l+1-k)}\nonumber\\
\label{NLNSK decay1.2}\lesssim&  \sum\limits_{k=1}^{+\infty}e^{2c'k}I_k^2 \lesssim\nu^{-1}\|e^{c'\sqrt{|\gamma|}t}f_{\neq}(t)\|_{L^2L^2}^2.
\end{align}

It follows from \eqref{03} that
 \begin{align*}
&\f{1}{2}\partial_te^{2c'\sqrt{|\gamma|}t}\big(\|(\Delta  v_2^{NL})_{\neq}\|_{L^2}^2-k_f^2\|(\nabla v_2^{NL})_{\neq}\|_{L^2}^2\big)+\f{\nu}{2}(1-k_f^2)e^{2c'\sqrt{|\gamma|}t}\|(\nabla\Delta v_2^{NL})_{\neq}\|_{L^2}^2\\
&\lesssim\nu^{-1} e^{2c'\sqrt{|\gamma|}t}\|f_{\neq}\|_{L^2}^2+{|\gamma|}^{\frac{1}{2}}e^{2c'\sqrt{|\gamma|}t}\big(\|(\Delta  v_2^{NL})_{\neq}\|_{L^2}^2-k_f^2\|(\nabla v_2^{NL})_{\neq}\|_{L^2}^2\big),
\end{align*}
from which and (\ref{NLNSK decay1.2}), it follows that
\begin{align}
\nonumber&\nu\int_0^{+\infty}e^{2c'\sqrt{|\gamma|}t} \| (\nabla \Delta v_2^{NL})_{\neq}(t)\|_{L^2}^2dt\\
\nonumber&\leq C \nu^{-1}\int_0^{+\infty}e^{2c'\sqrt{|\gamma|}t} \| f_{\neq}(t)\|_{L^2}^2dt+C\int_0^{+\infty}{|\gamma|}^{\frac{1}{2}}e^{2c'\sqrt{|\gamma|}t}\|(\Delta  v_2^{NL})_{\neq}(t)\|_{L^2}^2dt\\
\label{NLNSK decay1.3}&\leq C\nu^{-1}\|e^{c'\sqrt{|\gamma|}t}f_{\neq}(t)\|_{L^2L^2}^2.
\end{align}

Summing up (\ref{NLNSK decay1.1}), (\ref{NLNSK decay1.2}) and (\ref{NLNSK decay1.3}), we conclude (\ref{NLNSK decay1}).
\end{proof}

\begin{proposition}\label{Prop5.3}
Let $c$ be as in Proposition \ref{Prop: L1} and $c'\in (0,c)$.
It holds that
\begin{align}
\label{NLNSK decay1'}&\|e^{c'\sqrt{|\gamma|}t}(\Delta v_2)_{\neq}\|_{L^{\infty}L^2}^2+|\gamma|^{\f12}\|e^{c'\sqrt{|\gamma|}t}(\Delta v_2)_{\neq}\|_{L^2L^2}^2+\nu\|e^{c'\sqrt{|\gamma|}t}(\nabla \Delta v_2)_{\neq}\|_{L^{2}L^2}^2\\
\nonumber&\lesssim  \|(\Delta v_2)_{\neq}(0)\|_{L^2}^2+\nu^{-1}\|e^{c'\sqrt{|\gamma|}t}f_{\neq} \|_{L^2L^2}^2.
\end{align}
\end{proposition}
\begin{proof}
Thanks to $(\Delta v_2^{L})(t)=e^{-t\mathcal{L}}(\Delta v_2^{L})(0)$, it follows from (\ref{shear part}) and (\ref{shear part energy}) that
\begin{align*}
&\|e^{c'\sqrt{|\gamma|}t}(\Delta v_2^{L})_{\neq}\|_{L^{\infty}L^2}^2+{|\gamma|}^{\frac{1}{2}}\|e^{c'\sqrt{|\gamma|}t}(\Delta v_2^{L})_{\neq}\|_{L^{2}L^2}^2+\nu\|e^{c'\sqrt{|\gamma|}t}(\nabla\Delta v_2^{L})_{\neq}\|_{L^{2}L^2}^2\\ &\lesssim \|(\Delta v_2)_{\neq}(0)\|_{L^2}^2,
\end{align*}
which along with (\ref{NLNSK decay1}) gives  \eqref{NLNSK decay1'}  due to $(\Delta v_2)_{\neq}=(\Delta v_2^{NL})_{\neq}+(\Delta v_2^{L})_{\neq}$. \end{proof}

\subsection{Decay estimates of the shear wise vorticity}
We write
\begin{align}
\label{NLNSK2}
(\partial_t+\mathcal{H})\omega_2=-\f{\gamma}{\nu k_f}\cos(k_fy)\partial_zv_2+\text{div}g,
\end{align}
where $\mathcal{H}=\f{\gamma}{\nu k_f^2}\sin(k_fy)\partial_x-\nu\Delta$ and $g=-\omega_2(v_1,v_2,v_3)+v_2(\omega_1,\omega_2,\omega_3).$

Similar to the proof of Proposition \ref{Prop5.3}, we have

\begin{proposition}\label{Prop5.4}
Let $c$ be as in  Proposition \ref{Prop: H1} and $c'\in (0,c)$.
Let $(g_1,g_2)$ satisfy $(\partial_t+\mathcal{H})g_1=\text{div}g_2 $. If $P_0g_1=P_0 g_2=0$, then we have
\begin{align}
\label{NLNSK decay3}&\|g_1\|_{X_{c'}}^2 \leq C\| g_1(0)\|_{L^2}^2+C\nu^{-1}\|e^{c'\sqrt{|\gamma|}t}g_2\|_{L^{2}L^2}^2,
\end{align}
where
\begin{align*}
\|g_1\|_{X_{c'}}^2=\|e^{c'\sqrt{|\gamma|}t}g_1\|_{L^{\infty}L^2}^2+{|\gamma|}^{\frac{1}{2}}\|e^{c'\sqrt{|\gamma|}t}g_1\|_{L^{2}L^2}^2+\nu\|e^{c'\sqrt{|\gamma|}t}\nabla g_1\|_{L^{2}L^2}^2.
\end{align*}
\end{proposition}

To handle \eqref{NLNSK2}, we need to use the wave operator.  Recall that
we introduce $\mathbb{D}_2=\mathbb{D}_2^{(\al)}$ in section 5.1. For $f(x,y,z)=\sum\limits_{k_1\neq 0}\widehat{f}(k_1,k_fy,k_3)e^{i(k_1x+k_3z)}$, we define
\beno
\widetilde{\mathbb{D}}_2f(x,y,z)=\sum\limits_{k_1\neq 0}\mathbb{D}_2^{(\sqrt{k_1^2+k_3^2}/k_f)}\widehat{f}(k_1,\cdot,k_3)(k_fy)e^{i(k_1x+k_3z)}.
\eeno
Similar to \eqref{oper,sinyD2}-\eqref{D6}, we can prove
\begin{align}
&\widetilde{\mathbb{D}}_2\big(\sin (k_fy)(1+k_f^2\Delta^{-1})\omega\big)=\sin (k_fy)\widetilde{\mathbb{D}}_2(\omega)-k_f^2\Delta^{-1}\omega,\label{oper,sinyD21}\\
&\|\cos (k_fy)\widetilde{\mathbb{D}}_2(\Delta\omega)-\Delta(\cos (k_fy)\widetilde{\mathbb{D}}_2(\omega))\|_{L^2}\leq C\|\nabla\omega\|_{L^2},\label{est of D21}\\
&\|(\partial_x,\partial_z)\cos (k_fy)\widetilde{\mathbb{D}}_2(\omega)\|_{L^2}\leq C\|\omega\|_{L^2},\label{D2.2}\\
&\|\nabla(\partial_x,\partial_z)\cos (k_fy)\widetilde{\mathbb{D}}_2(\omega)\|_{L^2}\leq C\|\nabla\omega\|_{L^2}.\label{D2.3}
\end{align}

\begin{proposition}
Let $c$ be as in Proposition \ref{Prop: H1} and $c'\in (0,c)$. Then  we have
\begin{align}
\label{NLNSK decay2}\|\partial_x\omega_2\|_{X_{c'}}^2\lesssim&\|(\Delta v_2)_{\neq}(0)\|_{L^2}^2+\|\partial_x\omega_2(0)\|_{L^2}^2\\&+ \nu^{-1}\Big(\|e^{c'\sqrt{|\gamma|}t}f_{\neq}\|_{L^{2}L^2}^2+\|e^{c'\sqrt{|\gamma|}t}\partial_xg\|_{L^{2}L^2}^2\Big).\nonumber
\end{align}
\end{proposition}

\begin{proof}
First of all,  using \eqref{oper,sinyD21} and $\widetilde{\mathbb{D}}_2\partial_x=\partial_x\widetilde{\mathbb{D}}_2 $, we have
\begin{align*}
&\cos (k_fy)\widetilde{\mathbb{D}}_2\circ\mathcal{L}-\mathcal{H}\circ\cos (k_fy)\widetilde{\mathbb{D}}_2=-\nu\big[\cos (k_fy)\widetilde{\mathbb{D}}_2,\Delta\big]-\f{\gamma}{\nu }\cos (k_fy)\partial_x\Delta^{-1}.
\end{align*}
Let $g_1=\partial_x\omega_2+\f{1}{k_f}\partial_z\cos (k_fy)\widetilde{\mathbb{D}}_2((\Delta v_2)_{\neq})$. Then we find that
\begin{align*}
&(\partial_t+\mathcal{H})g_1=\partial_x(\partial_t+\mathcal{H})\omega_2+\f{1}{k_f}\partial_z(\partial_t+\mathcal{H})\cos (k_fy)\widetilde{\mathbb{D}}_2((\Delta v_2)_{\neq})\\
&=\partial_x\big(-\f{\gamma}{\nu k_f}\cos(k_fy)\partial_3v_2+\text{div}g\big)+\f{1}{k_f}\partial_z\partial_t\cos (k_fy)\widetilde{\mathbb{D}}_2((\Delta v_2)_{\neq})\\
&\quad+\f{1}{k_f}\partial_z\cos (k_fy)\widetilde{\mathbb{D}}_2(\mathcal{L}(\Delta v_2)_{\neq})+\f{1}{k_f}\partial_z\Big(\nu [\cos (k_fy)\widetilde{\mathbb{D}}_2,\Delta]+\f{\gamma}{\nu }\cos (k_fy)\partial_1\Delta^{-1}\Big)(\Delta v_2)_{\neq}\\
&=-\f{\gamma}{\nu k_f}\cos(k_fy)\partial_x\partial_zv_2+\text{div}\partial_1 g+\f{1}{k_f}\partial_z\cos (k_fy)\widetilde{\mathbb{D}}_2((\partial_t+\mathcal{L})(\Delta v_2)_{\neq})\\
&\quad+\partial_z\big(\f{1}{k_f}\nu [\cos (k_fy)\widetilde{\mathbb{D}}_2,\Delta](\Delta v_2)_{\neq}\big)+\f{\gamma}{\nu k_f}\cos(k_fy)\partial_z\partial_x(v_2)_{\neq}\\
&=\text{div}\partial_x g+\f{1}{k_f}\partial_z\cos (k_fy)\widetilde{\mathbb{D}}_2((\Delta f_0)_{\neq})+\partial_z\Big(\f{1}{k_f}\nu [\cos (k_fy)\widetilde{\mathbb{D}}_2,\Delta](\Delta v_2)_{\neq}\Big)\\
&=\text{div}\partial_xg+\f{1}{k_f}\partial_z\Delta\cos (k_fy)\widetilde{\mathbb{D}}_2((f_0)_{\neq})+\partial_z\Big(\f{1}{k_f} [\cos (k_fy)\widetilde{\mathbb{D}}_2,\Delta](\nu\Delta v_2+f_0)_{\neq}\Big):=\text{div}g_2,
\end{align*}
where $\partial_x(v_2)_{\neq}=\partial_xv_2$, $f_0=-V\cdot\na v_2-\pa_yp$ and
\begin{align}
 g_2=\partial_x g+\f{1}{k_f}\partial_z\nabla\cos (k_fy)\widetilde{\mathbb{D}}_2((f_0)_{\neq})+\Big(0,0,\f{1}{k_f} [\cos (k_fy)\widetilde{\mathbb{D}}_2,\Delta](\nu\Delta v_2+f_0)_{\neq}\Big).\label{Key trick}
\end{align}
By  \eqref{est of D21} and \eqref{D2.3}, we get
\begin{align}
\nonumber\|g_2\|_{L^2}&\leq \|\partial_xg\|_{L^2}+\f{1}{k_f}\big\|\partial_z\nabla\cos (k_fy)\widetilde{\mathbb{D}}_2((f_0)_{\neq})\big\|_{L^2}+\f{1}{k_f} \|[\cos (k_fy)\widetilde{\mathbb{D}}_2,\Delta](\nu\Delta v_2+f_0)_{\neq}\|_{L^2}\\
\nonumber&\leq\|\partial_xg\|_{L^2}+C\|\nabla (f_0)_{\neq}\|_{L^2}+C\|\nabla (\nu\Delta v_2+f_0)_{\neq}\|_{L^2}\\
&\leq\|\partial_xg\|_{L^2}+C\|\nabla (f_0)_{\neq}\|_{L^2}+C\nu\|\nabla (\Delta v_2)_{\neq}\|_{L^2}.\label{KEY estimate}
\end{align}
Using (\ref{NLNSK decay3}), (\ref{KEY estimate}), \eqref{D2.2}, $\text{div}\ f=\Delta f_0$ and \eqref{NLNSK decay1'}, we deduce
\begin{align*}
\|g_1\|_{X_{c'}}^2\lesssim& \nu^{-1}\|e^{c'\sqrt{|\gamma|}t}g_2\|_{L^{2}L^2}^2+\| g_1(0)\|_{L^2}^2\\
\lesssim&\nu^{-1}\|e^{c'\sqrt{|\gamma|}t}\partial_xg\|_{L^{2}L^2}^2+\nu^{-1}\|e^{c'\sqrt{|\gamma|}t}\nabla (f_0)_{\neq}\|_{L^{2}L^2}^2+\nu\|e^{c'\sqrt{|\gamma|}t}\nabla (\Delta v_2)_{\neq}\|_{L^{2}L^2}^2\\
&+\| \partial_x\omega_2(0)\|_{L^2}^2+\| k_f^{-1}\partial_z\cos (k_fy)\widetilde{\mathbb{D}}_2((\Delta v_2)_{\neq})(0)\|_{L^2}^2\\
 \lesssim& \nu^{-1}\|e^{c'\sqrt{|\gamma|}t}\partial_x g\|_{L^{2}L^2}^2+\nu^{-1}\|e^{c'\sqrt{|\gamma|}t} f_{\neq}\|_{L^{2}L^2}^2+\big(\|(\Delta v_2)_{\neq}(0)\|_{L^2}^2+\nu^{-1}\|e^{c'\sqrt{|\gamma|}t}f_{\neq}\|_{L^2L^2}^2\big)\\
& +\| \partial_x\omega_2(0)\|_{L^2}^2+\| (\Delta v_2)_{\neq}(0)\|_{L^2}^2.
 \end{align*}
This shows that
 \begin{align}
\|g_1\|_{X_{c'}}^2 \lesssim& \| \partial_x\omega_2(0)\|_{L^2}^2+\| (\Delta v_2)_{\neq}(0)\|_{L^2}^2+\f{1}{\nu}\big(\|e^{c'\sqrt{|\gamma|}t}\partial_1 g\|_{L^{2}L^2}^2+\|e^{c'\sqrt{|\gamma|}t} f_{\neq}\|_{L^{2}L^2}^2\big).\label{KEY estimate1}
\end{align}
On the other hand, by \eqref{D2.2} and \eqref{D2.3}, we have
\begin{align*}
\f{1}{k_f}\|\partial_z\cos (k_fy)\widetilde{\mathbb{D}}_2((\Delta v_2)_{\neq})\|_{X_{c'}}^2\leq C\| (\Delta v_2)_{\neq}\|_{X_{c'}}^2,
\end{align*}
and by \eqref{NLNSK decay1'},
\begin{align*}
\| (\Delta v_2)_{\neq}\|_{X_{c'}}^2\leq C\big(\|(\Delta v_2)_{\neq}(0)\|_{L^2}^2+\nu^{-1}\|e^{c'\sqrt{|\gamma|}t}f_{\neq} \|_{L^2L^2}^2\big).
\end{align*}
Therefore, we conclude
\begin{align*}
\|\partial_x\omega_2\|_{X_{c'}}^2\lesssim& \|g_1\|_{X_{c'}}^2+\f{1}{k_f}\|\partial_z\cos (k_fy)\widetilde{\mathbb{D}}_2((\Delta v_2)_{\neq})\|_{X_{c'}}^2\\
\lesssim& \|g_1\|_{X_{c'}}^2+\big(\|(\Delta v_2)_{\neq}(0)\|_{L^2}^2+\nu^{-1}\|e^{c'\sqrt{|\gamma|}t}f_{\neq} \|_{L^2L^2}^2\big)\\
\lesssim& \| \partial_x\omega_2(0)\|_{L^2}^2+\|(\Delta v_2)_{\neq}(0)\|_{L^2}^2+\nu^{-1}\|e^{c'\sqrt{|\gamma|}t}\partial_1 g\|_{L^{2}L^2}^2+\nu^{-1}\|e^{c'\sqrt{|\gamma|}t} f_{\neq}\|_{L^{2}L^2}^2,
\end{align*}
which gives \eqref{NLNSK decay2}.
\end{proof}

\subsection{Nonlinear estimates}

The following lemma gives the estimates of the velocity in terms of
$(v_2,w_2)$.

\begin{lemma}\label{lem1}
It holds that for $i,j\in\{1,3\}$,
\begin{align}
\label{5.37}&\|\partial_i\partial_j(v_i)_{\neq}\|_{L^2}+\|\partial_j(v_i)_{\neq}\|_{L^2}+\|\partial_x\partial_jV\|_{L^2}\leq C\big(\|(\Delta v_2)_{\neq}\|_{L^2}+\|\partial_x\omega_2\|_{L^2}\big),\\
\label{5.38}&\|\partial_i\nabla(v_i)_{\neq}\|_{L^2}+\|\nabla(v_i)_{\neq}\|_{L^2}+\|\partial_x\omega\|_{L^2}\leq C\big(\|(\Delta v_2)_{\neq}\|_{L^2}+\|\nabla\partial_x\omega_2\|_{L^2}\big),
\end{align}
\end{lemma}

\begin{proof}
First of, we have the following relations
\begin{align*}
v_1=(\partial_x^2+\partial_z^2)^{-1}(\partial_z\omega_2-\partial_x\partial_yv_2),\quad v_3=-(\partial_x^2+\partial_z^2)^{-1}(\partial_x\omega_2+\partial_z\partial_yv_2).
\end{align*}
Using the fact that $ \|\varphi_{\neq}\|_{L^2}\leq \|\partial_x\varphi_{\neq}\|_{L^2}=\|\partial_x\varphi\|_{L^2},$ we infer that
\begin{align*}
&\sum_{i,j\in\{1,3\}}(\|\partial_i\partial_j(v_i)_{\neq}\|_{L^2}+\|\partial_j(v_i)_{\neq}\|_{L^2}+\|\partial_x\partial_jV\|_{L^2})\\
&\lesssim \|\partial_x^2(v_1)_{\neq}\|_{L^2}+\|\partial_1\partial_z(v_1)_{\neq}\|_{L^2}+
\|\partial_z^2(v_3)_{\neq}\|_{L^2}+\|\partial_x\partial_z(v_3)_{\neq}\|_{L^2}\\&\quad+\|\partial_x(v_1)_{\neq}\|_{L^2}
+\|\partial_z(v_1)_{\neq}\|_{L^2}+\|\partial_x(v_3)_{\neq}\|_{L^2}
+\|\partial_z(v_3)_{\neq}\|_{L^2}\\
&\quad+\|\partial_x^2v_1\|_{L^2}+\|\partial_x\partial_zv_1\|_{L^2}
+\|\partial_x^2v_2\|_{L^2}+\|\partial_x\partial_zv_2\|_{L^2}
+\|\partial_x^2v_3\|_{L^2}+\|\partial_x\partial_zv_3\|_{L^2}\\
&\leq \|\partial_x^2(v_1)\|_{L^2}+\|\partial_x\partial_zv_1\|_{L^2}+
\|\partial_z^2(v_3)_{\neq}\|_{L^2}+\|\partial_x\partial_zv_3\|_{L^2}
+\|\partial_x^2v_3\|_{L^2}+\|\partial_x\nabla v_2\|_{L^2}
)\\
&\leq  \|(\partial_x^2+\partial_z^2)^{-1}(\partial_x^2\partial_z\omega_2-\partial_x^3\partial_yv_2) \|_{L^2}+\|(\partial_x^2+\partial_z^2)^{-1}(\partial_z^2\partial_x\omega_2-\partial_x^2\partial_z\partial_yv_2)\|_{L^2}\\&\quad+
\|(\partial_x^2+\partial_z^2)(v_3)_{\neq}\|_{L^2}+\|\partial_x\nabla v_2\|_{L^2}
)\\
&\lesssim  \|\partial_x\omega_2\|_{L^2}+\|\partial_y\partial_xv_2\|_{L^2}+
\|(\partial_x\omega_2+\partial_z\partial_yv_2)_{\neq}\|_{L^2}+\|\partial_x\nabla v_2\|_{L^2}
)\\ &\lesssim \|\partial_x\omega_2\|_{L^2}+\|\partial_x\nabla v_2\|_{L^2}+
\|(\Delta v_2)_{\neq}\|_{L^2}
)\lesssim \|\partial_x\omega_2\|_{L^2}+
\|(\Delta v_2)_{\neq}\|_{L^2},
\end{align*}
which gives \eqref{5.37}.

Similarly, we have
\begin{align*}
&\sum\limits_{i\in \{ 1,3 \}} \|\partial_i\nabla(v_i)_{\neq}\|_{L^2}+\sum\limits_{i\in \{ 1,3 \}} \|\nabla(v_i)_{\neq}\|_{L^2}+\|\partial_x\omega\|_{L^2}\\
&\leq \|\partial_x\nabla(v_1)_{\neq}\|_{L^2}+\|\partial_z\nabla(v_3)_{\neq}\|_{L^2}+\|\nabla(v_1)_{\neq}\|_{L^2}
+\|\nabla(v_3)_{\neq}\|_{L^2}+\|\partial_x\nabla V\|_{L^2}\\
&\lesssim \|\partial_x\nabla(v_1)_{\neq}\|_{L^2}+\|\partial_z\nabla(v_3)_{\neq}\|_{L^2}
+\|\partial_x\nabla(v_3)_{\neq}\|_{L^2}+\|\partial_x\nabla v_2\|_{L^2})\\
&\lesssim \|(\partial_x^2+\partial_z^2)^{-1}(\partial_x\partial_z\nabla\omega_2-\partial_x^2\partial_y\nabla v_2)_{\neq}\|_{L^2}+\|(\partial_x^2+\partial_z^2)^{-1}(\partial_x\partial_z\nabla\omega_2+\partial_z^2\partial_y\nabla v_2)_{\neq}\|_{L^2}\\&\quad
+\|(\partial_x^2+\partial_z^2)^{-1}(\partial_x^2\nabla\omega_2+\partial_z\partial_x\partial_y\nabla v_2)_{\neq}\|_{L^2}
+\|\partial_x\nabla v_2\|_{L^2}\\
&\lesssim  \|(\nabla\omega_2)_{\neq}\|_{L^2}+\|(\partial_y\nabla v_2)_{\neq}\|_{L^2}
+\|\partial_x\nabla v_2\|_{L^2})\lesssim \|(\Delta v_2)_{\neq}\|_{L^2}
+\|\nabla\partial_x\omega_2\|_{L^2},
\end{align*}
which gives \eqref{5.38}.
\end{proof}

For the nonlinear estimates, we need the following lemma.

\begin{lemma}\label{lem2}
It holds that for $j\in\{1,3\}$
\begin{align*}
&\|f_1\partial_jf_2\|_{L^2}\leq C\big(\|\partial_jf_1\|_{L^2}+\|f_1\|_{L^2}\big)\|\Delta f_2\|_{L^2}.
\end{align*}
\end{lemma}

\begin{proof}
Let us consider the case of $j=1$. We write
$f_l(x,y,z)=\sum\limits_{k\in\mathbb{Z}}e^{ikx}f_{l,k}(y,z)$ for $l=1,2.$ Then we have
\beno
&&\|f_1\|_{L^2}^2+\|\partial_xf_1\|_{L^2}^2=2\pi\sum\limits_{k\in\mathbb{Z}}(1+k^2)\|f_{1,k}\|_{L^2}^2,\\
&&\|\Delta f_2\|_{L^2}^2=2\pi\sum\limits_{k\in\mathbb{Z}}\big(k^4\|f_{2,k}\|_{L^2}^2+2k^2\|\nabla f_{2,k}\|_{L^2}^2+\|\Delta f_{2,k}\|_{L^2}^2\big).
\eeno
Therefore,
\begin{align}\label{f1}\Big(\sum\limits_{k\in\mathbb{Z}}\|f_{1,k}\|_{L^2}\Big)^2\leq \sum\limits_{k\in\mathbb{Z}}(1+k^2)\|f_{1,k}\|_{L^2}^2\sum\limits_{k\in\mathbb{Z}}\dfrac{1}{1+k^2}\leq C\big(\|{f_1}\|_{L^2}^2+\|\partial_xf_1\|_{L^2}^2\big).
\end{align}
By 2-D Gagliardo-Nirenberg inequality, we have
\begin{align}\label{f2}k^2\|f_{2,k}\|_{L^{\infty}}^2\leq Ck^2\|f_{2,k}\|_{L^{2}}\|f_{2,k}\|_{H^{2}}\leq C\big(k^4\|f_{2,k}\|_{L^2}^2+\|\Delta f_{2,k}\|_{L^2}^2\big). \end{align}
We write $e^{imx}f_{1,m}(y,z)\partial_xf_2(x,y,z)=\sum\limits_{k\in\mathbb{Z}}ike^{i(k+m)x}f_{1,m}(y,z)f_{2,k}(y,z),$ and by \eqref{f2},
\begin{align*}
\|e^{imx}f_{1,m}(y,z)\partial_jf_2\|_{L^2}^2&=2\pi\sum\limits_{k\in\mathbb{Z}}\|ikf_{1,m}f_{2,k}\|_{L^2}^2\leq
2\pi\sum\limits_{k\in\mathbb{Z}}k^2\|f_{1,m}\|_{L^2}^2\|f_{2,k}\|_{L^{\infty}}^2\\ &\leq C\sum\limits_{k\in\mathbb{Z}}\|f_{1,m}\|_{L^2}^2(k^4\|f_{2,k}\|_{L^2}^2+\|\Delta f_{2,k}\|_{L^2}^2)\leq C\|f_{1,m}\|_{L^2}^2\|\Delta f_2\|_{L^2}^2,
\end{align*}
from which and \eqref{f1}, we infer that
\begin{align*}
\|f_1\partial_xf_2\|_{L^2}=&\big\|\sum\limits_{m\in\mathbb{Z}}e^{imx}f_{1,m}(y,z)\partial_xf_2\big\|_{L^2}\leq \sum\limits_{m\in\mathbb{Z}}\|e^{imx}f_{1,m}(y,z)\partial_jf_2\|_{L^2}\\
\leq& \sum\limits_{l\in\mathbb{Z}} C\|f_{1,m}\|_{L^2}\|\Delta f_2\|_{L^2}\leq C\big(\|\partial_xf_1\|_{L^2}+\|f_1\|_{L^2}\big)\|\Delta f_2\|_{L^2}.
\end{align*}

The case of $j=3$ is similar.
\end{proof}

\begin{lemma}
\label{Union}
It holds that
\begin{align}
\label{UNION1}&\|(\Delta p)_{\neq}\|_{L^2}+\|\nabla(V\cdot\nabla v_2)_{\neq}\|_{L^2}+\|\partial_x(\omega_2V)\|_{L^2}+\|\partial_x( v_2\omega)\|_{L^2}\\
\nonumber&\leq C\| v_2\|_{H^2}\|\nabla\partial_x\omega_2\|_{L^2}+
C\big(\|(\Delta v_2)_{\neq}\|_{L^2}+\|\partial_x\omega_2\|_{L^2}\big)
\big(\|\nabla\Delta v_2\|_{L^2}+\| V\|_{H^2}\big),
\end{align}
and
\begin{align}
\label{UNION2}&\|\Delta p\|_{L^2}+\|\nabla(V\cdot\nabla v_2)\|_{L^2}+\|P_0(V\cdot\nabla v_3)\|_{L^2}\nonumber\\&\leq C\big(\| v_2\|_{H^2}+\|P_0v_3\|_{H^1}+\|\partial_x\omega_2\|_{L^2}\big)\big(\|\nabla\Delta v_2\|_{L^2}+\|\nabla\partial_x\omega_2\|_{L^2}+\|\nabla (P_0v_3)\|_{L^2}\big)\\
\nonumber&\quad+C\big(\|(\Delta v_2)_{\neq}\|_{L^2}+\|\partial_x\omega_2\|_{L^2}\big)
\|V\|_{H^2}.
\end{align}
\end{lemma}

\begin{proof}

{\bf Step 1.} Estimates of the pressure $p$.\smallskip

Recall that
\begin{align}
\label{p1}&\Delta p=-\sum_{i,j=1}^3\partial_iv_j\partial_jv_i=-\sum_{i,j=1}^3\partial_iv_j\partial_j(v_i)_{\neq}-\sum_{i,j=1}^3\partial_iv_j\partial_jP_0v_i.
\end{align}
For $i,j\in\{1,3\}$, by Lemma \ref{lem2} and Lemma \ref{lem1}, we get
\begin{align*}
\|\partial_iv_j\partial_j(v_i)_{\neq}\|_{L^2}&\leq C\big(\|\partial_i\partial_j(v_i)_{\neq}\|_{L^2}+\|\partial_j(v_i)_{\neq}\|_{L^2}\big)\|\Delta v_j\|_{L^2}\\&\leq C\big(\|(\Delta v_2)_{\neq}\|_{L^2}+\|\partial_x\omega_2\|_{L^2}\big)\| V\|_{H^2},
\end{align*}
and for $i\in\{1,3\},j=2$, we have
\begin{align*}
\|\partial_iv_j\partial_j(v_i)_{\neq}\|_{L^2}&\leq C\big(\|\partial_i\partial_j(v_i)_{\neq}\|_{L^2}+\|\partial_j(v_i)_{\neq}\|_{L^2}\big)\|\Delta v_j\|_{L^2}\\&\leq C\big(\|(\Delta v_2)_{\neq}\|_{L^2}+\|\nabla\partial_x\omega_2\|_{L^2}\big)\|\Delta v_2\|_{L^2},
\end{align*}
and for $i=2$, by Sobolev embedding, we have
\begin{align*}
\|\partial_iv_j\partial_j(v_i)_{\neq}\|_{L^2}&\leq \|\partial_iv_j\|_{L^4}\|\partial_j(v_i)_{\neq}\|_{L^4}\leq C\| V\|_{H^2}\|\Delta(v_2)_{\neq}\|_{L^2}.
\end{align*}
This shows that
\begin{align}
\label{p2}\sum_{i,j=1}^3\|\partial_iv_j\partial_j(v_i)_{\neq}\|_{L^2}\leq C\big(\|(\Delta v_2)_{\neq}\|_{L^2}+\|\partial_x\omega_2\|_{L^2}\big)\| V\|_{H^2}+C\|\nabla\partial_x\omega_2\|_{L^2}\| v_2\|_{H^2}.
\end{align}
Similarly, we can obtain
\begin{align}
\nonumber&\sum_{i,j=1}^3\|(\partial_iv_j\partial_jP_0v_i)_{\neq}\|_{L^2}=\sum_{i,j=1}^3\|(\partial_iv_j)_{\neq}\partial_jP_0v_i\|_{L^2}\\
\label{p3}&\leq C\big(\|(\Delta v_2)_{\neq}\|_{L^2}+\|\partial_x\omega_2\|_{L^2}\big)\|P_0 V\|_{H^2}+C\|\nabla\partial_x\omega_2\|_{L^2}\| P_0v_2\|_{H^2}.
\end{align}

Due to $\partial_xP_0=0 $, we get
\begin{align*}
\sum_{i,j=1}^3\|P_0(\partial_iv_j\partial_jP_0v_i)\|_{L^2}&=\sum_{i,j=1}^3\|\partial_iP_0v_j\partial_jP_0v_i\|_{L^2}=2\|(\partial_2P_0v_2)^2\|_{L^2}+2\|\partial_2P_0v_3\partial_3P_0v_2\|_{L^2}\\
&\leq C\big(\|\partial_2P_0v_2\|_{L^4}^2+\|\partial_2P_0v_3\|_{L^2}\|\partial_3P_0v_2\|_{L^{\infty}}\big),
\end{align*}
which gives
\begin{align}
 \label{p4}\sum_{i,j=1}^3\|P_0(\partial_iv_j\partial_jP_0v_i)\|_{L^2}\leq &C\big(\|\Delta v_2\|_{L^2}\|\nabla\Delta v_2\|_{L^2}+\|P_0v_3\|_{H^1}\|\nabla\Delta v_2\|_{L^2}\big).
\end{align}

Summing up \eqref{p1}-\eqref{p4},  we deduce that
\begin{align}
\label{p5}\|(\Delta p)_{\neq}\|_{L^2}\leq& C\big(\|(\Delta v_2)_{\neq}\|_{L^2}+\|\partial_x\omega_2\|_{L^2}\big)\| V\|_{H^2}+C\|\nabla\partial_x\omega_2\|_{L^2}\| v_2\|_{H^2},\\
 \label{p6}\|\Delta p\|_{L^2}\leq&C\big(\| v_2\|_{H^2}+\|P_0v_3\|_{H^1}\big)\big(\|\nabla\Delta v_2\|_{L^2}+\|\nabla\partial_x\omega_2\|_{L^2}\big)\\
 \nonumber&+C\big(\|(\Delta v_2)_{\neq}\|_{L^2}+\|\partial_x\omega_2\|_{L^2}\big)\| V\|_{H^2} .
\end{align}

{\bf Step 2}. Estimates of  $\nabla(V\cdot\nabla v_2)$.\smallskip

We write
\begin{align}
\label{v1}&\nabla(V\cdot\nabla v_2)=\sum_{j=1}^3\nabla(v_j\partial_jv_2)=\sum_{j=1}^3\nabla((v_j)_{\neq}\partial_jv_2)+\sum_{j=1}^3\nabla(P_0v_j\partial_jv_2).
\end{align}
For $j\in\{1,3\}$, by Lemma \ref{lem2} and Lemma \ref{lem1}, we get
\begin{align*}
&\|\nabla((v_j)_{\neq}\partial_jv_2)\|_{L^2}\leq \|\nabla(v_j)_{\neq}\partial_jv_2\|_{L^2}+\|(v_j)_{\neq}\nabla\partial_j v_2\|_{L^2}\\
&\leq C\big(\|\nabla\partial_j(v_j)_{\neq}\|_{L^2}+\|\nabla(v_j)_{\neq}\|_{L^2}\big)\|\Delta v_2\|_{L^2}+C\big(\|\partial_j(v_j)_{\neq}\|_{L^2}+\|(v_j)_{\neq}\|_{L^2}\big)\|\nabla\Delta v_2\|_{L^2}\\
&\leq C\big(\|(\Delta v_2)_{\neq}\|_{L^2}+\|\nabla\partial_x\omega_2\|_{L^2}\big)\|\Delta v_2\|_{L^2}+C\big(\|(\Delta v_2)_{\neq}\|_{L^2}+\|\partial_x\omega_2\|_{L^2}\big)\|\nabla\Delta v_2\|_{L^2},
\end{align*}
and for $j=2$, we have
\begin{align*}
&\|\nabla((v_j)_{\neq}\partial_jv_2)\|_{L^2}\leq C\|(v_j)_{\neq}\|_{H^2}\|\partial_jv_2\|_{H^1}\leq C\|\Delta(v_2)_{\neq}\|_{L^2}\|\Delta v_2\|_{L^2}.
\end{align*}
This shows that
\begin{align}
\label{v2}\sum_{j=1}^3\|\nabla((v_j)_{\neq}\partial_jv_2)\|_{L^2}\lesssim& \big(\|(\Delta v_2)_{\neq}\|_{L^2}+\|\partial_x\omega_2\|_{L^2}\big)\|\nabla\Delta v_2\|_{L^2}+\|\nabla\partial_x\omega_2\|_{L^2}\| v_2\|_{H^2}.
\end{align}
For $j\in\{1,2,3\}$, we have
\begin{align}
\label{v3}
\|\nabla(P_0v_j\partial_jv_2)_{\neq}\|_{L^2}&=\|\nabla(P_0v_j\partial_j(v_2)_{\neq})\|_{L^2}\\ \nonumber&\leq C\|P_0v_j\|_{H^2}\|\partial_j(v_2)_{\neq}\|_{H^1}\leq C\|V\|_{H^2}\|\Delta(v_2)_{\neq}\|_{L^2}.
\end{align}
Due to $\partial_xP_0=0 $, we have
\begin{align}
\label{v4}
&\sum_{j=1}^3\|P_0\nabla(P_0v_j\partial_jv_2)\|_{L^2}=\sum_{j=1}^3\|\nabla(P_0v_j\partial_jP_0v_2)\|_{L^2}
\leq\sum_{j=2}^3\|P_0v_j\partial_jP_0v_2\|_{H^1}\\ \nonumber &\leq C\sum_{j=2}^3\|P_0v_j\|_{H^1}\|\partial_jP_0v_2\|_{H^2}  \leq C\big(\| v_2\|_{H^2}+\|P_0v_3\|_{H^1}\big)\|\nabla\Delta v_2\|_{L^2}.
\end{align}

Summing up \eqref{v1}-\eqref{v4}, we conclude
\begin{align}
\label{v5}
\|\nabla(V\cdot\nabla v_2)_{\neq}\|_{L^2}\leq& C\big(\|(\Delta v_2)_{\neq}\|_{L^2}+\|\partial_x\omega_2\|_{L^2}\big)\big(\| V\|_{H^2}+\|\nabla\Delta v_2\|_{L^2}\big)\\ \nonumber &+C\|\nabla\partial_x\omega_2\|_{L^2}\| v_2\|_{H^2},\\
\label{v6}\|\nabla(V\cdot\nabla v_2)\|_{L^2}\leq& C\big(\| v_2\|_{H^2}+\|P_0v_3\|_{H^1}+\|\partial_x\omega_2\|_{L^2}\big)\|\nabla\Delta v_2\|_{L^2}\\ \nonumber&+C\|\nabla\partial_x\omega_2\|_{L^2}\| v_2\|_{H^2}.
\end{align}

{\bf Step 3.} Estimate of $\partial_x(V\omega_2)$.\smallskip

We write
\begin{align*}
\partial_x(V\omega_2)=(\partial_xV)(\partial_z v_1-\partial_x v_3)+V\partial_x\omega_2.
\end{align*}
It follows from Lemma \ref{lem2} that
\begin{align*}
&\|\partial_x(V\omega_2)\|_{L^2}\leq \|\partial_xV\partial_z v_1\|_{L^2}+\|\partial_xV\partial_x v_3\|_{L^2}+\|V\partial_x\omega_2\|_{L^2}\\
 &\leq C\big(\|\partial_x\partial_zV\|_{L^2}+\|\partial_xV\|_{L^2}\big) \|\Delta v_1\|_{L^2}+ C\big(\|\partial_x^2V\|_{L^2}+\|\partial_xV\|_{L^2}\big) \|\Delta v_3\|_{L^2}+\|V\|_{L^{\infty}}\|\partial_x\omega_2\|_{L^2}\\
 &\leq C\big(\|\partial_x\partial_zV\|_{L^2}+\|\partial_xV\|_{L^2}+\|\partial_x^2V\|_{L^2}\big) \|\Delta V\|_{L^2}+C\|V\|_{H^{2}}\|\partial_x\omega_2\|_{L^2}\\
 &\leq C\big(\|\partial_x\partial_zV\|_{L^2}+\|\partial_x^2V\|_{L^2}+\|\partial_x\omega_2\|_{L^2}\big)\|V\|_{H^{2}},
\end{align*}
from which and  Lemma \ref{lem1}, we infer that
\begin{align}
\label{om1}&\|\partial_x(V\omega_2)\|_{L^2}\leq C\big(\|(\Delta v_2)_{\neq}\|_{L^2}+\|\partial_x\omega_2\|_{L^2}\big)\|V\|_{H^{2}}.
\end{align}

{\bf Step 4.} Estimate of $ \partial_x(\omega v_2)$.\smallskip

We have
\begin{align*}
 \partial_x(\omega v_2)=(\partial_x\omega)v_2+\omega\partial_xv_2,\quad\partial_xv_2=\partial_x(v_2)_{\neq}.
\end{align*}
It follows from Lemma \ref{lem1} that
\begin{align}
\nonumber\|\partial_x(\omega v_2)\|_{L^2}\leq &\|(\partial_x\omega)v_2\|_{L^2}+\|\omega\partial_xv_2\|_{L^2}\\
\nonumber\leq &\|\partial_x\omega\|_{L^2}\|v_2\|_{L^{\infty}}+C\|\omega\|_{H^1}\|\partial_x(v_2)_{\neq}\|_{H^1}\\
\nonumber \leq& C\big(\|(\Delta v_2)_{\neq}\|_{L^2}+\|\nabla\partial_x\omega_2\|_{L^2}\big) \| v_2\|_{H^2}+ C\|V\|_{H^2}\|(v_2)_{\neq}\|_{H^2}\\
\label{om2}\leq &C\|(\Delta v_2)_{\neq}\|_{L^2}\|V\|_{H^2}+C\|\nabla\partial_x\omega_2\|_{L^2}\| v_2\|_{H^2}.
\end{align}

Then \eqref{UNION1} follows from \eqref{p5},\eqref{v5},\eqref{om1} and \eqref{om2}.\smallskip

{\bf Step 5.} Estimate of  $P_0(V\cdot\nabla v_3)$.\smallskip

We write
\begin{align*}
P_0(V\cdot\nabla v_3)=P_0V\cdot\nabla P_0v_3+P_0(V_{\neq}\cdot\nabla (v_3)_{\neq}).
\end{align*}
Thanks to $\partial_xP_0=0 $ and $\partial_yP_0 v_2+\partial_zP_0 v_3=0 $, we get
\begin{align*}
&P_0v\cdot\nabla P_0v_3=P_0v_2\partial_y P_0v_3+P_0v_3\partial_z P_0v_3=P_0v_2\partial_y P_0v_3-P_0v_3\partial_y P_0v_2,
\end{align*}
which gives
\begin{align}
\nonumber\|P_0v\cdot\nabla P_0v_3\|_{L^2}\leq &\|P_0v_2\|_{L^{\infty}}\|\partial_y P_0v_3\|_{L^2}+\|P_0v_3\|_{L^4}\|\partial_y P_0v_2\|_{L^4}\\
\label{5.60}\leq &C\|v_2\|_{H^{2}}\|\partial_yP_0v_3\|_{L^2}+C\|P_0v_3\|_{H^1}\|\nabla \Delta v_2\|_{L^2}.
\end{align}

For $j\in\{1,3\}$, by Lemma \ref{lem1},  we have
\begin{align*}
&\|(v_j)_{\neq}\partial_j (v_3)_{\neq}\|_{L^2}\leq \|(v_j)_{\neq}\|_{L^{\infty}}\|\partial_j(v_3)_{\neq}\|_{L^2}\leq C\|V\|_{H^2}\big(\|(\Delta v_2)_{\neq}\|_{L^2}+\|\partial_x\omega_2\|_{L^2}\big),
\end{align*}
and for $j=2$, we have
\begin{align*}
&\|(v_j)_{\neq}\partial_j (v_3)_{\neq}\|_{L^2}\leq \|(v_j)_{\neq}\|_{L^{\infty}}\|\partial_j(v_3)_{\neq}\|_{L^2}\leq C\|v_2\|_{H^2}\big(\|(\Delta v_2)_{\neq}\|_{L^2}+\|\nabla\partial_x\omega_2\|_{L^2}\big).
\end{align*}
This shows that
\begin{align*}
&\|V_{\neq}\cdot\nabla (v_3)_{\neq}\|_{L^2}\leq C\|V\|_{H^2}\big(\|(\Delta v_2)_{\neq}\|_{L^2}+\|\partial_x\omega_2\|_{L^2}\big)+C\|v_2\|_{H^2}\|\nabla\partial_x\omega_2\|_{L^2}.
\end{align*}
Thus, we obtain
\begin{align}
\label{v7}&\|P_0(V\cdot\nabla v_3)\|_{L^2}\leq \|P_0V\cdot\nabla P_0v_3\|_{L^2}+\|V_{\neq}\cdot\nabla (v_3)_{\neq}\|_{L^2}\\
\nonumber &\leq C\|v_2\|_{H^{2}}\|\partial_2 P_0v_3\|_{L^2}+\|P_0v_3\|_{H^1}\|\nabla \Delta v_2\|_{L^2}+C\|V\|_{H^2}\big(\|(\Delta v_2)_{\neq}\|_{L^2}+\|\partial_x\omega_2\|_{L^2}\big)\\
\nonumber&\quad+C\|v_2\|_{H^2}\|\nabla\partial_x\omega_2\|_{L^2}.
\end{align}

Then \eqref{UNION1} follows from \eqref{p6}, \eqref{v6} and \eqref{v7}.
\end{proof}

\begin{lemma}
\label{lem3}
It holds that
\begin{align*}
&\|\nabla(V\cdot\nabla V)_{\neq}\|_{L^2}+\|V_{\neq}\cdot\nabla V_{\neq}\|_{H^1}\leq C\|V\|_{H^2}\|\Delta V_{\neq}\|_{L^2}.
\end{align*}
\end{lemma}
\begin{proof}
Thanks to $V\cdot\nabla V=V\cdot\nabla V_{\neq}+V\cdot\nabla P_0V$, we have
\begin{align*}
\|\nabla(V\cdot\nabla V)_{\neq}\|_{L^2}\leq&\|V\cdot\nabla V_{\neq}\|_{H^1}+\|V_{\neq}\cdot\nabla P_0V\|_{H^1}\\\leq& C\|V\|_{H^2}\|\nabla V_{\neq}\|_{H^1}+C\|V_{\neq}\|_{H^2}\|\nabla P_0V\|_{H^1}\leq C\|V\|_{H^2}\|\Delta V_{\neq}\|_{L^2}.
\end{align*}
Similarly, we have
\begin{align*}
& \|V_{\neq}\cdot\nabla V_{\neq}\|_{H^1}\leq C\|V_{\neq}\|_{H^2}\|\nabla V_{\neq}\|_{H^1}\leq C\|V_{\neq}\|_{H^2}^2\leq C\|V\|_{H^2}\|\Delta V_{\neq}\|_{L^2}.
\end{align*}
This completes the proof.
\end{proof}

\subsection{Nonlinear stability}

Let us introduce the following notations
\begin{align*}
&M_0(T)=\sup_{0\leq t\leq T}\big(\| v_2(t)\|_{H^2}+e^{c'\sqrt{\gamma}t}\|(\Delta v_2)_{\neq}(t)\|_{L^2}+e^{c'\sqrt{\gamma}t}\|\partial_x\omega_2(t)\|_{L^2}+\| P_0v_3(t)\|_{H^1}\big),\\
&M_1(T)=\sup_{0\leq t\leq T}\| V(t)\|_{H^2},\quad \|f\|_{Y_0}^2:=\|f\|_{L^{\infty}L^2}^2+\nu\|\nabla f\|_{L^{2}L^2}^2.
\end{align*}

\begin{theorem} \label{Thm:non}
Given $k_f\in (0,1),$ there exist constants $a_0,\varepsilon_1\in(0,1),\ C_1>1, $ such that for any $0< |\gamma|<1,\ 0<\nu\leq a_0|\gamma|^{\f12},\ T>0$ , if the solution $U=V+U^{*}\in C([0,T],H^2)\cap L^2([0,T],H^3)$ of \eqref{eq:NS}  satisfies
\begin{align*}
M_0(T)< \varepsilon_1\nu,\quad M_1(T)< \varepsilon_1\nu^{\frac{1}{2}}|\gamma|^{\frac{1}{4}},
\end{align*}
then there hold
\begin{align}
\label{Nonlinear bound1}&\|\nabla P_0v_3\|_{Y_0}^2+\|\Delta v_2\|_{Y_0}^2+\|(\Delta v_2)_{\neq}\|_{X_{c'}}^2+\|\partial_x\omega_2\|_{X_{c'}}^2+\nu^{-1}\|\Delta p\|_{L^2L^2}^2\leq C\|V_0\|_{H^2}^2,\\
\label{Nonlinear bound2}&\|\Delta V_{\neq}\|_{X_{c'}}^2+\|\Delta P_0v_{3}\|_{Y_0}^2\leq  {C|\gamma|^{\frac{1}{2}}}/{\nu^2 }\|V_0\|_{H^2}^2,\\
\label{Nonlinear bound3}&M_0(T)\leq  C_1\|V_0\|_{H^2},\quad M_1(T)\leq  C_1(|\gamma|/\nu^2)\|V_0\|_{H^2}.
\end{align}
\end{theorem}

\begin{proof}

{\bf Step 1.} Proof of  \eqref{Nonlinear bound1}.\smallskip

Recall  the definitions of $f,g,p$ in \eqref{NLNSK1} and \eqref{NLNSK2}. Then by Lemma \ref{Union},  we have
\begin{align*}
&\|e^{c'\sqrt{\gamma}t}\partial_x g\|_{L^{2}L^2}^2+\|e^{c'\sqrt{\gamma}t} f_{\neq}\|_{L^{2}L^2}^2+\|e^{c'\sqrt{\gamma}t}(\Delta p)_{\neq}\|_{L^2L^2}^2\\
&\le C\big(\|e^{c'\sqrt{\gamma}t}(\Delta p)_{\neq}\|_{L^2L^2}^2+\|e^{c'\sqrt{\gamma}t}\nabla(V\cdot\nabla v_2)_{\neq}\|_{L^2L^2}^2+\|e^{c'\sqrt{\gamma}t}\partial_x(V\omega_2)\|_{L^2L^2}^2\\&\qquad+\|e^{c'\sqrt{\gamma}t}\partial_x(\omega v_2)\|_{L^2L^2}^2\big)
\\
&\leq C\| v_2\|_{L^{\infty}H^2}^2\|e^{c'\sqrt{\gamma}t}\nabla\partial_x\omega_2\|_{L^2L^2}^2+
C\big(\|e^{c'\sqrt{\gamma}t}(\Delta v_2)_{\neq}\|_{L^2L^2}^2+\|e^{c'\sqrt{\gamma}t}\partial_x\omega_2\|_{L^2L^2}^2\big)\|V\|_{L^{\infty}H^2}^2\\
&\qquad+C\big(\|e^{c'\sqrt{\gamma}t}(\Delta v_2)_{\neq}\|_{L^{\infty}L^2}^2+\|e^{c'\sqrt{\gamma}t}\partial_x\omega_2\|_{L^{\infty}L^2}^2\big)
\|\nabla\Delta v_2\|_{L^2L^2}^2,
\end{align*}
which gives
\begin{align}
\label{Nonlinear Key1}&\|e^{c'\sqrt{\gamma}t}\partial_1 g\|_{L^{2}L^2}^2+\|e^{c'\sqrt{\gamma}t} f_{\neq}\|_{L^{2}L^2}^2+\|e^{c'\sqrt{\gamma}t}(\Delta p)_{\neq}\|_{L^2L^2}^2\\
\nonumber&\leq CM_0^2\nu^{-1}\|\partial_x\omega_2\|_{X_{c'}}^2+
C|\gamma|^{-\frac{1}{2}}\big(\|(\Delta v_2)_{\neq}\|_{X_{c'}}^2+\|\partial_x\omega_2\|_{X_{c'}}^2\big)M_1^2+CM_0^2
\nu^{-1}\|\Delta v_2\|_{Y_0}^2.
\end{align}

It follows from Proposition \ref{Prop5.3}, Proposition \ref{Prop5.4} and \eqref{Nonlinear Key1} that
 \begin{align}
 \label{Nonlinear Key3}&\|(\Delta v_2)_{\neq}\|_{X_{c'}}^2+\|\partial_x\omega_2\|_{X_{c'}}^2\\
 \nonumber&\leq C\| \partial_x\omega_2(0)\|_{L^2}^2+C\| (\Delta v_2)_{\neq}(0)\|_{L^2}^2+C\nu^{-1}\|e^{c'\sqrt{\gamma}t}\partial_1 g\|_{L^{2}L^2}^2+C\nu^{-1}\|e^{c'\sqrt{\gamma}t} f_{\neq}\|_{L^{2}L^2}^2\\
 \nonumber&\leq C\Big[\| V(0)\|_{L^2}^2+\nu^{-1}\Big(M_0^2\nu^{-1}\|\partial_x\omega_2\|_{X_{c'}}^2+|\gamma|^{-\frac{1}{2}}(\|(\Delta v_2)_{\neq}\|_{X_{c'}}^2+\|\partial_x\omega_2\|_{X_{c'}}^2)M_1^2\\
\nonumber &\qquad+M_0^2\nu^{-1}\|\Delta v_2\|_{Y_0}^2\Big)\Big].
\end{align}

It is easy to find that
\begin{align*}
&(\partial_t-\nu\Delta)P_0v_3=-P_0(V\cdot\nabla v_3)-P_0\partial_3p,\\
&(\partial_t-\nu\Delta) P_0\Delta v_2=-\Delta P_0(V\cdot\nabla v_2)-\partial_2 P_0\Delta p=\text{div}P_0f.
\end{align*}
Then we can deduce that
\beno
&&\|P_0\Delta v_2\|_{Y_0}^2\leq C\big(\|P_0\Delta v_2(0)\|_{L^2}^2+\nu^{-1}\|P_0f\|_{L^2L^2}^2\big),\\
&&\|\nabla P_0v_3\|_{Y_0}^2\leq C\big(\|\nabla P_0v_3(0)\|_{L^2}^2+\nu^{-1}\|P_0(V\cdot\nabla v_3)\|_{L^2L^2}^2+\nu^{-1}\|P_0\partial_3p\|_{L^2L^2}^2\big),
\eeno
from which and Lemma \ref{Union}, we infer that
\begin{align*}
&\|\nabla P_0v_3\|_{Y_0}^2+\|P_0\Delta v_2\|_{Y_0}^2+\nu^{-1}\|\Delta p\|_{L^2L^2}^2\\
&\leq C\big(\|\nabla P_0v_3(0)\|_{L^2}^2+\nu^{-1}\|P_0\partial_3p\|_{L^2L^2}^2+\nu^{-1}\|P_0(V\cdot\nabla v_3)\|_{L^2L^2}^2+\|\Delta v_2(0)\|_{L^2}^2\\
&\quad+\nu^{-1}\|f\|_{L^2L^2}^2\big)+\nu^{-1}\|\Delta p\|_{L^2L^2}^2\\
&\leq C\|V_0\|_{H^2}^2+C\nu^{-1}\big(\|\Delta p\|_{L^2L^2}^2+\|\nabla(V\cdot\nabla v_2)\|_{L^2L^2}^2+\|P_0(V\cdot\nabla v_3)\|_{L^2L^2}^2\big)\\
&\leq C\|V_0\|_{H^2}^2+C\nu^{-1}\big(\| v_2\|_{L^{\infty}H^2}^2+\|P_0v_3\|_{L^{\infty}H^1}^2+\|\partial_x\omega_2\|_{L^{\infty}L^2}^2\big)\big(\|\nabla\Delta v_2\|_{L^2L^2}^2+\|\nabla\partial_x\omega_2\|_{L^2L^2}^2\\
&\qquad+\|\nabla (P_0v_3)\|_{L^2L^2}^2\big)+C\nu^{-1}\big(\|(\Delta v_2)_{\neq}\|_{L^2L^2}^2+\|\partial_x\omega_2\|_{L^2L^2}^2\big)\| V\|_{L^{\infty}H^2}^2,
\end{align*}
which gives
\begin{align}
\label{Nonlinear Key2}&\|\nabla P_0v_3\|_{Y_0}^2+\|P_0\Delta v_2\|_{Y_0}^2+\nu^{-1}\|\Delta p\|_{L^2L^2}^2\\
\nonumber&\leq C \Big[\|V(0)\|_{H^2}^2+\nu^{-1}M_0^2\nu^{-1}\big(\|\Delta v_2\|_{Y_0}^2+\|\nabla( P_0v_3)\|_{Y_0}^2+\|\partial_x\omega_2\|_{X_{c'}}^2\big)\\
\nonumber&+\nu^{-1}|\gamma|^{-\frac{1}{2}}\big(\|(\Delta v_2)_{\neq}\|_{X_{c'}}^2+\|\partial_x\omega_2\|_{X_{c'}}^2\big)M_1^2\Big].
\end{align}

Using the fact that
\begin{align*}
\|\Delta v_2\|_{Y_0}\leq \|P_0\Delta v_2\|_{Y_0}+\|(\Delta v_2)_{\neq}\|_{Y_0}\leq \|P_0\Delta v_2\|_{Y_0}+\|(\Delta v_2)_{\neq}\|_{X_{c'}},
\end{align*}
and \eqref{Nonlinear Key3}-\eqref{Nonlinear Key2}, we obtain
\begin{align*}
&\|\nabla (P_0v_3)\|_{Y_0}^2+\|\Delta v_2\|_{Y_0}^2+\|(\Delta v_2)_{\neq}\|_{X_{c'}}^2+\|\partial_x\omega_2\|_{X_{c'}}^2+\nu^{-1}\|\Delta p\|_{L^2L^2}^2\\
 &\leq  C\|V_0\|_{H^2}^2+C\nu^{-1}M_0^2\nu^{-1}\big(\|\Delta v_2\|_{Y_0}^2+\|\partial_x\omega_2\|_{X_{c'}}^2+\|\nabla (P_0v_3)\|_{Y_0}^2\big)\\
 &\quad+C\nu^{-1}|\gamma|^{-\frac{1}{2}}\big(\|(\Delta v_2)_{\neq}\|_{X_{c'}}^2+\|\partial_x\omega_2\|_{X_{c'}}^2\big)M_1^2+C\nu^{-1}\Big[M_0^2\nu^{-1}\|\partial_x\omega_2\|_{X_{c'}}^2\\
 &\quad+|\gamma|^{-\frac{1}{2}}\big(\|(\Delta v_2)_{\neq}\|_{X_{c'}}^2+\|\partial_x\omega_2\|_{X_{c'}}^2\big)M_1^2+M_0^2
\nu^{-1}\|\Delta v_2\|_{Y_0}^2\Big]\\
&\leq C\|V_0\|_{H^2}^2+C(M_0^2\nu^{-2}+\nu^{-1}|\gamma|^{-\frac{1}{2}}M_1^2)\Big[\|\Delta v_2\|_{Y_0}^2+\|(\Delta v_2)_{\neq}\|_{X_{c'}}^2\\
&\qquad+\|\partial_x\omega_2\|_{X_{c'}}^2+\|\nabla (P_0v_3)\|_{Y_0}^2\Big].\end{align*}
Therefore,  there exists  $\varepsilon_1\in(0,1)$ so that if
 \begin{align*}
 M_0\leq \varepsilon_1\nu,\quad M_1\leq \varepsilon_1\nu^{\frac{1}{2}}|\gamma|^{\frac{1}{4}},
\end{align*}
then
\begin{align*}
&\|\nabla (P_0v_3)\|_{Y_0}^2+\|\Delta v_2\|_{Y_0}^2+\|(\Delta v_2)_{\neq}\|_{X_{c'}}^2+\|\partial_x\omega_2\|_{X_{c'}}^2+\nu^{-1}\|\Delta p\|_{L^2L^2}^2\leq C\|V_0\|_{H^2}^2.
\end{align*}

{\bf Step 2.} Estimate of $M_0$.\smallskip

In order to estimate $M_0$, we need the momentum conservation. Integrating \eqref{eq:NS-V} in $ \mathbb{T}_{2\pi}\times\mathbb{T}_{2\pi/k_f}\times\mathbb{T}_{2\pi}$, we have
\begin{align*}
\partial_t\langle v_i,1\rangle=0,
\end{align*}
which implies that
\begin{align*}
a_i:=\langle v_i(t),1\rangle/\langle 1,1\rangle
\end{align*}
is a constant and $|a_i|\leq C\|V_0\|_{H^2}$ for $i=1,2,3.$ This along with \eqref{Nonlinear bound1} gives
\begin{align*}
&\|v_2(t)\|_{H^2}\leq C\big(|a_2|+\|\Delta v_2(t)\|_{L^2})\leq C(\|V_0\|_{H^2}+\|\Delta v_2\|_{Y_0}\big)\leq C\|V_0\|_{H^2},\\
&\|P_0v_3(t)\|_{H^1}\leq C\big(|a_3|+\|\nabla (P_0v_3)(t)\|_{L^2}\big)\leq C\big(\|V_0\|_{H^2}+\|\nabla (P_0v_3)\|_{Y_0}\big)\leq C\|V_0\|_{H^2},\\
&e^{c'\sqrt{\gamma}t}\|(\Delta v_2)_{\neq}(t)\|_{L^2}+e^{c'\sqrt{\gamma}t}\|\partial_x\omega_2(t)\|_{L^2}\leq \|(\Delta v_2)_{\neq}\|_{X_{c'}}+\|\partial_x\omega_2\|_{X_{c'}}\leq C\|V_0\|_{H^2},
\end{align*}
which shows
 \begin{align*}
M_0\leq C\|V_0\|_{H^2}.
\end{align*}

{\bf Step 3}. Proof of  \eqref{Nonlinear bound2}.\smallskip

First of all, we find that
\begin{align*}
&\big(\partial_t+\f{\gamma}{\nu k_f^2}\sin(k_fy)\partial_x-\nu\Delta\big)\Delta v_3\\
&=2\partial_3\Big(\f{\gamma}{\nu k_f}\cos(k_fy)\partial_xv_2\Big)-\Delta(V\cdot\nabla v_3)-\partial_z\Delta p-\f{2\gamma}{\nu k_f}\cos(k_fy)\partial_x\partial_yv_3+\f{\gamma}{\nu }\sin(k_fy)\partial_xv_3\\
&:=\text{div}g_{2,3},
\end{align*}
where
\begin{align*}
g_{2,3}=-\nabla(V\cdot\nabla v_3)-\Big(\f{\gamma}{\nu }\sin(k_fy)v_3,\f{2\gamma}{\nu k_f}\cos(k_fy)\partial_xv_3,\Delta p-\f{2\gamma}{\nu k_f}\cos(k_fy)\partial_xv_2\Big),
\end{align*}
and
\begin{align*}
&\big(\partial_t+\f{\gamma\sin(k_fy)}{\nu k_f^2}\partial_x-\nu\Delta\big)\Delta v_1\\
&=-\Delta\big(\f{\gamma\cos(k_fy)}{\nu k_f}v_2\big)+2\partial_x\big(\f{\gamma\cos(k_fy)}{\nu k_f}\partial_xv_2\big)-\partial_y\big(\f{2\gamma}{\nu k_f}\cos(k_fy)\partial_xv_1\big)-\partial_x\f{\gamma}{\nu }\sin(k_fy)v_1\\
&-\Delta(V\cdot\nabla v_1)-\partial_x\Delta p:=\text{div}g_{2,1},
\end{align*}
where
\begin{align*}
 g_{2,1}=&-\nabla\big(\f{\gamma\cos(k_fy)}{\nu k_f}v_2\big)-\Big(\Delta p-\f{2\gamma}{\nu k_f}\cos(k_fy)\partial_xv_2+\f{\gamma}{\nu }\sin(k_fy)v_1,\f{2\gamma}{\nu k_f}\cos(k_fy)\partial_xv_1,0\Big)\\&-\nabla(V\cdot\nabla v_1).
 \end{align*}

Thanks to Lemma \ref{lem1} and Lemma \ref{lem3}, we have
\begin{align}
\nonumber\| (g_{2,3})_{\neq}\|_{L^2}&\leq \| \nabla(V\cdot\nabla v_3)_{\neq}\|_{L^2}+\f{|\gamma|}{\nu }\| (v_3)_{\neq}\|_{L^2}+\f{2|\gamma|}{\nu k_f}\| \partial_xv_3\|_{L^2}+\f{2|\gamma|}{\nu k_f}\| \partial_xv_2\|_{L^2}+\| (\Delta p)_{\neq}\|_{L^2}\\
\nonumber&\leq \| \nabla(V\cdot\nabla v_3)_{\neq}\|_{L^2}+\f{C|\gamma|}{\nu }\| \partial_xv_3\|_{L^2}+\f{2|\gamma|}{\nu k_f}\| \partial_xv_2\|_{L^2}+\| (\Delta p)_{\neq}\|_{L^2}\\
\label{g2,3}&\leq C\|V\|_{H^2}\|\Delta V_{\neq}\|_{L^2}+\f{C|\gamma|}{\nu }\big(\|\big(\Delta v_2)_{\neq}\|_{L^2}+\|\partial_x\omega_2\|_{L^2}\big)+\| (\Delta p)_{\neq}\|_{L^2},
\end{align}
and
\begin{align}
\nonumber\| (g_{2,1})_{\neq}\|_{L^2}\leq &\f{|\gamma|}{\nu k_f}\| \nabla(\cos(k_fy) v_2)_{\neq}\|_{L^2}+\| \nabla(V\cdot\nabla v_1)_{\neq}\|_{L^2}+\f{|\gamma|}{\nu }\| (v_1)_{\neq}\|_{L^2}+\f{2|\gamma|}{\nu k_f}\| \partial_xv_1\|_{L^2}\\
\nonumber&+\f{2|\gamma|}{\nu k_f}\| \partial_xv_2\|_{L^2}+\| (\Delta p)_{\neq}\|_{L^2}\\
\nonumber\leq &\| \nabla(V\cdot\nabla v_1)_{\neq}\|_{L^2}+\f{C|\gamma|}{\nu }\| \partial_xv_1\|_{L^2}+\f{C|\gamma|}{\nu }\|(v_2)_{\neq}\|_{H^1}+\| (\Delta p)_{\neq}\|_{L^2}\\
\label{g2,1}\leq &C\|V\|_{H^2}\|\Delta V_{\neq}\|_{L^2}+\f{C|\gamma|}{\nu }\big(\|(\Delta v_2)_{\neq}\|_{L^2}+\|\partial_x\omega_2\|_{L^2}\big)+\| (\Delta p)_{\neq}\|_{L^2}.
\end{align}
Then it follows from \eqref{NLNSK decay3}, \eqref{g2,1} and \eqref{g2,3} that
\begin{align*}
&\|(\Delta v_1)_{\neq}\|_{X_{c'}}^2+\|(\Delta v_3)_{\neq}\|_{X_{c'}}^2\\
&\leq C\|(\Delta v_1)_{\neq}(0)\|_{L^2}^2+C\| (\Delta v_3)_{\neq}(0)\|_{L^2}^2+C\nu^{-1}\|e^{c'\sqrt{\gamma}t}(g_{2,1})_{\neq}\|_{L^{2}L^2}^2+C\nu^{-1}\|e^{c'\sqrt{\gamma}t}(g_{2,3})_{\neq}\|_{L^{2}L^2}^2\\
&\leq  C\|V_0\|_{H^2}^2+C\nu^{-1}\|V\|_{L^{\infty}H^2}^2\|e^{c'\sqrt{\gamma}t}\Delta V_{\neq}\|_{L^2L^2}^2+C\nu^{-1}\| e^{c'\sqrt{\gamma}t}(\Delta p)_{\neq}\|_{L^2L^2}^2\\
&\qquad+\f{C|\gamma|}{\nu^2 }\big(\|e^{c'\sqrt{\gamma}t}(\Delta v_2)_{\neq}\|_{L^2L^2}^2+\|e^{c'\sqrt{\gamma}t}\partial_x\omega_2\|_{L^2L^2}^2\big)\\
&\leq  C\|V_0\|_{H^2}^2+C\nu^{-1}M_1^2|\gamma|^{-\frac{1}{2}}\|\Delta V_{\neq}\|_{X_{c'}}^2+\f{C|\gamma|^{\frac{1}{2}}}{\nu^2 }\big(\|(\Delta v_2)_{\neq}\|_{X_{c'}}^2+\|\partial_x\omega_2\|_{X_{c'}}^2\big)\\
&\qquad+C\nu^{-1}\Big[M_0^2\nu^{-1}\|\partial_x\omega_2\|_{X_{c'}}^2+
|\gamma|^{-\frac{1}{2}}\big(\|(\Delta v_2)_{\neq}\|_{X_{c'}}^2+\|\partial_x\omega_2\|_{X_{c'}}^2\big)M_1^2+M_0^2
\nu^{-1}\|\Delta v_2\|_{Y_0}^2\Big],
\end{align*}
which gives
\begin{align}
\label{5.66}&
\|(\Delta v_1)_{\neq}\|_{X_{c'}}^2+\|(\Delta v_3)_{\neq}\|_{X_{c'}}^2\\
\nonumber&\leq  C\|V_0\|_{H^2}^2+C\nu^{-1}M_1^2|\gamma|^{-\frac{1}{2}}\|\Delta V_{\neq}\|_{X_{c'}}^2+\f{C|\gamma|^{\frac{1}{2}}+M_0^2}{\nu^2 }\|V_0\|_{H^2}^2.
\end{align}
Thanks to \eqref{Nonlinear bound1}, we know
\begin{align*}
\|(\Delta v_2)_{\neq}\|_{X_{c'}}^2\leq  C\|V_0\|_{H^2}^2,
\end{align*}
which along with \eqref{5.66} ensures that
\begin{align*}
\|\Delta V_{\neq}\|_{X_{c'}}^2\leq& C\big(\|(\Delta v_1)_{\neq}\|_{X_{c'}}^2+\|(\Delta v_2)_{\neq}\|_{X_{c'}}^2+\|(\Delta v_3)_{\neq}\|_{X_{c'}}^2\big)\\
\leq & C\|V_0\|_{H^2}^2+C\nu^{-1}M_1^2|\gamma|^{-\frac{1}{2}}\|\Delta V_{\neq}\|_{X_{c'}}^2+\f{C|\gamma|^{\frac{1}{2}}+M_0^2}{\nu^2 }\|V_0\|_{H^2}^2.
\end{align*}
Hence, if $M_1\leq \varepsilon_1\nu^{\frac{1}{2}}|\gamma|^{\frac{1}{4}}$ for some constant $\varepsilon_1\in(0,1)$, then
\begin{align}
\label{Nonlinear5.67}&\|\Delta V_{\neq}\|_{X_{c'}}^2\leq C\|V_0\|_{H^2}^2+\f{C|\gamma|^{\frac{1}{2}}+M_0^2}{\nu^2 }\|V_0\|_{H^2}^2\leq \f{C|\gamma|^{\frac{1}{2}}}{\nu^2 }\|V_0\|_{H^2}^2.
\end{align}

Recall that
\begin{align*}
(\partial_t-\nu\Delta)\Delta P_0v_3=-\Delta P_0(V\cdot\nabla v_3)-\partial_3 P_0\Delta p.
\end{align*}
A direct calculation gives
\begin{align*}
&P_0(V\cdot\nabla v_3)=P_0V\cdot\nabla P_0v_3+P_0(V_{\neq}\cdot\nabla (v_3)_{\neq}),\\
&P_0V\cdot\nabla P_0v_3=P_0v_2\partial_yP_0v_3+P_0v_3\partial_z P_0v_3.
\end{align*}
It is easy to see that
\begin{align*}
\|P_0V\cdot\nabla P_0v_3\|_{H^1}
\leq C\big(\|v_2\|_{H^{2}}+\|P_0v_3\|_{H^1}\big)\|\nabla \Delta P_0v_3\|_{L^2},
\end{align*}
from which and Lemma \ref{lem3}, we infer that
\begin{align*}
\|\nabla P_0(V\cdot\nabla v_3)\|_{L^2}&\leq \|P_0V\cdot\nabla P_0v_3\|_{H^1}+\|V_{\neq}\cdot\nabla (v_3)_{\neq}\|_{H^1}\\
&\leq  C\big(\|v_2\|_{H^{2}}+\|P_0v_3\|_{H^1}\big)\|\nabla \Delta P_0v_3\|_{L^2}+C\|V\|_{H^2}\|\Delta V_{\neq}\|_{L^2}.
\end{align*}
Then we have
\begin{align*}
\|\Delta P_0v_3\|_{Y_0}^2\leq& C\big(\|\Delta P_0v_3(0)\|_{L^2}^2+\nu^{-1}\|\nabla P_0(V\cdot\nabla v_3)\|_{L^2L^2}^2+\nu^{-1}\|P_0\Delta p\|_{L^2L^2}^2\big)\\
\leq& C\| V_0\|_{H^2}^2+C\nu^{-1}\big(\|v_2\|_{L^{\infty}H^{2}}^2+\|P_0v_3\|_{L^{\infty}H^1}^2\big)\|\nabla \Delta P_0v_3\|_{L^2L^2}^2\\
&+C\nu^{-1}\|V\|_{L^{\infty}H^2}^2\|\Delta V_{\neq}\|_{L^2L^2}^2+C\nu^{-1}\|\Delta p\|_{L^2L^2}^2\\
\leq& C\|V_0\|_{H^2}^2+C\nu^{-1}M_0^2\nu^{-1}\|\Delta P_0v_3\|_{Y_0}^2+C\nu^{-1}M_1^2|\gamma|^{-\frac{1}{2}}\|\Delta V_{\neq}\|_{Y_0}^2+C\| V_0\|_{H^2}^2.
\end{align*}
Hence, if $M_0\leq \varepsilon_1\nu,\ M_1\leq \varepsilon_1\nu^{\frac{1}{2}}|\gamma|^{\frac{1}{4}} $ for some constant $\varepsilon_1\in(0,1), $ then
\begin{align}
\label{Nonlinear5.68}&\|\Delta P_0v_3\|_{Y_0}^2\leq  C\|V_0\|_{H^2}^2+C\|\Delta V_{\neq}\|_{Y_0}^2\leq \f{C|\gamma|^{\frac{1}{2}}}{\nu^2 }\|V_0\|_{H^2}^2,
\end{align}
which along with \eqref{Nonlinear5.67} gives \eqref{Nonlinear bound2}.
\smallskip

{\bf Step 4.} Estimate of $M_1$.\smallskip

By \eqref{Nonlinear bound2} and $|a_3|\le \|V_0\|_{H^2}$, we get
\begin{align}
\nonumber\|v_3(t)\|_{H^2}^2\leq& C\big(|a_3|+\|\Delta P_0v_3(t)\|_{L^2}+\|(\Delta v_3)_{\neq}(t)\|_{L^2}\big)^2\\
\label{v3,h2}\leq& C\|V_0\|_{H^2}^2+\f{C|\gamma|^{\frac{1}{2}}}{\nu^2 }\| V_0\|_{H^2}^2\leq \f{C|\gamma|^{\frac{1}{2}}}{\nu^2 }\|V_0\|_{H^2}^2.
\end{align}

The estimate of $P_0v_1$ is very subtle.  Recall that
\begin{align*}
&(\partial_t-\nu\Delta)P_0v_1+\f{\gamma\cos(k_fy)}{\nu k_f}P_0v_2=-P_0(V\cdot\nabla v_1)\\
&\quad=-P_0v_2\partial_yP_0v_1-P_0v_3\partial_zP_0v_1-P_0(V_{\neq}\cdot\nabla (v_1)_{\neq}).
\end{align*}

Let $v_1^{(1)}=-\dfrac{\nu k_f \cos(k_fy)+a_2\sin(k_fy)}{(\nu k_f)^2+a_2^2}\dfrac{\gamma a_2}{\nu k_f^2}$. Then we find that \begin{align*}
&-\nu\Delta v_1^{(1)}+a_2\partial_yv_1^{(1)}+\f{\gamma\cos(k_fy)}{\nu k_f}a_2=0,
\end{align*}
and
\begin{align}
\label{v11}&\| v_1^{(1)}\|_{H^2}=\dfrac{\| \cos(k_fy)\|_{H^2}}{((\nu k_f)^2+a_2^2)^{\frac{1}{2}}}\dfrac{|\gamma a_2|}{\nu k_f^2}\leq \dfrac{C|\gamma a_2|}{\nu^2}\leq \dfrac{C|\gamma|}{\nu^2}\| V_0\|_{H^2}\leq \dfrac{C|\gamma|}{\nu^2}M_0\leq \dfrac{C|\gamma|}{\nu}.\end{align}
Let $P_0v_1=v_1^{(1)}+v_1^{(2)} $. Due to $\partial_zv_1^{(1)}=0$, we have
\begin{align*}
&(\partial_t-\nu\Delta)v_1^{(2)}+\f{\gamma\cos(k_fy)}{\nu k_f}(P_0v_2-a_2)\\
&=-P_0v_2\partial_yv_1^{(2)}-(P_0v_2-a_2)\partial_yv_1^{(1)}-P_0v_3\partial_zv_1^{(2)}-P_0(V_{\neq}\cdot\nabla (v_1)_{\neq}),\end{align*}
which gives
\begin{align*}
&(\partial_t-\nu\Delta)\Delta v_1^{(2)}=\text{div} f_1,
\end{align*}
where
\begin{align*}
 f_1=&-\nabla\big[(\f{\gamma\cos(k_fy)}{\nu k_f}+\partial_yv_1^{(1)})(P_0v_2-a_2)\big]-\nabla (P_0v_2\partial_y v_1^{(2)}+ P_0v_3\partial_zv_1^{(2)})\\&-\nabla P_0(V_{\neq}\cdot\nabla (v_1)_{\neq}).
\end{align*}
Using the fact that $\| P_0v_2-a_2\|_{H^2}\leq  C\|\nabla\Delta v_2\|_{L^2} $ and Lemma \ref{lem3}, we infer that
\begin{align*}
\|f_1\|_{L^2}\leq& \big\|(\f{\gamma\cos(k_fy)}{\nu k_f}+\partial_yv_1^{(1)})(P_0v_2-a_2)\big\|_{H^1}+\| P_0v_2\partial_y v_1^{(2)}\|_{H^1}+\| P_0v_3\partial_z v_1^{(2)}\|_{H^1}\\
&+\|V_{\neq}\cdot\nabla (v_1)_{\neq}\|_{H^1}\\
\leq& C\Big[\big\|\f{\gamma\cos(k_fy)}{\nu k_f}+\partial_yv_1^{(1)}\big\|_{H^1}\| P_0v_2-a_2\|_{H^2}+\|P_0v_2\|_{H^2}\| \partial_y v_1^{(2)}\|_{H^2}\\
&+\|P_0v_3\|_{H^1}\| \partial_z v_1^{(2)}\|_{H^2}+\|V\|_{H^2}\|\Delta V_{\neq}\|_{L^2}\Big]\\
\leq&  C\Big[\big(\f{|\gamma|}{\nu }+\|v_1^{(1)}\|_{H^{2}}\big)\|\nabla \Delta v_2\|_{L^2}+\big(\|v_2\|_{H^2}+\|P_0v_3\|_{H^1}\big)\| \nabla\Delta v_1^{(2)}\|_{L^2}+\|V\|_{H^2}\|\Delta V_{\neq}\|_{L^2}\Big].
\end{align*}
Then we conclude
\begin{align}
\nonumber&\|\Delta v_1^{(2)}\|_{Y_0}^2\leq C\big(\|\Delta v_1^{(2)}(0)\|_{L^2}^2+\nu^{-1}\|f_1\|_{L^2L^2}^2\big)\\
\nonumber&\leq C\big(\|\Delta P_0 v_1(0)\|_{L^2}^2+\|\Delta v_1^{(1)}\|_{L^2}^2\big)+C\nu^{-1}\big(\f{|\gamma|}{\nu }+\|v_1^{(1)}\|_{H^{2}}\big)^2\|\nabla \Delta v_2\|_{L^2L^2}^2\\
\nonumber&+C\nu^{-1}\big(\|v_2\|_{L^{\infty}H^{2}}^2+\|P_0v_3\|_{L^{\infty}H^1}^2\big)\|\nabla \Delta v_1^{(2)}\|_{L^2L^2}^2+C\nu^{-1}\|V\|_{L^{\infty}H^2}^2\|\Delta V_{\neq}\|_{L^2L^2}^2\\
&\leq C\big(\|V_0\|_{H^2}^2+\|v_1^{(1)}\|_{H^{2}}^2\big)+C\nu^{-1}M_0^2\nu^{-1}\| \Delta v_1^{(2)}\|_{Y_0}^2+C\nu^{-1}M_1^2|\gamma|^{-\frac{1}{2}}\|\Delta V_{\neq}\|_{Y_0}^2\nonumber\\
\nonumber&\quad+C\nu^{-2}(|\gamma|/\nu)^2\| \Delta v_2\|_{Y_0}^2.
\end{align}
Hence, if $M_0\leq \varepsilon_1\nu,\ M_1\leq \varepsilon_1\nu^{\frac{1}{2}}|\gamma|^{\frac{1}{4}}$, then
\begin{align}
\nonumber\|\Delta v_1^{(2)}\|_{Y_0}^2\leq & C(\|V_0\|_{H^2}^2+\|v_1^{(1)}\|_{H^{2}}^2)+C\|\Delta V_{\neq}\|_{Y_0}^2+C\nu^{-2}(|\gamma|/\nu)^2\| \Delta v_2\|_{Y_0}^2\\
\nonumber\leq& C(|\gamma|/\nu^2)^2\|V_0\|_{H^2}^2+C(|\gamma|^{\frac{1}{2}}/\nu^2) \| V_0\|_{H^2}^2+C(|\gamma|/\nu^2)^2\|V_0\|_{H^2}^2\\
\label{v12}\leq& C(|\gamma|/\nu^2)^2\|V_0\|_{H^2}^2,
\end{align}
which along with \eqref{v11} and \eqref{v12} gives
\begin{align*}
\|\Delta P_0v_1(t)\|_{L^2}\leq \|\Delta v_1^{(2)}(t)\|_{L^2}+\|\Delta v_1^{(1)}(t)\|_{L^2}\leq C(|\gamma|/\nu^2)\|V_0\|_{H^2}.
\end{align*}
Therefore,
\begin{align}
\nonumber\|v_1(t)\|_{H^2}^2\leq& C\big(|a_1|+\|\Delta P_0v_1(t)\|_{L^2}+\|(\Delta v_1)_{\neq}(t)\|_{L^2}\big)^2\\
\label{v1,h2}\leq& C\|V_0\|_{H^2}^2+{C|\gamma|^{\frac{1}{2}}}/{\nu^2 }\| V_0\|_{H^2}^2+C(|\gamma|/\nu^2)^2\| V_0\|_{H^2}^2\leq C(|\gamma|/\nu^2)^2\| V_0\|_{H^2}^2.
\end{align}

Summing up \eqref{v1,h2}, \eqref{v3,h2} and \eqref{Nonlinear bound1},
we conclude
\begin{align*}
\nonumber\|V(t)\|_{H^2}^2=&\|v_1(t)\|_{H^2}^2+\|v_2(t)\|_{H^2}^2+\|v_3(t)\|_{H^2}^2 \\
\leq& C(|\gamma|/\nu^2)^2\|V_0\|_{H^2}^2+C\|V_0\|_{H^2}^2+{C|\gamma|^{\frac{1}{2}}}/{\nu^2 }\| V_0\|_{H^2}^2\leq C(|\gamma|/\nu^2)^2\|V_0\|_{H^2}^2,
\end{align*}
which implies that $M_1\leq C(|\gamma|/\nu^2)\|V_0\|_{H^2}$.
\end{proof}

\subsection{Proof of  Theorem \ref{thm:NS}}
Obviously, we have
\beno
M_0(0)+M_1(0)\leq  C_1\|V_0\|_{H^2}.
\eeno
While by the assumption, we know
\begin{align}
\|V_0\|_{H^2}\leq\varepsilon_1\nu^{\frac{5}{2}}|\gamma|^{-\frac{3}{4}}/(2C_1).
\end{align}
Therefore, there exists $T>0$ so that $M_0(T)< \varepsilon_1\nu,\ M_1(T)< \varepsilon_1\nu^{\frac{1}{2}}|\gamma|^{\frac{1}{4}}$, which  imply by Theorem \ref{Thm:non} that
\begin{align*}
 M_0(T)\leq \varepsilon_1\nu/2,\quad M_1(T)\leq \varepsilon_1\nu^{\frac{1}{2}}|\gamma|^{\frac{1}{4}}/2.
 \end{align*}
Then Theorem \ref{thm:NS} follows from a continuous argument.

\section{Appendix}

\subsection{Some basic inequalities}

\begin{lemma}
Let $\lambda\in [0,1]$, $\sin y_1=\sin y_2=\lambda$ and $0\leq y_1\leq \f{\pi}{2}\leq y_2\leq\pi$. Let
\begin{align*}
g(y)=\f{(y-y_1)(\pi-y-y_1)}{2(\sin y-\lambda)},\quad y_1< y< y_2.
\end{align*}
Then $g'(y)\leq 0$ for $y\in(y_1,\pi/2]$ and
\begin{align}
\sup_{y_1< y< y_2}g(y)=\lim_{y\to y_1+}g(y)=\f{\f{\pi}{2}-y_1}{\cos y_1}\label{A.5.1}.
\end{align}
\end{lemma}

\begin{proof} For $y\in(y_1,y_2)$, we have
\begin{align*}
g'(y)=\f{(\pi-2y)(\sin y-\lambda)-\cos y(y-y_1)(\pi-y-y_1)}{2(\sin y-\lambda)^2}.
\end{align*}
Let $h(y):=(\pi-2y)(\sin y-\lambda)-\cos y(y-y_1)(\pi-y-y_1)$.
Using Taylor's formula
\begin{align*}
\sin y_1=\sin y+\cos y (y_1-y)-\f{\sin y}{2}(y_1-y)^2-\f{\cos (\theta y_1+(1-\theta)y)}{3!}(y_1-y)^3,
\end{align*}
for some $\theta\in[0,1]$, we get
\begin{align*}
h(y)=&(\pi-2y)\Big[(y-y_1)\cos y+\f{\sin y}{2}(y-y_1)^2-\f{\cos (\theta y_1+(1-\theta)y)}{3!}(y-y_1)^3 \Big]\\
&-\cos y(y-y_1)(\pi-y-y_1)\\
=&\big[(\f{\pi}{2}-y)\sin y-\cos y\big](y-y_1)^2-(\pi-2y)\f{\cos (\theta y_1+(1-\theta)y)}{3!}(y-y_1)^3.
\end{align*}
Let $t=\f{\pi}{2}-y$, then $(\f{\pi}{2}-y)\sin y-\cos y=t\cos t-\sin t=\cos t(t-\tan t)\leq 0$ for $t\le 0$. Hence,
$h(y)\leq 0$ for $y\in(y_1,\f{\pi}{2}]$, which implies that
\begin{align*} \sup_{y_1< y\leq\f{\pi}{2}}g(y)=\lim_{y\to y_1+}g(y)=\f{\f{\pi}{2}-y_1}{\cos y_1}.
\end{align*}
While for $y\in[\f{\pi}{2},y_2)$, we have $g(y)=g(\pi-y)$, therefore,
\begin{align*}
\sup_{y_1< y< y_2}g(y)=\sup_{y_1< y\leq\f{\pi}{2}}g(y)=\lim_{y\to y_1+}g(y)=\f{\f{\pi}{2}-y_1}{\cos y_1}.
\end{align*}
\end{proof}

\begin{lemma}
Let $\lambda\in[0,1]$, $\sin y_1=\sin y_2=\lambda$ and $0\leq y_1\leq\f{\pi}{2}\leq y_2\leq\pi$. Then for any $0<\delta\leq1$,  we have
\begin{align}
&\f{1}{|\sin(y_1-\delta)-\sin y_1|}\sim\f{1}{(y_2-y_1+\delta)\delta},\label{A4.2.1}\\
&\Big\|\f{1}{\sin y-\lambda}\Big\|_{L^2(y_2+\delta,y_1+2\pi-\delta)}\lesssim\f{1}{(y_2-y_1+\delta)\delta^{\f12}},\label{A4.2.2}\\
&\Big\|\f{1}{\sin y-\lambda}\Big\|_{L^1(y_2+\delta,y_1+2\pi-\delta)}\lesssim\f{1+\ln(1+\frac{y_2-y_1}{\delta})}{(y_2-y_1+\delta)},\label{A4.2.4}\\
&\Big\|\f{\cos y}{(\sin y-\lambda)^2}\Big\|_{L^2(y_2+\delta,y_1+2\pi-\delta)}\lesssim \f{1}{(y_2-y_1+\delta)\delta^{\f32}}.\label{A4.2.3}
\end{align}
\end{lemma}

\begin{proof}
As $\f{\pi}{2}-y_1+\f{\delta}{2}\in(0,\f{3\pi}{4})$, $\sin(\f{\pi}{2}-y_1+\f{\delta}{2})\sim \f{\pi}{2}-y_1+\f{\delta}{2}$, hence,
\begin{align*}
|\sin (y_1-\delta)-\sin y_1|&=2\big|\cos (y_1-\f{\delta}{2})\sin\f{\delta}{2}\big|=2\big|\sin(\f{\pi}{2}-y_1+\f{\delta}{2})\sin\f{\delta}{2}\big|\\
&\sim \big(\f{\pi}{2}-y_1+\f{\delta}{2}\big)\delta=\f{y_2-y_1+\delta}{2}\delta\sim(y_2-y_1+\delta)\delta,
\end{align*}
which gives (\ref{A4.2.1}).

For $y\in[y_2+\delta,\f{3\pi}{2}]$,  $\f{y-y_1}{2}\in[\f{y_2-y_1+\delta}{2},\f{3\pi}{4}]$, then $|\sin\f{y-y_1}{2}|\sim(y-y_1)$ and
\begin{align*}
|\sin y-\sin y_2|&=2\big|\cos\f{y+y_2}{2}\sin\f{y-y_2}{2}\big|=2\big|\sin\f{y-y_1}{2}\sin\f{y-y_2}{2}\big|\\
&\sim(y-y_1)(y-y_2)\geq(y_2-y_1+\delta)(y-y_2),
\end{align*}
hence,
\begin{align*}
\int_{y_2+\delta}^{y_1+2\pi-\delta}\f{1}{(\sin y-\lambda)^2}dy\lesssim\f{1}{(y_2-y_1+\delta)^2} \int_{y_2+\delta}^{\f{3\pi}{2}}\f{1}{(y-y_2)^2}dy\lesssim\f{1}{(y_2-y_1+\delta)^2\delta},
\end{align*}
which gives (\ref{A4.2.2}).  Similarly, we have
\begin{align*}
\int_{y_2+\delta}^{y_1+2\pi-\delta}\f{1}{|\sin y-\lambda|}dy&=2 \int_{y_2+\delta}^{\f{3\pi}{2}}\f{1}{(\sin y-\lambda)}dy\leq C\int_{y_2+\delta}^{\f{3\pi}{2}}\f{1}{(y-y_1)(y-y_2)}dy\\
&\leq \frac{C}{y_2-y_1+\delta}\int_{y_2+\delta}^{2y_2-y_1+\delta}\f{1}{(y-y_2)}dy+C\int_{2y_2-y_1+\delta}^{+\infty}\f{1}{(y-y_2)^2}dy\\
&\leq C\f{1+\ln(1+\frac{y_2-y_1}{\delta})}{(y_2-y_1+\delta)},
\end{align*}
which gives (\ref{A4.2.4}).

If $y_2+\delta\leq\pi$, then for $y\in[y_2+\delta,\pi]$
\begin{align*}
\f{|\cos y|}{(\sin y-\lambda)^2}&=\f{|\sin(y-\f{\pi}{2})|}{2|\sin\f{y-y_1}{2}\sin\f{y-y_2}{2}|^2}\sim\f{y-\f{\pi}{2}}{(y-y_1)^2(y-y_2)^2}\\
&=\f{1}{2}\Big(\f{1}{(y-y_1)(y-y_2)^2}+\f{1}{(y-y_1)^2(y-y_2)}\Big)\\
&\leq\f{1}{(y-y_1)^2(y-y_2)}\lesssim\f{1}{(y_2-y_1+\delta)(y-y_2)^2},
\end{align*}
which gives
\begin{align*}
\int_{y_2+\delta}^{\pi}\f{\cos^2 y}{(\sin y-\lambda)^4}dy\lesssim\f{1}{(y_2-y_1+\delta)^2}\int_{y_2+\delta}^{\pi}\f{1}{(y-y_2)^4}dy\lesssim\f{1}{(y_2-y_1+\delta)^2\delta^3},\nonumber
\end{align*}
and for $y\in[\pi,\f{3\pi}{2}]$, we have
\begin{align*}
&\f{|\cos y|}{(\sin y-\lambda)^2}=\f{|\sin(\f{3\pi}{2}-y)|}{2|\sin\f{y-y_1}{2}\sin\f{y-y_2}{2}|^2}\sim\f{\f{3\pi}{2}-y}{(y-y_1)^2(y-y_2)^2}\sim \f{\f{3\pi}{2}-y}{y_2^2(y-y_2)^2}\lesssim\f{1}{(y-y_2)^2},
\end{align*}
which gives
\begin{align*}
\int_{\pi}^{\f{3\pi}{2}}\f{\cos^2 y}{(\sin y-\lambda)^4}dy\lesssim\int_{\pi}^{\f{3\pi}{2}}\f{1}{(y-y_2)^4}dy\lesssim\f{1}{\delta^3}.
\end{align*}
This shows (\ref{A4.2.3}) for $y_2+\delta\le pi$. If $y_2+\delta\geq\pi$, then
\begin{align*}
\int_{y_2+\delta}^{\f{3\pi}{2}}\f{\cos^2 y}{(\sin y-\lambda)^4}dy\lesssim\int_{y_2+\delta}^{\f{3\pi}{2}}\f{1}{(y-y_2)^4}dy\lesssim\f{1}{\delta^3}.
\end{align*}
\end{proof}

\begin{lemma}
Let $\lambda\in [0,1]$, $\sin y_1=\sin y_2=\lambda$ and $0\leq y_1\leq \f{\pi}{2}\leq y_2\leq\pi$. Then for any $0< \delta\leq\min(1,\f{y_2-y_1}{4})$, we have
\begin{align}
&{|\sin(y_1+\delta)-\sin y_1|}\sim|y_2-y_1-\delta|\delta,\label{A4.3.1}\\
&\Big\|\f{1}{\sin y-\lambda}\Big\|_{L^2(y_1+\delta,y_2-\delta)}\lesssim\f{1}{(y_2-y_1)\delta^{\f12}},\label{A4.3.2}\\
&\Big\|\f{\cos y}{(\sin y-\lambda)^2}\Big\|_{L^2(y_1+\delta,y_2-\delta)}\lesssim\f{1}{(y_2-y_1)\delta^{\f32}}.\label{A4.3.3}
\end{align}
\end{lemma}

\begin{proof}
For any $y\in[y_1,y_2]$, we have
\begin{align}
|\sin y-\sin y_1|=2\Big|\sin\f{y_2-y}{2}\sin\f{y-y_1}{2}\Big|\sim(y_2-y)(y-y_1)\label{F4.3},
\end{align}
which gives (\ref{A4.3.1}). Similarly, we have
\begin{align*}
\int_{y_1+\delta}^{y_2-\delta}\f{1}{(\sin y-\lambda)^2}dy\sim &\int_{y_1+\delta}^{y_2-\delta}\f{1}{(y_2-y)^2( y-y_1)^2}dy\\
=&\f{1}{(y_2-y_1)^2} \int_{y_1+\delta}^{y_2-\delta}(\f{1}{y_2-y}+\f{1}{y-y_1})^2dy\lesssim\f{1}{(y_2-y_1)^2\delta},
\end{align*}
which gives (\ref{A4.3.2}). For $y\in[y_1+\delta,y_2-\delta]$, we have
\begin{align*}
\f{|\cos y|}{(\sin y-\lambda)^2}\sim\f{|y-\f{\pi}{2}|}{(y-y_1)^2(y_2-y)^2}=\f{1}{2}\Big|\f{1}{(y-y_1)(y_2-y)^2}-\f{1}{(y_2-y)(y-y_1)^2}\Big|,
\end{align*}
hence,
\begin{align*}
&\int_{y_1+\delta}^{y_2-\delta}\f{|\cos y|^2}{(\sin y-\lambda)^4}dy\sim \int_{y_1+\delta}^{y_2-\delta}\Big|\f{1}{(y-y_1)(y_2-y)^2}-\f{1}{(y_2-y)(y-y_1)^2}\Big|^2dy\\
&\lesssim \f{1}{(y_2-y_1)^2}\Big (\int_{y_1+\delta}^{\f{\pi}{2}}\f{1}{(y-y_1)^4}dy+\int_{\f{\pi}{2}}^{y_2-\delta}\f{1}{(y_2-y)^4}dy\Big)\lesssim\f{1}{(y_2-y_1)^2\delta^3},
\end{align*}
which gives (\ref{A4.3.3}).
\end{proof}

\subsection{Wave operator}

The construction of wave operator is similar to \cite{WZZ3}.

\subsubsection{Notations and some estimates}

We first introduce some notations from \cite{WZZ3}.  \smallskip

Let $\al>1$, $u(y)=-\cos y$ and let $\phi_1(y,c)$ be a function constructed in  Proposition 3.6, which satisfies the properties in Lemma 3.7 and Proposition 3.8.
For $c=u(y_c)\in D_0=(-1,1)$ and $j=0,1$, $k=1,2,$ let
\begin{align*}
&\mathrm{II}(c)=\int_0^{\pi}\frac{1}{(u(y)-c)^2}\Big(\frac{1}{\phi_{1}(y,c)^2}-1\Big)dy,\\
&\rmA(c)=\sin^3 y_c\mathrm{II}(c),\quad \rmB(c)=\pi\cos y_c,\\
&J_j^k(c)=\f{-u'(y_c)(u(j\pi)-c)^k}{\phi_1((1-j)\pi,c)^{k-1}\phi_1'((1-j)\pi,c)},\ \ J_j(c)=J_j^2(c),\\
&\rmA_1(c)= J_1(c)-J_0(c)+\sin^2 y_c\rmA(c),\quad \rmB_1(c)=\sin^2 y_c \rmB,\\
&\rho(c)=(u(\pi)-c)(c-u(0))=1-c^2=\sin^2y_c.
\end{align*}
We define $\Int(\varphi) $ to be a $2\pi$ periodic function so that
\begin{align*}
\Int(\varphi)(y)=\int_{0}^y\varphi(y')\,dy' \quad\textrm{for}\, y\in [0,\pi],\quad \Int(\varphi)(y)=\Int(\varphi)(-y)\quad\textrm{for}\, y\in [-\pi,0].
\end{align*}
We denote
\begin{align*}
&\textrm{II}_{1,1}(\varphi)(y_c)
=p.v.\int_0^{\pi}\f{\Int(\varphi)(y)-\Int(\varphi)(y_c)}{(u(y)-u(y_c))^2}\,dy,\\
&\mathrm{II}_1(\varphi)(y_c)=p.v.\int_{0}^{\pi}\f{\int_{y_c}^{y}\varphi(y')\phi_1(y',c)dy'}{\phi(y,c)^2}dy,\\
&\rmE_j(\varphi)(y_c)=\int_{y_c}^{j\pi}\varphi(y)\phi_1(y,c)dy\quad \textrm{for}\ j=0,1,\\
&\mathcal{L}_0(\varphi)(y_c)=\int_0^{\pi}\int_{y_c}^{z}{\varphi}(y)
\f{1}{(u(z)-c)^2}\left(\f{\phi_1(y,c)\,}{\phi_1(z,c)^2}-1\right)\,dydz.
\end{align*}
It is easy to see that
\beno
\mathrm{II}_{1}(\varphi)=\mathrm{II}_{1,1}(\varphi)+\mathcal{L}_{0}
(\varphi).
\eeno
We denote
\beno
&&\Lambda_{1}(\varphi)(y_c)=\Lambda_{1,1}(\varphi)(y_c)+\Lambda_{1,2}(\varphi)(y_c),\\
&&\Lambda_3(\varphi)(y_c)=\rho(c)\Lambda_{1}(\varphi)(y_c)+\Lambda_{3,1}(\varphi)(y_c),
\eeno
where
\begin{align*}
&\Lambda_{1,1}(\varphi)(y_c)=\rho(c)u''(y_c)\mathrm{II}_{1,1}(\varphi)(c),\\
&\Lambda_{1,2}(\varphi)(y_c)=\rho(c)u''(y_c)\mathcal{L}_0(\varphi)(c)+u'(y_c)\rho(c)\mathrm{II}(c)\varphi(y_c),\\
&\Lambda_{3,1}(\varphi)(y_c)=J_j\left(\frac{u''(y_c)}{u'(y_c)}\rmE_{1-j}(\varphi)
+{\varphi(y_c)}\right)\bigg|_{j=0}^{1}.
\end{align*}
It holds that  if $(\al^2-\partial_y^2)\psi=\om$ with $\psi(0)=\psi(\pi)=0,$ then
\ben
\rho\mathrm{II}_1\big(u''\psi+u\om\big)
=\rho u\mathrm{II}_1(\om)-\psi(y_c)\rmA,
\een
if $(\al^2-\partial_y^2)\psi=\om$ with $\psi'(0)=\psi'(\pi)=0,$  then
\ben
\rho\mathrm{II}_1\big(u''\psi+u\om\big)
=\rho u\mathrm{II}_1(\om)-\psi(y_c)\rmA.
\een
See the proof of Proposition 10.2 in \cite{WZZ3}.\smallskip

Let us recall the following estimates for $\rmA, \rmB, \rmA_1,\rmB_1, J_j$:
\begin{align*}
&C^{-1}(1+\al\sin y_c)^2\leq (\mathrm{A}(c)^2+\mathrm{B}(c)^2) \leq C(1+\al\sin y_c)^2,\\
&|\mathrm{A}(c)|\leq C\al \sin y_c,\quad
|\pa_{c}\mathrm{A}(c)|
\leq C\min(\al^2, \al/\sin y_c),\\
&|\pa_{c}^2\mathrm{A}(c)|
\leq C\min(\al^2/\sin^2y_c,\al/\sin^3y_c),\\
&\pa_c\rmB(c)=-\pi,\quad \pa_c^2\rmB(c)=0,
\end{align*}
see Proposition 4.3 in \cite{WZZ3}, and
\begin{align*}
&\f{(1+\al\sin  y_c)^6}{C\al^4}\leq \rmA_1(c)^2+\rmB_1(c)^2\leq \f{C(1+\al\sin  y_c)^6}{\al^4},\\
&|\mathrm{A_1}(c)|\leq C(1+\al \sin y_c)^3/\al^2,\quad
|\pa_{c}\mathrm{A}_1(c)|\leq C(1+\al \sin y_c),\\
&|\pa_{c}^2\mathrm{A}_1(c)|\leq C\min(\al^2,\f{\al}{\sin y_c}),\\
&|\rmB_1(c)|\leq C\sin^2 y_c,\ |\pa_c\rmB_1(c)|\leq C,\ |\pa_c^2\rmB_1(c)|\leq C,
\end{align*}
see Proposition 4.5 in \cite{WZZ3}, and for $m=0,1,2,$
\begin{align*}
&|\partial_c^m J_{1-j}(c)|
\leq C\min\left(\f{|(1-j)\pi-y_c|^{5}}{\al^2|\sin y_c|^{2m}\phi_1(j\pi,c)},\al^{2m-2}\right),
\end{align*}
which implies\begin{align}\label{5.2}
&|\partial_{y_c}^m J_{1-j}(c)|
\leq C\min\left(\f{|(1-j)\pi-y_c|^{5}}{\al^2|\sin y_c|^{m}\phi_1(j\pi,c)},\al^{m-2}\right),
\end{align}
see Lemma 4.8 in \cite{WZZ3}. As $\partial_{y_c}=u'(y_c)\partial_{c} ,$  we can deduce that for $k=0,1,2,$
\begin{align}\label{5.1}
\left|\partial_{y_c}^k\Big(\f{1}{\rmA+i\rmB}\Big)\right|
\leq \f{C\al^k}{(1+\al\sin y_c)^{1+k}},\ \ \ \ \left|\partial_{y_c}^k\Big(\f{1}{\rmA_1+i\rmB_1}\Big)\right|
\leq \f{C\al^{2+k}}{(1+\al\sin y_c)^{3+k}}.
\end{align}

Now we recall some estimates of  $\mathrm{II}_{1,1}. $ See Lemma 10.4, Lemma 7.4 and  Lemma 7.5 in \cite{WZZ3}.
\begin{lemma}\label{lem: commu_II_1,1}
It holds that
\beno
\|\sin y_c[\pa_y^2,\rho\mathrm{II}_{1,1}]\om\|_{L^2}\leq C\|\om\|_{H^1}.
\eeno
For $k=0,1,2$, we have
\beno
\|\rho\mathrm{II}_{1,1}(\varphi)\|_{H_{ y_c}^k}
\leq C\|\varphi\|_{H^k}.
\eeno
If $\varphi\in H_0^1(0,\pi)$, then
\beno
\|\sin y_c\mathrm{II}_{1,1}(\varphi)\|_{L_{ y_c}^2}\leq C\|\varphi\|_{H^1}.
\eeno
\end{lemma}

Next we recall some estimates of $\mathcal{L}_0. $ See  Lemma 10.5, Lemma 7.6 and  Lemma 7.7 in \cite{WZZ3}.
\begin{lemma}\label{lem: commL_0}
It holds that
\beno
&&|\sin^3 y_c[\pa_{y_c}^2,\mathcal{L}_0]\varphi|\leq C\al^{\f12}(\|\varphi\|_{H^1}+\al\|\varphi\|_{L^2}),\\
&&\big\|u'(y_c)^3\pa_c\mathcal{L}_{0}(\varphi)\big\|_{L^{\infty}}\leq C\al\|\varphi\|_{L^{\infty}(0,\pi)}+C\al^{\f12}\|\varphi'\|_{L^2(0,\pi)},\\
&&\|u'(y_c)^2\mathcal{L}_{0}(\varphi)\|_{L^{\infty}}\leq C\al^{\frac{1}{2}}\|\varphi\|_{L^{2}(0,\pi)}.
\eeno
For $p\in (1,\infty)$, if ${\varphi}/{u'}\in L^p$, then there exists a constant $C$ such that,
\begin{align*}
&\|\sin y_c\mathcal{L}_{0}(\varphi)\|_{L^{\infty}(0,\pi)}
\leq C\al^{\f1p}\|\varphi\|_{W^{1,p}(0,\pi)},\\
&\|\mathcal{L}_{0}(\varphi)\|_{L^{\infty}(0,\pi)}
\leq C\al\|\varphi\|_{W^{1,\infty}(0,\pi)},\\
&\|\sin^3 y_c\pa_c\mathcal{L}_{0}(\varphi)\|_{L^{\infty}(0,\pi)}\leq C\al^{\f12}\|\varphi\|_{H^1(0,\pi)}.
\end{align*}
\end{lemma}

Finally, we present some estimates of $E_j$. First of all, by
the proof of Lemma 4.8 in \cite{WZZ3}, we know that for $\om\in H^2(0,\pi),\ \om'(0)=\om'(\pi)=0,$
\begin{align*}
&-E_j(\pa_y^2\om)=\om'(y_c)
+\int_{y_c}^{j\pi}\pa_y\om(y)\pa_y\phi_1(y,c)dy.
\end{align*}
On the other hand, we have
\begin{align*}
&\big\|J_{j}\pa_y\om\big\|_{L^2}\leq C\al^{-2}\|\pa_y\om\|_{L^2},\\
&\Big|J_{j}\int_{y_c}^{(1-j)\pi}\pa_y\phi_1(y,c)\pa_y\om dy\Big|
\leq C\f{\sin y_c}{\al^{\f12}}\|\pa_y\om\|_{L^2},
\end{align*}
from which and \eqref{5.1}, we deduce that
\begin{align*}
\left\|\frac{J_{j}E_{1-j}(\om'')}{\rmA_1+i\rmB_1} \right\|_{L^2}
\leq C\left\|\frac{\al^{-2}}{\rmA_1+i\rmB_1} \right\|_{L^{\infty}}\|\pa_y\om\|_{L^2}+\left\|\frac{\al^{-\frac{1}{2}}\sin y_c}{\rmA_1+i\rmB_1} \right\|_{L^2}\|\pa_y\om\|_{L^2}\leq C\|\pa_y\om\|_{L^2}.
\end{align*}
The following estimates come from (8.1), (8.9), (8.10), (8.12) and (8.13) in \cite{WZZ3}:
\begin{align}\label{7.1}
&\Big\|\f{\rmE_j(\varphi)}{\phi_1(j\pi,c)(j\pi-y_c)}\Big\|_{L^p}\leq C\|\varphi\|_{L^p},\ \ \ 1<p\leq+\infty,\\
\label{7.9}
&\left\|\f{(\pi-y_c)}{\min(\al y_c,1)}\int_{y_c}^0\varphi(y)\f{\partial_c\phi_1(y,c)}{\phi_1(0,c)}dy\right\|_{L^p}\leq C\al \|\varphi\|_{L^{p}},\ \ 1<p\leq \infty,\\
\label{7.10}
&\left\|\f{y_c}{\min(\al (\pi-y_c),1)}\int_{y_c}^{\pi}\varphi(y)\f{\partial_c\phi_1(y,c)}{\phi_1(\pi,c)}dy\right\|_{L^p}\leq C\al \|\varphi\|_{L^{p}},\ \ 1<p\leq \infty,\\
\label{7.11}
&\left\|y_c(\pi-y_c)^2\int_{y_c}^0\varphi(y)\f{\partial_c^2\phi_1(y,c)}{\phi_1(0,c)}dy\right\|_{L^p}\leq C\al^2\|\varphi\|_{L^{p}},\ \ 1<p\leq \infty,\\ \label{7.12}
&\left\|y_c^2(\pi-y_c)\int_{y_c}^{\pi}\varphi(y)\f{\partial_c^2\phi_1(y,c)}{\phi_1(\pi,c)}dy\right\|_{L^p}\leq C\al^2\|\varphi\|_{L^{p}},\ \ 1<p\leq \infty.
\end{align}

To deal with the even part, we introduce
\beno
\Lambda_{3,2}(\varphi)=J_1(c)\rmE_0(\varphi)-J_0(c)\rmE_1(\varphi)=J_j(c)\rmE_{1-j}(\varphi)|_{j=0}^1.
\eeno
Using the fact that $\phi_1(y,c)\geq 1,\ \phi_1(y_c,c)= 1,\ \partial_c\phi_1(y_c,c)=0,$  we infer that
\begin{align*}
\pa_{y_c}\Lambda_{3,2}(\om)
&=(\pa_{y_c}{J_{j}}\rmE_{1-j}(\om)-J_j\om(y_c))\bigg|_{j=0}^1
+{J_{j}}u'\int_{y_c}^{(1-j)\pi}\pa_c\phi_1(y,c)\om(y) dy-\om(y_c)\bigg|_{j=0}^1
\\&=T_1+T_2,
\end{align*}
and
\begin{align*}
\pa_{y_c}^2\Lambda_{3,2}(\om)
&=(\pa_{y_c}^2{J_{j}})\rmE_{1-j}(\om)\bigg|_{j=0}^1
+(2\pa_{y_c}{J_{j}}u'+{J_{j}}u'')\int_{y_c}^{(1-j)\pi}\pa_c\phi_1(y,c)\om(y) dy\bigg|_{j=0}^1\\
&\quad+J_{j}u'^2\int_{y_c}^{(1-j)\pi}\pa_c^2\phi_1(y,c)\om(y) dy\bigg|_{j=0}^1-(J_j\om'(y_c)+(\partial_{y_c}J_j)\om(y_c))\bigg|_{j=0}^1\\
&=T_3+T_4+T_5+T_6.
\end{align*}
By \eqref{5.2} and \eqref{7.1}, we have
\beno
&&\|\Lambda_{3,2}(\om)/\sin y_c\|_{L^2}\leq C\al^{-2}\|\om\|_{L^2},\quad \|T_1\|_{L^2}\leq C\al^{-2}\|\om\|_{L^2},\\
&& \|T_3\|_{L^2}\leq C\al^{-1}\|\om\|_{L^2},\quad \|T_6\|_{L^2}\leq C\al^{-1}\|\partial_y\om\|_{L^2}+C\al^{-2}\|\om\|_{L^2}.
\eeno
By \eqref{5.2}, \eqref{7.9} and \eqref{7.10}, we have
\beno
\|T_2/\sin y_c\|_{L^2}\leq C\al^{-1}\|\om\|_{L^2},\quad \|T_4\|_{L^2}\leq C\al^{-1}\|\om\|_{L^2}.
\eeno
By \eqref{5.2}, \eqref{7.11} and \eqref{7.12}, we have
\beno
 \|T_5/\sin y_c\|_{L^2}\leq C\|\om\|_{L^2}.
\eeno
By \eqref{5.1}, we have
\begin{align}\label{5.8}
&\left\|\frac{\Lambda_{3,2}(\om)}{1+\al\sin y_c}\right\|_{L^2}\leq C\al^{-3}\|\om\|_{L^2},\\ \label{5.9}&\left\|\frac{\pa_{y_c}\Lambda_{3,2}(\om)}{1+\al\sin y_c}\right\|_{L^2}\leq \|T_1\|_{L^2}+\al^{-1}\|T_2/\sin y_c\|_{L^2}\leq C\al^{-2}\|\om\|_{L^2},\\\label{5.10}&\left\|\frac{\pa_{y_c}^2\Lambda_{3,2}(\om)}{\rmA_1+i\rmB_1}\right\|_{L^2}
\leq C\al^2\left\|\frac{\pa_{y_c}^2\Lambda_{3,2}(\om)}{1+\al\sin y_c}\right\|_{L^2}\\ \nonumber
&\quad\leq C\al^2(\|T_3\|_{L^2}+\|T_4\|_{L^2}+\|T_6\|_{L^2})+C\al\left\|T_5/\sin y_c\right\|_{L^2}\leq C(\al\|\om\|_{L^2}+\|\partial_y\om\|_{L^2}).
\end{align}
In summary, we conclude
\begin{align}\label{5.11}
\left\|\frac{[\pa_{y_c}^2,\Lambda_{3,2}](\om)}{\rmA_1+i\rmB_1} \right\|_{L^2}
\leq C(\|\pa_y\om\|_{L^2}+\al\|\om\|_{L^2}).
\end{align}

\subsubsection{Construction of wave operator}

For any function $\varphi(y)$ defined on $[-\pi,\pi]$, we define \begin{align*}
-\bbD_1({\varphi}_{o})(-y_c)=\bbD_1({\varphi}_{o})(y_c)=&\frac{1}{u''}\left(\frac{\Lambda_1(\varphi_{o})(y_c)}{\rmA+i\rmB}-\varphi_{o}(y_c)\right)\\
=&\frac{\rho\mathrm{II}_1(\varphi_{o})(y_c)-\pi i\varphi_{o}(y_c) }{\rmA+i\rmB},\\
\bbD_1({\varphi}_{e})(-y_c)=\bbD_1({\varphi}_{e})(y_c)=&\frac{1}{u''}\left(\frac{\Lambda_3(\varphi_{e})(y_c)}{\rmA_1+i\rmB_1}-\varphi_{e}(y_c)\right)\\
=&\frac{\rho u'(y_c)(\rho\mathrm{II}_1(\varphi_{e})-\pi i\varphi_{e})+J_1(c)\rmE_0(\varphi_{e})-J_0(c)\rmE_1(\varphi_{e}) }{u'(y_c)(\rmA_1+i\rmB_1)},
\end{align*}
where $\varphi_{o}=\frac{\varphi(y)-\varphi(-y)}{2}$ and $\varphi_{e}=\frac{\varphi(y)+\varphi(-y)}{2}$.\smallskip

Let us first prove the property \eqref{oper,cosyD1}.\smallskip

Let $ \psi=-(\pa_y^2-\al^2)^{-1}\om.$ If $\om$ is odd, then $\psi$ is odd and $\psi(0)=\psi(\pi)=0$. Then we have
\begin{align*}
\bbD_1(u''\psi+u\om)=&\frac{\rho\mathrm{II}_1(u''\psi+u\om)-\pi i(u''\psi+u\om) }{\rmA+i\rmB}\\
=&\frac{\rho u\mathrm{II}_1(\om)-\psi(y_c)\rmA-\pi i(u''\psi+u\om) }{\rmA+i\rmB}=u\bbD_1(\om)-\psi,
\end{align*}
which shows that for $\om$ odd,
\beno
-\bbD_1(\cos y(\om-\psi))=-\cos y\bbD_1(\om)-\psi.
\eeno
If $\om$ is even, then $\psi$ is also even and  $\pa_y\psi(0)=\pa_y\psi(\pi)=0$. Then we have
\begin{align*}
u''\bbD_1(u''\psi+u\om)=&\frac{\Lambda_3(u''\psi+u\om)(y_c)}{\rmA_1+i\rmB_1}-(u''\psi+u\om)(y_c)\\
=&\frac{u\Lambda_3(\om)(y_c)}{\rmA_1+i\rmB_1}-(u''\psi+u\om)(y_c)=uu''\bbD_1(\om)-u''\psi,
\end{align*} and $ -\bbD_1(u''\psi+u\om)=-u\bbD_1(\om)+\psi$. This shows that for $\om$ even,
\beno
\bbD_1(\cos y(\om-\psi))=\cos y\bbD_1(\om)+\psi.
\eeno

\subsubsection{Commutator estimate of wave operator}
Here we prove \eqref{est of D1} and consider two cases.\smallskip

\no{\bf 1. The odd case.} \smallskip

Let $\om\in H^2$ be odd and $\om(0)=\om(\pi)=0$. Direct calculation gives
\begin{align*}
&\pa_{y_c}^2(\sin y_c\bbD_1(\om))-\sin y_c\bbD_1(\pa_y^2\om)\\
&=\rho\mathrm{II}_{1,1}(\om)\pa_{y_c}^2\Big(\f{\sin y_c}{\rmA+i\rmB}\Big)+2\pa_{y_c}\big(\rho\mathrm{II}_{1,1}(\om)\big)\pa_{y_c}\Big(\f{\sin y_c}{\rmA+i\rmB}\Big)+\f{ \sin y_c[\pa_{y_c}^2,\rho\mathrm{II}_{1,1}]\om}{\rmA+i\rmB}\\
&\quad+\mathcal{L}_{0}(\om)\pa_{y_c}^2\Big(\rho\f{\sin y_c}{\rmA+i\rmB}\Big)+2\pa_{y_c}\mathcal{L}_{0}(\om)\pa_{y_c}\Big(\rho\f{\sin y_c}{\rmA+i\rmB}\Big)+\f{\rho \sin y_c[\pa_{y_c}^2,\mathcal{L}_{0}]\om}{\rmA+i\rmB}\\
&\quad-\om\pa_{y_c}^2\Big(\f{\pi i\sin y_c}{\rmA+i\rmB}\Big)-2\pa_{y_c}\om\pa_{y_c}\Big(\f{\pi i\sin y_c}{\rmA+i\rmB}\Big)\\
&=I_1+\cdots+I_8.
\end{align*}
By Lemma \ref{lem: commu_II_1,1} and \eqref{5.1}, we get
\begin{align*}
&\|I_1\|_{L^2}\leq C\|\sin y_c\mathrm{II}_{1,1}\|_{L^2}\Big\|\sin y_c\pa_{y_c}^2\Big(\f{\sin y_c}{\rmA+i\rmB}\Big)\Big\|_{L^{\infty}}\leq C\|\om\|_{H^1},\\
&\|I_2\|_{L^2}\leq C\|\rho\mathrm{II}_{1,1}\|_{H^1}\Big\|\pa_{y_c}\Big(\f{\sin y_c}{\rmA+i\rmB}\Big)\Big\|_{L^{\infty}}\leq C\|\om\|_{H^1},\\
&\|I_3\|_{L^2}\leq C\|\sin y_c[\pa_y^2,\rho\mathrm{II}_{1,1}]\om\|_{L^2}\leq C\|\om\|_{H^1}.
\end{align*}
By Lemma \ref{lem: commL_0} and \eqref{5.1}, we get
\begin{align*}
\|I_4\|_{L^2}&\leq C\al^{-\f12}\|\sin y_c\mathcal{L}_0(\om)\|_{L^{\infty}}\Big\|\frac{\al^{\f12}}{\sin y_c}\pa_{y_c}^2\Big(\f{\rho\sin y_c}{\rmA+i\rmB}\Big)\Big\|_{L^2}\\
&\leq C\|\om\|_{H^1}\Big\|\f{\al^{\f12}}{(1+\al\sin y_c)}\Big\|_{L^2}\leq C\|\om\|_{H^1},\\
\|I_5\|_{L^2}&\leq C\al^{-\f12}\|\sin^2 y_c\pa_{y_c}\mathcal{L}_0(\om)\|_{L^{\infty}}\Big\|\f{\al^{\f12}}{\sin^2 y_c}\pa_{y_c}\Big(\f{\rho\sin y_c}{\rmA+i\rmB}\Big)\Big\|_{L^2}\\
&\leq C\|\om\|_{H^1}\Big\|\f{\al^{\f12}}{(1+\al\sin y_c)}\Big\|_{L^2}\leq C\|\om\|_{H^1},
\end{align*}
and by Lemma \ref{lem: commL_0} again,
\begin{align*}
\|I_6\|_{L^2}&\leq C\al^{-\f12}\Big\|\f{\al^{\f12}}{(1+\al\sin y_c)}\Big\|_{L^2}\|\sin^3 y_c[\pa_{y_c}^2,\mathcal{L}_0]\om\|_{L^{\infty}}\\
&\leq C(\|\om\|_{H^1}+\al\|\om\|_{L^2}).
\end{align*}
By \eqref{5.1}, we have
\begin{align*}
\|I_7\|_{L^2}&\leq C\|\om/u'\|_{L^2}\Big\|\sin y_c\pa_{y_c}^2\Big(\f{\pi i\sin y_c}{\rmA+i\rmB}\Big)\Big\|_{L^{\infty}}\leq C\|\om\|_{H^1},\\
\|I_8\|_{L^2}&\leq C\|\pa_y\om\|_{L^2}\Big\|\pa_{y_c}\Big(\f{\pi i\sin y_c}{\rmA+i\rmB}\Big)\Big\|_{L^{\infty}}\leq C\|\pa_y\om\|_{L^2}.
\end{align*}

Summing up, we conclude the odd case.\smallskip

\no{\bf 2. The even case.} \smallskip

Let $\om\in H^2$ be even and $\om'(0)=\om'(\pi)=0 $. Direct calculation gives
\begin{align*}
&\pa_{y_c}^2(\sin y_c\bbD_1(\om))-\sin y_c\bbD_1(\pa_y^2\om)\\
&=\f{\rho u'(y_c)}{\rmA_1+i\rmB_1}[\pa_{y_c}^2,\rho\mathrm{II}_{1,1}]\om
+2\pa_{y_c}\Big(\f{\rho u'(y_c)}{\rmA_1+i\rmB_1}\Big)\pa_{y_c}\big(\rho\mathrm{II}_{1,1}(\om)\big)
+\pa_{y_c}^2\Big(\f{\rho u'(y_c)}{\rmA_1+i\rmB_1}\Big)\rho\mathrm{II}_{1,1}(\om)\\
&\quad+\f{\rho^2 u'(y_c)}{\rmA_1+i\rmB_1}[\pa_{y_c}^2,\mathcal{L}_{0}]\om
+\pa_{y_c}\Big(\f{\rho^2 u'(y_c)}{\rmA_1+i\rmB_1}\Big)\pa_{y_c}\big(\mathcal{L}_{0}(\om)\big)
+\pa_{y_c}^2\Big(\f{\rho^2 u'(y_c)}{\rmA_1+i\rmB_1}\Big)\mathcal{L}_{0}(\om)\\
&\quad-\pi i\om\pa_{y_c}^2\Big(\f{\rho u'(y_c)}{\rmA_1+i\rmB_1}\Big)-2\pi i\pa_{y_c}\Big(\f{\rho u'(y_c)}{\rmA_1+i\rmB_1}\Big)\pa_{y_c}\om\\
&\quad+\pa_{y_c}^2\Big(\f{1}{\rmA_1+i\rmB_1}\Big)\Lambda_{3,2}(\om)+2\pa_{y_c}\Big(\f{1}{\rmA_1+i\rmB_1}\Big)\pa_{y_c}\Lambda_{3,2}(\om)
+\f{[\pa_{y_c}^2,\Lambda_{3,2}]\om}{\rmA_1+i\rmB_1}\\
&=I_1'+\cdots+I_{11}'.
\end{align*}

By \eqref{5.1} and Lemma \ref{lem: commu_II_1,1}, we get
\beno
&&\|I_1'\|_{L^2}\leq C\left\|\f{\rho }{\rmA_1+i\rmB_1}\right\|_{L^{\infty}}\big\|\sin y_c[\pa_{y_c}^2,\rho\mathrm{II}_{1,1}]\om\big\|_{L^2}\leq C(\al\|\om\|_{L^2}+\|\pa_y\om\|_{L^2}),\\
&&\|I_2'\|_{L^2}\leq C\left\|\pa_{y_c}\Big(\f{\rho u'(y_c)}{\rmA_1+i\rmB_1}\Big)\right\|_{L^{\infty}}\big\|\pa_{y_c}\big(\rho\mathrm{II}_{1,1}(\om)\big)\big\|_{L^2}\leq C\|\om\|_{H^1},\\
&&\|I_3'\|_{L^2}\leq C\left\|\pa_{y_c}^2\Big(\f{\rho u'(y_c)}{\rmA_1+i\rmB_1}\Big)\right\|_{L^{\infty}}\big\|\rho\mathrm{II}_{1,1}(\om)\big\|_{L^2}\leq C\al\|\om\|_{L^2}.
\eeno
By \eqref{5.1} and Lemma \ref{lem: commL_0}, we get
\beno
\|I_4'\|_{L^2}\leq C\left\|\f{\rho}{\rmA_1+i\rmB_1}\right\|_{L^{2}}\|\sin^3 y_c[\pa_{y_c}^2,\mathcal{L}_{0}]\om\|_{L^{\infty}}\leq C(\al\|\om\|_{L^2}+\|\pa_y\om\|_{L^2}),
\eeno
and
\begin{align*}
\|I_5'\|_{L^2}&\leq C\left\|\f{1}{\sin^2 y_c}\pa_{y_c}\Big(\f{\rho^2 u'(y_c)}{\rmA_1+i\rmB_1}\Big)\right\|_{L^2}\big\|\sin^2 y_c\pa_{y_c}\big(\mathcal{L}_{0}(\om)\big)\|_{L^{\infty}}\\
&\leq C(\al^{\frac{1}{2}}\|\om\|_{L^{\infty}}+\|\om'\|_{L^{2}})\leq C(\al\|\om\|_{L^{2}}+\|\om'\|_{L^{2}}),
\end{align*}
and
\begin{align*}
\|I_6'\|_{L^2}&\leq C\left\|\f{1}{\sin^2 y_c}\pa_{y_c}^2\Big(\f{\rho^2 u'(y_c)}{\rmA_1+i\rmB_1}\Big)\right\|_{L^2}\big\|\rho\mathcal{L}_{0}(\om)\big\|_{L^{\infty}}\\
&\leq C\al\|\om\|_{L^2}.
\end{align*}
We get by \eqref{5.1} that
\beno
&&\|I_7'\|_{L^2}\leq C\|\om\|_{L^2}\left\|\pa_{y_c}^2\Big(\f{\rho u'(y_c)}{\rmA_1+i\rmB_1}\Big)\right\|_{L^{\infty}}\leq C\al\|\om\|_{L^2},\\
&&\|I_8'\|_{L^2}\leq C\|\pa_y\om\|_{L^2}\left\|\pa_{y_c}\Big(\f{\rho u'(y_c)}{\rmA_1+i\rmB_1}\Big)\right\|_{L^{\infty}}\leq C\|\pa_y\om\|_{L^2}.
\eeno
By \eqref{5.1}, \eqref{5.8} and \eqref{5.9}, we get
\beno
&&\|I_9'\|_{L^2}\leq C\al^4\left\|\frac{\Lambda_{3,2}(\om)}{1+\al\sin y_c}\right\|_{L^2}\leq C\al\|\om\|_{L^2},\\
&&\|I_{10}'\|_{L^2}\leq C\al^3\left\|\frac{\pa_{y_c}\Lambda_{3,2}(\om)}{1+\al\sin y_c}\right\|_{L^2}\leq C\al\|\om\|_{L^2}.
\eeno
And by \eqref{5.11}, we have
\beno
\|I_{11}'\|_{L^2}\leq C(\|\pa_y\om\|_{L^2}+\al\|\om\|_{L^2}).
\eeno

Summing up, we conclude the even case.

\subsubsection{$H^1$ estimate of wave operator}

We need the following lemmas.

\begin{lemma}\label{lem4}
It holds that
\beno
\|\sin^3 y_c[\partial_{y_c},\mathcal{L}_0](\varphi)\|_{L^{\infty}}\leq C(\|\varphi\|_{L^{\infty}}+\al^{\f12}\|\varphi\|_{L^{2}}).
\eeno
\end{lemma}

\begin{proof}
By (7.6) in \cite{WZZ3}, we get
\beno
\partial_{y_c}\mathcal{L}_0(\varphi)
=\left(\mathcal{L}_0(\varphi')-I_{0,1}(\varphi)+I_{0,0}(\varphi)\right)+u'(y_c)\mathcal{L}_{1}(\varphi),
\eeno
which gives
\beno
[\partial_{y_c},\mathcal{L}_0](\varphi)
=-I_{0,1}(\varphi)+I_{0,0}(\varphi)+u'(y_c)\mathcal{L}_{1}(\varphi).
\eeno
We get by (7.7) in \cite{WZZ2} that
\beno
|\sin^3 y_cI_{0,j}(\varphi)|\leq C\|\varphi\|_{L^{\infty}},
\eeno
and by Lemma 7.6 in \cite{WZZ2}, we get
\beno
|u'(y_c)^4\mathcal{L}_{1}(\varphi)|\leq C\al^{\f12}\|\varphi\|_{L^{2}}.
\eeno
This proves the lemma.
\end{proof}

\begin{lemma}\label{lem5}It holds that
\beno
\|\rho\mathrm{II}_{1}(\varphi)\|_{L^2}
\leq C\|\varphi\|_{L^2},\ \ \|\rho\mathcal{L}_0(\varphi)\|_{L^2}
\leq C\|\varphi\|_{L^2}.
\eeno
\end{lemma}

\begin{proof}
We can write $\mathcal{L}_0(\varphi)=\mathcal{L}_{0,1}(\varphi)+\mathcal{L}_{0,2}(\varphi)+\mathcal{L}_{0,3}(\varphi)$ so that for $y_c\in I_k,$
\beno
&&\mathcal{L}_{0,1}(\varphi)(y_c)=\mathcal{L}_{0,1}(\varphi\chi_{J_k})(y_c),\\
&&\mathcal{L}_{0,2}(\varphi)(y_c)=\int_0^{\pi}\int_{y_c}^{z}{\varphi}(y)(1-\chi_{J_k}(y))
\f{1}{(u(z)-c)^2}\f{\phi_1(y,c)\,}{\phi_1(z,c)^2}\,dydz ,\\
&&\mathcal{L}_{0,3}(\varphi)(y_c)=-\int_0^{\pi}\int_{y_c}^{z}{\varphi}(y)(1-\chi_{J_k}(y))
\f{1}{(u(z)-c)^2}\,dydz,
\eeno
Here $I_k=((k-1)\pi/n,k\pi/n]$ for $1\leq k\leq n,\ k\in\mathbb{Z}$ and $I_k=\emptyset$ otherwise, and $J_k=I_{k-1}\cup I_{k}\cup I_{k+1},$ and $n$ is the integer such that $n\leq \pi\al<n+1.$

 Due to $\al>1,$ we have $n\geq 3,\ \pi\al<n+1<3n/2$ and $|y-y_c|\leq 2\pi/n<3/\al$ for $y_c\in I_k,\ y\in J_k.$
 By Lemma \ref{lem: commL_0}, we have
 \beno
 |u'(y_c)^2\mathcal{L}_{0,1}(\varphi)(y_c)|^2\leq C\al\int_0^{\pi}|\varphi(y)|^2\chi_{\{|y-y_c|\leq3/\al\}}dy,
 \eeno
 which gives
 \begin{align}\label{L0,1}
\|(u')^2\mathcal{L}_{0,1}(\varphi)\|_{L^2}^2\leq C\al\int_0^{\pi}|\varphi(y)|^2\int_0^{\pi}\chi_{\{|y-y_c|\leq3/\al\}}d{y_c}dy\leq C\|\varphi\|_{L^2}^2.
\end{align}
By Lemma 3.7 and Proposition 3.8 in \cite{WZZ3}, we know that
 $1\leq\phi_1(y,c)\leq \phi_1(z,c)$ and $0< \f{\phi_1(y,c)\,}{\phi_1(z,c)^2}\leq \f{1}{\phi_1(z,c)}\leq Ce^{-C^{-1}\al|z-y_c|}$ for $y_c\leq y\leq z\leq \pi$ or $0\leq z\leq y_c\leq y.$ On the other hand, $u'(y_c)|z-y_c|\leq C|u(z)-c|$ and $|y-y_c|\geq \pi/n\geq1/\al$ for $y_c\in I_k,\ y\not\in J_k$. Thus, we obtain
 \begin{align*}
|u'(y_c)^2\mathcal{L}_{0,2}(\varphi)(y_c)|&\leq C\|\varphi\|_{L^{\infty}}\int_{\{|z-y_c|\geq1/\al\}\cap[0,\pi]}
\f{|z-y_c|u'(y_c)^2}{(u(z)-c)^2}e^{-C^{-1}\al|z-y_c|}dz\\ &\leq C\|\varphi\|_{L^{\infty}}\int_{\{|z-y_c|\geq1/\al\}\cap[0,\pi]}
\f{1}{|z-y_c|}e^{-C^{-1}\al|z-y_c|}dz\\ &\leq C\|\varphi\|_{L^{\infty}}\int_{\{|z|\geq1\}}
\f{1}{|z|}e^{-C^{-1}|z|}dz\leq C\|\varphi\|_{L^{\infty}},
\end{align*}
and
\begin{align*}
\|(u')^2\mathcal{L}_{0,2}(\varphi)\|_{L^1}&\leq C\int_{1/\al}^{\pi}\int_0^{y_c-1/\al}\int_{z}^{y_c-1/\al}|{\varphi}(y)|
\f{u'(y_c)^2}{(u(z)-c)^2}e^{-C^{-1}\al|z-y_c|}\,dydzdy_c
\\ &\quad+C\int_0^{\pi-1/\al}\int_{y_c+1/\al}^{\pi}\int_{y_c+1/\al}^{z}|{\varphi}(y)|
\f{u'(y_c)^2}{(u(z)-c)^2}e^{-C^{-1}\al|z-y_c|}\,dydzdy_c\\ &\leq C\al^2\int_0^{\pi-1/\al}|{\varphi}(y)|\int_{0}^y\int_{y+1/\al}^{\pi}
e^{-C^{-1}\al|z-y_c|}dzdy_cdy\\ &\quad +C\al^2\int_{\al}^{\pi}|{\varphi}(y)|\int_{0}^{y-1/\al}\int_{y}^{\pi}
e^{-C^{-1}\al|z-y_c|}dy_cdzdy\\ &\leq C\al\int_0^{\pi-1/\al}|{\varphi}(y)|\int_{0}^y
e^{-C^{-1}\al|y-y_c|}dy_cdy\\ & \quad+C\al\int_{1/\al}^{\pi}|{\varphi}(y)|\int_{0}^{y-1/\al}
e^{-C^{-1}\al|z-y|}dzdy\leq C\int_0^{\pi}|{\varphi}(y)|dy,
\end{align*}
from which and the interpolation, we deduce that
\begin{align}\label{L0,2}
\|(u')^2\mathcal{L}_{0,2}(\varphi)\|_{L^2}\leq C\|\varphi\|_{L^2}.
\end{align}

Using the fact that
\begin{align*}
\textrm{II}_{1,1}(\varphi)(y_c)
=p.v.\int_0^{\pi}\f{\Int(\varphi)(y)-\Int(\varphi)(y_c)}{(u(y)-u(y_c))^2}\,dy
=p.v.\int_0^{\pi}\f{\int_{y_c}^{z}{\varphi}(y)dy}{(u(z)-c)^2}\,dz,\end{align*}
we find that $\mathcal{L}_{0,3}(\varphi)(y_c)=-\textrm{II}_{1,1}((1-\chi_{J_k})\varphi)(y_c) $, thus,
\beno
\mathcal{L}_{0,3}(\varphi)(y_c)+\textrm{II}_{1,1}(\varphi)(y_c)=\textrm{II}_{1,1}(\chi_{J_k}\varphi)(y_c):=\mathcal{L}_{0,4}(\varphi)(y_c),
\eeno
 for $y_c\in I_k.$ By Lemma \ref{lem: commu_II_1,1}, we have \begin{align}\label{L0,3}
\|\rho\textrm{II}_{1,1}(\varphi)\|_{L^2}\leq C\|\varphi\|_{L^2},
\end{align}
which implies
\begin{align*}
\|\rho\mathcal{L}_{0,4}(\varphi)\|_{L^2(I_k)}\leq \|\rho\textrm{II}_{1,1}(\chi_{J_k}\varphi)\|_{L^2}
\leq C\|\chi_{J_k}\varphi\|_{L^2}.
\end{align*}
Thus, we have
\begin{align}\label{L0,4}
\|\rho\mathcal{L}_{0,4}(\varphi)\|_{L^2}^2=&\sum_{k=1}^n\|\rho\mathcal{L}_{0,4}(\varphi)\|_{L^2(I_k)}^2\leq C\sum_{k=1}^n\|\chi_{J_k}\varphi\|_{L^2}^2\\ \nonumber=&C\sum_{k=1}^n
\big(\|\varphi\|_{L^2(I_{k-1})}^2+\|\varphi\|_{L^2(I_{k-1})}^2+\|\varphi\|_{L^2(I_{k+1})}^2\big)\leq C\|\varphi\|_{L^2}^2.
\end{align}

Notice that $\rho(c)=(\sin y_c)^2=u'(y_c)^2,$ and $ \mathcal{L}_0(\varphi)=\mathcal{L}_{0,1}(\varphi)+\mathcal{L}_{0,2}(\varphi)+\mathcal{L}_{0,3}(\varphi),$ and $ \textrm{II}_{1}(\varphi)=\mathcal{L}_0(\varphi)+\textrm{II}_{1,1}(\varphi)
=\mathcal{L}_{0,1}(\varphi)+\mathcal{L}_{0,2}(\varphi)+\mathcal{L}_{0,4}(\varphi)$. Then the result is a consequence of \eqref{L0,1},\eqref{L0,2},\eqref{L0,3} and \eqref{L0,4}.
\end{proof}

Now we are in a position to prove \eqref{D1.1} and \eqref{D1.2}.

\begin{proof}
Let $\om=\om_o+\om_e$ with $\om_o=\f{\om(y)-\om(-y)}{2}$ and $\om_e=\f{\om(y)+\om(-y)}{2}$. Then by \eqref{5.1}, \eqref{5.8} and Lemma \ref{lem5}, we have
\begin{align*}
\|\sin y\bbD_1(\om_o)\|_{L^2}\leq\left\|\frac{\sin y_c}{\rmA+i\rmB} \right\|_{L^{\infty}}(\|\rho\mathrm{II}_{1}(\om_o)\|_{L^2}+\pi\|\om_o\|_{L^2})\leq C\al^{-1}\|\om_o\|_{L^2},
\end{align*}
and
\begin{align*}
\|\sin y\bbD_1(\om_e)\|_{L^2}\leq& C\left\|\frac{\al^2\sin^3 y_c(\rho\mathrm{II}_1(\om_{e})-\pi i\om_{e})+\al^2\Lambda_{3,2}(\om_e)}{(1+\al\sin y_c)^3} \right\|_{L^{2}}\\ \leq& C\al^{-1}(\|\rho\mathrm{II}_{1}(\om_e)\|_{L^2}+\pi\|\om_e\|_{L^2})+C\al^2\left\|\frac{\Lambda_{3,2}(\om_e)}{1+\al\sin y_c}\right\|_{L^2}\\
\leq& C\al^{-1}\|\om_e\|_{L^2},
\end{align*}
which show the inequality \eqref{D1.1}.\smallskip

 If $\om\in H^2( \mathbb{T})$, by  \eqref{est of D1} , we have
 \beno
 \sin y\bbD_1(\om)\in  H^2( (-\pi,0)\cup(0,\pi)),\quad \bbD_1(\om)\in  H^2( (-5\pi/6,-\pi/6)\cup(\pi/6,5\pi/6)).
 \eeno
 By the proof of Proposition 6.1 in \cite{WZZ3}, we have $ \|\frac{\partial_{y_c}^k\Lambda_{1}(\om_o)}{(\sin y_c)^{2-k}(1+\al\sin y_c)}\|_{L^2}\leq C\al^{\frac{1}{2}}\|\om_o\|_{H^2}$ for $k=0,1,2,$ using \eqref{5.1} and $ -u''\bbD_1({\om}_{o})(-y_c)=u''\bbD_1({\om}_{o})(y_c)=\frac{\Lambda_1(\om_{o})(y_c)}{\rmA+i\rmB}-\om_{o}(y_c)$, we deduce that  $u''\bbD_1({\om}_{o})\in H^2( \mathbb{T}).$  By the proof of Proposition 6.2 in \cite{WZZ3}, we have $ \|\frac{\partial_{y_c}^k\Lambda_{1}(\om_e)}{(\sin y_c)^{2-k}(1+\al\sin y_c)^3}\|_{L^2}\leq C\al^{-\frac{1}{2}}\|\om_e\|_{H^2}$ for $k=0,1,2,$ using \eqref{5.1} and $ u''\bbD_1({\om}_{e})(-y_c)=u''\bbD_1({\om}_{e})(y_c)=\frac{\Lambda_3(\om_{e})(y_c)}{\rmA_1+i\rmB_1}-\om_{e}(y_c)$, we deduce that $u''\bbD_1({\om}_{e})\in H^2( \mathbb{T}).$ As $u''(y)=\cos y,$ and $u''\bbD_1({\om})\in H^2( \mathbb{T})$, we have $\bbD_1({\om})\in H^2( (-\pi/3,\pi/3)\cup(2\pi/3,4\pi/3)).$ Therefore, $\bbD_1({\om})\in H^2( \mathbb{T})$  if $\om H^2(\T)$.\smallskip

 Now we prove \eqref{D1.2}. Direct calculation gives
\begin{align*}
\pa_{y_c}(\sin y_c\bbD_1(\om_o))
&=\rho\mathrm{II}_{1,1}(\om_o)\pa_{y_c}\Big(\f{\sin y_c}{\rmA+i\rmB}\Big)+\pa_{y_c}\big(\rho\mathrm{II}_{1,1}(\om_o)\big)\Big(\f{\sin y_c}{\rmA+i\rmB}\Big)\\
&\quad+\mathcal{L}_{0}(\om_o)\pa_{y_c}\Big(\rho\f{\sin y_c}{\rmA+i\rmB}\Big)+\pa_{y_c}\mathcal{L}_{0}(\om_o)\Big(\rho\f{\sin y_c}{\rmA+i\rmB}\Big)\\
&\quad-\om_o\pa_{y_c}\Big(\f{\pi i\sin y_c}{\rmA+i\rmB}\Big)-\pa_{y_c}\om_o\Big(\f{\pi i\sin y_c}{\rmA+i\rmB}\Big)\\
&=I_1+\cdots+I_6.
\end{align*}

By Lemma \ref{lem: commu_II_1,1}  and \eqref{5.1}, we get
\begin{align*}
&\|I_1\|_{L^2}\leq C\|\sin y_c\mathrm{II}_{1,1}\|_{L^2}\Big\|\sin y_c\pa_{y_c}\Big(\f{\sin y_c}{\rmA+i\rmB}\Big)\Big\|_{L^{\infty}}\leq C\al^{-1}\|\om_o\|_{H^1},\\
&\|I_2\|_{L^2}\leq C\|\rho\mathrm{II}_{1,1}\|_{H^1}\Big\|\Big(\f{\sin y_c}{\rmA+i\rmB}\Big)\Big\|_{L^{\infty}}\leq C\al^{-1}\|\om_o\|_{H^1}.
\end{align*}

By Lemma \ref{lem: commL_0}, Lemma \ref{lem4}, Lemma \ref{lem5} and \eqref{5.1}, we get
\begin{align*}
\|I_3\|_{L^2}&\leq C\al^{-\f12}\|\sin^2 y_c\mathcal{L}_0(\om_o)\|_{L^{\infty}}\Big\|\frac{\al^{\f12}}{\sin^2 y_c}\pa_{y_c}\Big(\f{\rho\sin y_c}{\rmA+i\rmB}\Big)\Big\|_{L^2}\\
&\leq C\|\om_o\|_{L^2}\Big\|\f{\al^{\f12}}{(1+\al\sin y_c)}\Big\|_{L^2}\leq C\|\om_o\|_{L^2},\\
\|I_4\|_{L^2}&\leq C\|\sin^2 y_c\mathcal{L}_0(\om_o')\|_{L^{2}}\Big\|\f{1}{\sin^2 y_c}\Big(\f{\rho\sin y_c}{\rmA+i\rmB}\Big)\Big\|_{L^{\infty}}\\&\quad+C\Big\|\sin^3 y_c[\partial_{y_c},\mathcal{L}_0](\om_o)\|_{L^{\infty}}\Big\|\f{1}{\sin^3 y_c}\Big(\f{\rho\sin y_c}{\rmA+i\rmB}\Big)\Big\|_{L^{2}}\\
&\leq C\al^{-1}\|\om_o'\|_{L^2}+C\al^{-\f12}(\|\om_o\|_{L^{\infty}}+\al^{\f12}\|\om_o\|_{L^{2}})\leq C\al^{-1}\|\om_o'\|_{L^2}+C\|\om_o\|_{L^{2}}.
\end{align*}By \eqref{5.1}, we get
\begin{align*}
\|I_5\|_{L^2}&\leq C\|\om_o\|_{L^2}\Big\|\sin y_c\pa_{y_c}\Big(\f{\pi i\sin y_c}{\rmA+i\rmB}\Big)\Big\|_{L^{\infty}}\leq C\|\om_o\|_{H^1},\\
\|I_6\|_{L^2}&\leq C\|\pa_y\om_o\|_{L^2}\Big\|\Big(\f{\pi i\sin y_c}{\rmA+i\rmB}\Big)\Big\|_{L^{\infty}}\leq C\al^{-1}\|\pa_y\om\|_{L^2}.
\end{align*} Summing up, we conclude the odd case. For the even case, we have\begin{align*}
&\pa_{y_c}(\sin y_c\bbD_1(\om))\\
&=\Big(\f{\rho u'(y_c)}{\rmA_1+i\rmB_1}\Big)\pa_{y_c}\big(\rho\mathrm{II}_{1,1}(\om_e)\big)
+\pa_{y_c}\Big(\f{\rho u'(y_c)}{\rmA_1+i\rmB_1}\Big)\rho\mathrm{II}_{1,1}(\om_e)\\
&\quad
+\Big(\f{\rho^2 u'(y_c)}{\rmA_1+i\rmB_1}\Big)\pa_{y_c}\big(\mathcal{L}_{0}(\om_e)\big)
+\pa_{y_c}\Big(\f{\rho^2 u'(y_c)}{\rmA_1+i\rmB_1}\Big)\mathcal{L}_{0}(\om_e)\\
&\quad-\pi i\om_e\pa_{y_c}\Big(\f{\rho u'(y_c)}{\rmA_1+i\rmB_1}\Big)-\pi i\Big(\f{\rho u'(y_c)}{\rmA_1+i\rmB_1}\Big)\pa_{y_c}\om_e\\
&\quad+\pa_{y_c}\Big(\f{1}{\rmA_1+i\rmB_1}\Big)\Lambda_{3,2}(\om_e)+\Big(\f{1}{\rmA_1+i\rmB_1}\Big)\pa_{y_c}\Lambda_{3,2}(\om_e)\\
&=I_1'+\cdots+I_8'.
\end{align*}

By \eqref{5.1} and Lemma \ref{lem: commu_II_1,1}, we get
\beno
&&\|I_1'\|_{L^2}\leq C\left\|\Big(\f{\rho u'(y_c)}{\rmA_1+i\rmB_1}\Big)\right\|_{L^{\infty}}\big\|\pa_{y_c}\big(\rho\mathrm{II}_{1,1}(\om_e)\big)\big\|_{L^2}\leq C\al^{-1}\|\om_e\|_{H^1},\\
&&\|I_2'\|_{L^2}\leq C\left\|\pa_{y_c}\Big(\f{\rho u'(y_c)}{\rmA_1+i\rmB_1}\Big)\right\|_{L^{\infty}}\big\|\rho\mathrm{II}_{1,1}(\om_e)\big\|_{L^2}\leq C\|\om_e\|_{L^2}.
\eeno
By \eqref{5.1}, Lemma \ref{lem4} and Lemma \ref{lem5}, we get
\begin{align*}
\|I_3'\|_{L^2}&\leq C\left\|\f{1}{\sin^2 y_c}\Big(\f{\rho^2 u'(y_c)}{\rmA_1+i\rmB_1}\Big)\right\|_{L^{\infty}}\big\|\sin^2 y_c\mathcal{L}_{0}(\om_e')\|_{L^2}\\&+C\left\|\f{1}{\sin^3 y_c}\Big(\f{\rho^2 u'(y_c)}{\rmA_1+i\rmB_1}\Big)\right\|_{L^2}\big\|\sin^3 y_c[\partial_{y_c},\mathcal{L}_0](\om_e)\|_{L^{\infty}}\\
&\leq C\al^{-1}\|\om_e'\|_{L^{2}}+\al^{-\frac{1}{2}}(\|\om_e\|_{L^{\infty}}+\al^{\frac{1}{2}}\|\om_e\|_{L^{2}})\leq C(\|\om_e\|_{L^{2}}+\al^{-1}\|\om_e'\|_{L^{2}}).
\end{align*}
By \eqref{5.1} and Lemma \ref{lem: commL_0}, we get
\begin{align*}
\|I_4'\|_{L^2}&\leq C\left\|\f{1}{\sin^2 y_c}\pa_{y_c}\Big(\f{\rho^2 u'(y_c)}{\rmA_1+i\rmB_1}\Big)\right\|_{L^2}\big\|\rho\mathcal{L}_{0}(\om_e)\big\|_{L^{\infty}}\leq C\|\om_e\|_{L^2}.
\end{align*}
By \eqref{5.1}, we get
\beno
&&\|I_5'\|_{L^2}\leq C\|\om_e\|_{L^2}\left\|\pa_{y_c}\Big(\f{\rho u'(y_c)}{\rmA_1+i\rmB_1}\Big)\right\|_{L^{\infty}}\leq C\|\om_e\|_{L^2},\\
&&\|I_6'\|_{L^2}\leq C\|\pa_y\om_e\|_{L^2}\left\|\Big(\f{\rho u'(y_c)}{\rmA_1+i\rmB_1}\Big)\right\|_{L^{\infty}}\leq C\al^{-1}\|\pa_y\om_e\|_{L^2}.
\eeno
By \eqref{5.1}, \eqref{5.8} and \eqref{5.9}, we get
\beno
&&\|I_7'\|_{L^2}\leq C\al^3\left\|\frac{\Lambda_{3,2}(\om_e)}{1+\al\sin y_c}\right\|_{L^2}\leq C\|\om_e\|_{L^2},\\
&&\|I_{8}'\|_{L^2}\leq C\al^2\left\|\frac{\pa_{y_c}\Lambda_{3,2}(\om_e)}{1+\al\sin y_c}\right\|_{L^2}\leq C\|\om_e\|_{L^2}.
\eeno
Summing up, we conclude the even case.
\end{proof}

\section*{Acknowledgement}
Z. Zhang is partially supported by NSF of China under Grant 11425103.

\end{document}